\documentclass[pdflatex,sn-mathphys-num]{sn-jnl}

\usepackage{graphicx}
\usepackage{multirow}
\usepackage{amsmath,amssymb,amsfonts}
\usepackage{amsthm}
\usepackage{mathrsfs}
\usepackage[title]{appendix}
\usepackage{xcolor}
\usepackage{textcomp}
\usepackage{manyfoot}
\usepackage{booktabs}
\usepackage{algorithm}
\usepackage{algorithmicx}
\usepackage{algpseudocode}
\usepackage{listings}
\usepackage{enumitem}

\theoremstyle{thmstyleone}
\newtheorem{theorem}{Theorem}
\newtheorem{proposition}[theorem]{Proposition}

\theoremstyle{thmstyletwo}
\newtheorem{example}{Example}
\newtheorem{remark}{Remark}

\theoremstyle{thmstylethree}
\newtheorem{definition}{Definition}

\usepackage{url}
\usepackage{dsfont}
\usepackage{thmtools}
\usepackage{thm-restate}
\declaretheorem[name=Lemma,numberwithin=section]{lemma}
\declaretheorem[name=Corollary,numberwithin=section]{corollary}
\declaretheorem[name=Assumption]{assumption}

\newcounter{subassumption}

\usepackage{circuitikz}
\usetikzlibrary{decorations.pathreplacing,calligraphy}

\newcommand{\W}{\mathcal{W}}

\begin{document}

\title[Article Title]{Construction and simulation of a path-valued model of dendrite development}

\author*[1,2]{Andrew Nugent}\email{andrew.nugent@ucl.ac.uk}
\author[1,3]{Karen M. Page}
\author[1]{Alexey Zaikin}
\author[2]{Laura C. Andreae}
\affil[1]{Department of Mathematics, University College London}
\affil[2]{Centre for Developmental Neurobiology, King’s College London}
\affil[3]{Institute for the Physics of Living Systems (IPLS), University College London}

\abstract{Neurons receive information through their dendrites. During development, when synaptic connections are forming, dendrites grow, retract, and branch. The resulting dendritic tree shapes the structure of the broader neural network. Crucially, retraction and branching make it necessary to track whole dendritic paths rather than only their endpoints. While this is handled implicitly in some existing simulations, here we construct an explicitly path-valued stochastic process for dendrite growth. Combining this with a branching process, using ideas from measure-valued branching particle systems, we show that the model produces the typical tree structures of real dendrites. To complement this analytical work, we also outline several methods for numerical simulation, including time discretisations at different temporal scales and an approximation using a dynamic graph. This provides both a more rigorous mathematical framework and more structured simulation methods for modelling dendrite development.}


\maketitle

\section{Introduction}

The neurons in mammalian brains have highly branched processes called dendrites, which receive inputs from other neurons via synaptic connections \cite{spruston2008pyramidal}. As they are responsible for receiving and integrating information, the structure of these dendritic trees influences the brain's ability to process and store information \cite{makarov2023dendrites}. As such, an understanding of how this structure develops is crucial to understand the structure and function of the broader neural network. 

Various approaches have been taken to the mathematical modelling of dendrites: some models focus primarily on their topological features, investigating for example the distribution of branches \cite{van2011stochastic,van1997natural,villacorta2007mathematical}; other models focus more on the biophysical mechanisms underpinning dendrite growth \cite{zubler2009framework,mclean2004continuum}; others place fixed cues or branching points in the environment and observe how these become connected \cite{cuntz2012scaling,kirchner2025dendritic}. Existing modelling techniques such as phase field models \cite{takaki2015phase}, annihilating random walks \cite{hannezo2017unifying} and diffusion-limited aggregation \cite{luczak2006spatial} have also been adapted to describe neuronal development. The articles \cite{oliveri2022mathematical,kiddie2005biologically} both outline many of the approaches above. 

Many models focus on tracking the dendrite tips, as this is the point from which dendrites grow, for example the use of branching and annihilating random walks to model dendrite paths in \cite{hannezo2017unifying}. However, it is noted that due to their spatial structure these walks are not independent. We further developed this idea by proposing that it is necessary to track the entire dendrite paths, not just their endpoints, in order to fully capture their development. A similar issue arises in \cite{shree2022dynamic}, in which each dendrite switches between growing, paused and retracting states. When branches retract, points are removed until the new desired length is reached. This model also allows new branches to appear along existing dendrites (referred to as side-branching) rather than only at their tips. Again, these features require information on the full path of the dendrite. 

In this paper, we aim to directly address two related modelling challenges. Firstly, due to side-branching and switching between growth and retraction, it is not sufficient to track only the position of the dendrite tips and in fact their entire path must be recorded. Secondly, combining branching and retraction creates the possibility that some branches may become disconnected if their `parent' branch retracts too far. This is biologically infeasible and so must be avoided. 

A key goal of this paper is to rigorously construct a path-valued process for dendrite development, including growth, retraction and branching that addresses both of these modelling challenges. In addition, we aim to construct a model that is continuous in both time and space, then from this continuous model derive methods for numerical simulations. The proposed model is a combination of three processes: a length process, that determines whether a particular dendrite is currently growing or retracting; a growth process, that describes the way in which growing dendrites extend in space; and a branching process, that determines when and where new dendrite branches are formed. We introduce a framework that appropriately combines these processes into a single stochastic process for the evolution of the entire neuron. 

The approach developed here may also be applied to branching morphogenesis, which appears in many biological settings, such as angiogenesis, lung development or root growth \cite{pillay2017modeling,hannezo2019multiscale,harrison2010shaping,metzger2008branching,iber2013control,menshykau2014interplay}, although our focus in this paper is primarily on neurons.  

The model construction is undertaken in Section \ref{Section: Construction}, first introducing the length and growth processes, then incorporating branching. Throughout we use the generator to characterise these stochastic processes, with the generator for a single branch serving as a building block from which the full model is constructed. In Section \ref{Section: Properties} we examine key properties of this process, for example proving that the tree structure of dendrites is preserved and studying the evolution of the total dendrite length. In Section \ref{Section: numerical simulation} we provide two different methods to numerically simulate the model. Finally, Section \ref{Section: Extensions} discusses various extensions to the model before we conclude in Section \ref{Section: Conclusion}.

\section{Model construction} \label{Section: Construction}

Construction of the model proceeds in the following steps:
\begin{enumerate}
    \item We describe the evolution of a single dendrite branch using a path-valued process, including growth from the endpoint and retraction. This single dendrite process is based on the Brownian snake introduced by Le Gall \cite{gall1999spatial}. In addition, we discuss the generator of this stochastic process. 
    \item We introduce a measure that tracks a fixed number of independent branches, extending the generator for the single branch to a generator for this measure. The branching-retraction problem mentioned previously is also outlined in greater detail. 
    \item Finally, we describe the creation and removal of branches as jumps in the space of measures, as in the classical theory of branching processes. We then add the generator for this jump process to the generator for the growth/length process to complete the construction. 
\end{enumerate}
The result is a stochastic process over the space of measures that tracks a finite (but changing) number of dendrite paths. It is then shown in Section \ref{Section: Properties} that this maps back to the expected tree-like graph structure. 

\subsection{Length and growth processes} \label{Section: Length and growth processes}

We begin by considering a single dendrite. The combination of length and growth processes described below adapts the construction of the Brownian snake by Le Gall \cite{gall1999spatial} to this setting. In the original construction the length process is called the `lifetime' process; it has been renamed as here it directly describes the length of a dendrite branch. This path-valued process was originally studied due to its connections with superprocesses and the genealogy of random continuous trees. However, the growth and retraction behaviour of the Brownian snake, as well as its spatial embedding, provides an excellent model for a single dendrite branch.

The length process, denoted $\ell_t$, is a continuous time stochastic process taking values in $\mathbb{R}^+$ that describes the length of the dendrite. The classical construction in \cite{gall1999spatial} uses a reflecting Brownian motion for this length process. This is extended in \cite{dhersin2000stochastic,abraham2002representations} to include a drift term that depends on the position of the snake/dendrite. As several existing models have been developed that describe the evolution of dendrite lengths \cite{marino2025free,mclean2004continuum} and explain observed oscillations, various alternative length processes are discussed in Section \ref{Section: Extensions}. For simplicity, we take $\ell_t$ as a reflecting Brownian motion as our standard setup. 

For a length process started at $\ell$, and time $s>0$, define the joint distribution $\gamma_s^\ell(dp, dq)$ of the pair $(p,q)=(\inf_{0\leq r \leq s} \ell_r, \ell_s)$. This tracks both the current length (denoted $q$) of the dendrite and shortest length so far observed (denoted $p$). As the process does not necessarily begin with zero length, $p$ is not necessarily zero. 

We next consider the spatial embedding of the dendrite. Assume that it lies in some compact subset of $\mathbb{R}^n$. Typically \textit{in vivo} we will have $n=3$, although some neurons do grow in an effectively flat plane, while \textit{in vitro} we typically have $n=2$. Beyond these spatial dimensions the growth process may have other properties. For example, the dendrite could be modelled as growing forwards in its current direction, in which case it is necessary to record a growth angle $\theta \in [0,2\pi)\times [0,\pi]^{n-1}$ (where $n$ is the spatial dimension). In addition one may wish to record the diameter of dendritic branches and so some additional width $d \in [0,d_{\max}]$ could also be tracked. Let $\xi$ be this growth process and let $X$ denote its state space. Denote by $\Pi_x$ the law of $\xi$ started at $x \in X$ (that is, the probability distribution over paths of $\xi$). We make several assumptions on this process. 

\begin{assumption}
    We make the following assumptions on the growth process: 
    \begin{enumerate}
        \item The state space $X$ is compact. Where $X$ is a subset of $\mathbb{R}^n$ this is equivalent to $X$ being closed and bounded. \label{Assumption: X compact}
        \item The spatial motion $\xi$ is a Feller process on $X$. Informally, this means that $\xi$ is a Markov process in which the probability distribution of future states changes continuously with the current state, see \cite[Section 4.2]{ethier2009markov} for details. 
        \item The spatial motion is continuous. More specifically, as in \cite{gall1999spatial} we assume that for every $\epsilon > 0$
        \begin{align*}
            \lim_{t\rightarrow0} \bigg( \sup_{x \in X} \Pi_x \Big( \sup_{r \leq t} d_X(x,\xi_r) > \epsilon \Big) \bigg) = 0 \,.
        \end{align*}
        That is, we assume that paths are uniformly right-continuous. \label{Assumption: continuous spatial motion}
    \end{enumerate}
\end{assumption}

While the tip of the dendrite is a single point that evolves according to $\xi$, we track the spatial position of the entire dendrite. As the dendrite has finite length $\ell_t$ and does not have gaps, its spatial position at time $t$ will be represented by a function $\omega_t \in C([0,\ell_t] ; X)$. Denote 
\begin{align*}
    \mathcal{W} = \bigcup_{\ell \geq0} C([0,\ell] ; X) \,,
\end{align*}
the space of all such paths. For $\omega \in \W$ denote by $\ell_\omega$ its length and $\hat{\omega} = \omega(\ell_\omega)$ its endpoint. As these paths are continuous we may define the distance 
\begin{align*}
    d_\W(\omega,\nu) = |\ell_\omega - \ell_\nu| + \sup_{t \geq 0} d_X\big(\omega(t \wedge \ell_\omega) , \nu(t \wedge \ell_\nu)\big) \,.
\end{align*}
The first term describes the difference in length between the two paths while the second term describes their `physical' distance in the state space $X$. Under this distance, $(\mathcal{W},d_\W)$ is a Polish space \cite{gall1999spatial}, meaning it is a suitable space on which to define probability measures. Here $t \wedge \ell_\omega$ denotes the minimum of $t$ and $\ell_\omega$. 

We now describe the transition kernels $Q_s$ of the Brownian snake, beginning with a more intuitive description. Starting from $\omega \in \W$ with length $\ell_\omega$ at $t=0$, run the length process up to time $t=s$. Again denote by $q=\ell_t$ and $p = \inf_{0 \leq s \leq t}\ell_s$. Then the path $\nu$ is obtained by first restricting $(\omega(u) : u \in [0,\ell_\omega])$ to $\omega^{\leq p} = (\omega(u) : u \in [0,p])$, meaning we keep the path up until length $p$. Then generate a path of the growth process $\xi$ of length $p-q$, starting from $\omega(p)$, and append this to $\omega^{\leq p}$. This gives the new path $\nu$ with length $q$. 

For example, consider $\omega$ with $\ell_\omega=10$ at $t=0$. We run the length process until $t=1$ and observe a minimum length $p=9$ and final length $q=12$. We keep the path $\omega^{\leq 9}$ up to length $9$. Then, using the growth process $\xi$ started from $\omega(9)\in X$, create an extension to the path of length $q-p=3$ and append this to $\omega^{\leq 9}$ to give a path $\nu$ of length $\ell_\nu=12$. 

Recall that $\gamma$ describes the joint distribution of $(p,q)$. Assume that values of $(p,q)=(\inf_{0\leq r \leq s} \ell_r, \ell_s)$ are known. We can then define a probability measure $R_{p,q}(\omega,d\nu)$ that describes the probability of moving from a path $\omega\in \W$ to a path $\nu \in \W$, given known values of $p$ and $q$. Intuitively, this probability measure should say that the path of $\omega$ is preserved up to length $p$ and then extended to length $q$. The notation $R.a.s.$ is used to indicate that events happen almost surely with respect to this measure. The measure $R_{p,q}(\omega,d\nu)$ is defined by the following properties:  
\begin{enumerate}[label=\alph*)]
    \item $\ell_{\nu} = q$, $R.a.s.$, 
    \item $\nu(t) = \omega(t)$ for every $t \in [0,p]$ $R.a.s.$, 
    \item The law under $R_{p,q}(\omega,d\nu)$ of $(\nu(p+t), 0 \leq t \leq q-p)$ is the law of $(\xi_t, 0 \leq t \leq q-p)$ under $\Pi_{\omega(p)}$. 
\end{enumerate}
This means that the path of $\omega$ up to $\omega(p)$ is copied into $\nu$ (property b). The path of $\nu$ is then extended using the growth process $\xi$ started at $\omega(p)$ (property c) until it reaches length $q$ (property a). This is illustrated in Figure \ref{fig: Brownian snake}.

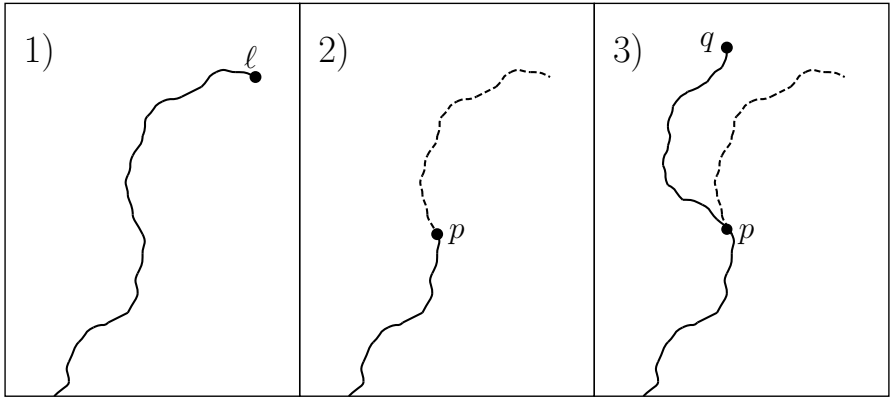
\begin{figure}[!ht]
\centering
\resizebox{.9\textwidth}{!}{
\begin{circuitikz}
\tikzstyle{every node}=[font=\fontsize{27.9pt}{36.15pt}\selectfont]
\draw [line width=1.5pt, short] (-5,4.375) .. controls (-4.75,4.625) and (-4.625,4.625) .. (-4.625,4.75);
\draw [line width=1.5pt, short] (-4.625,4.75) .. controls (-4.625,5) and (-4.75,5) .. (-4.625,5.25);
\draw [line width=1.5pt, short] (-4.625,5.25) .. controls (-4.5,5.5) and (-4.5,5.625) .. (-4.375,5.75);
\draw [line width=1.5pt, short] (-4.375,5.75) .. controls (-4.125,6) and (-4.25,6) .. (-4,6.125);
\draw [line width=1.5pt, short] (-4,6.125) .. controls (-3.75,6.25) and (-3.75,6.125) .. (-3.625,6.25);
\draw [line width=1.5pt, short] (-3,9) .. controls (-3,9.25) and (-3.125,9.375) .. (-3.125,9.5);
\draw [line width=1.5pt, short] (-3.125,9.5) .. controls (-3.125,9.75) and (-3.25,9.75) .. (-3.125,10);
\draw [line width=1.5pt, short] (-3.125,10) .. controls (-3,10.25) and (-3.125,10.25) .. (-3,10.5);
\draw [line width=1.5pt, short] (-3,10.5) .. controls (-2.875,10.75) and (-2.75,10.75) .. (-2.75,11);
\draw [line width=1.5pt, short] (-2.75,11) .. controls (-2.625,11.25) and (-2.75,11.375) .. (-2.625,11.5);
\draw [line width=1.5pt, short] (-2.625,11.5) .. controls (-2.375,11.75) and (-2.5,11.75) .. (-2.25,11.875);
\draw [line width=1.5pt, short] (-2.25,11.875) .. controls (-2,12) and (-2,11.875) .. (-1.75,12);
\draw [line width=1.5pt, short] (-1.75,12) .. controls (-1.5,12.125) and (-1.5,12.125) .. (-1.25,12.25);
\draw [line width=1.5pt, short] (-1.25,12.25) .. controls (-1,12.5) and (-1.125,12.5) .. (-0.875,12.625);
\draw [line width=1.5pt, short] (-0.875,12.625) .. controls (-0.625,12.75) and (-0.625,12.625) .. (-0.375,12.625);
\draw [line width=1.5pt, short] (-0.375,12.625) .. controls (-0.125,12.625) and (-0.125,12.625) .. (0.125,12.5);
\draw [line width=1.5pt, short] (-3,9) .. controls (-3,8.875) and (-2.875,8.75) .. (-2.75,8.5);
\draw [line width=1.5pt, short] (-2.75,8.5) .. controls (-2.625,8.25) and (-2.75,8.25) .. (-2.75,8);
\draw [line width=1.5pt, short] (-2.75,8) .. controls (-2.75,7.75) and (-2.75,7.75) .. (-2.875,7.5);
\draw [line width=1.5pt, short] (-2.875,7.5) .. controls (-3,7.25) and (-2.875,7.25) .. (-2.875,7);
\draw [line width=1.5pt, short] (-2.875,7) .. controls (-2.875,6.875) and (-3,6.75) .. (-3.125,6.5);
\draw [line width=1.5pt, short] (-3.125,6.5) .. controls (-3.375,6.375) and (-3.375,6.375) .. (-3.625,6.25);
\draw [line width=1.5pt, short] (11.875,8.875) .. controls (12,8.75) and (12.125,8.75) .. (12.25,8.5);
\draw [line width=1.5pt, short] (11.875,8.875) .. controls (11.625,9.125) and (11.75,9.125) .. (11.5,9.25);
\draw [line width=1.5pt, short] (11.5,9.25) .. controls (11.25,9.375) and (11.25,9.375) .. (11,9.375);
\draw [line width=1.5pt, short] (11,9.375) .. controls (10.75,9.625) and (10.875,9.625) .. (10.625,9.75);
\draw [line width=1.5pt, short] (10.625,9.75) .. controls (10.5,10) and (10.625,10) .. (10.5,10.25);
\draw [line width=1.5pt, short] (10.5,10.25) .. controls (10.5,10.625) and (10.625,10.5) .. (10.625,10.75);
\draw [line width=1.5pt, short] (10.625,10.75) .. controls (10.75,11.125) and (10.625,11) .. (10.625,11.25);
\draw [line width=1.5pt, short] (10.625,11.25) .. controls (10.625,11.5) and (10.75,11.5) .. (10.75,11.75);
\draw [line width=1.5pt, short] (10.75,11.75) .. controls (11,12.125) and (11,12) .. (11.125,12.125);
\draw [line width=1.5pt, short] (11.125,12.125) .. controls (11.5,12.25) and (11.375,12.25) .. (11.625,12.375);
\draw [line width=1.5pt, short] (11.625,12.375) .. controls (11.875,12.75) and (11.875,12.625) .. (12,12.75);
\draw [line width=1.5pt, short] (12,12.75) .. controls (12.125,13) and (12.125,13) .. (12.125,13.25);

\draw [line width=1pt ] (8.75,14.375) rectangle (16.25,4.375);
\draw [line width=1pt ] (-6.25,14.375) rectangle (1.25,4.375);
\draw [line width=1pt ] (1.25,14.375) rectangle (8.75,4.375);

\node [font=\fontsize{27.9pt}{36.15pt}\selectfont, fill={rgb,255:red,255; green,255; blue,255}, fill opacity=1, text opacity=1, inner xsep=0.080cm, inner ysep=0.085cm, rounded corners=0.020cm] at (-5.375,13.125) {1)};
\node [font=\fontsize{27.9pt}{36.15pt}\selectfont, fill={rgb,255:red,255; green,255; blue,255}, fill opacity=1, text opacity=1, inner xsep=0.080cm, inner ysep=0.085cm, rounded corners=0.020cm] at (2.125,13.125) {2)};
\node [font=\fontsize{27.9pt}{36.15pt}\selectfont, fill={rgb,255:red,255; green,255; blue,255}, fill opacity=1, text opacity=1, inner xsep=0.080cm, inner ysep=0.085cm, rounded corners=0.020cm] at (9.625,13.125) {3)};
\draw [line width=1.5pt, short] (2.5,4.375) .. controls (2.75,4.625) and (2.875,4.625) .. (2.875,4.75);
\draw [line width=1.5pt, short] (2.875,4.75) .. controls (2.875,5) and (2.75,5) .. (2.875,5.25);
\draw [line width=1.5pt, short] (2.875,5.25) .. controls (3,5.5) and (3,5.625) .. (3.125,5.75);
\draw [line width=1.5pt, short] (3.125,5.75) .. controls (3.375,6) and (3.25,6) .. (3.5,6.125);
\draw [line width=1.5pt, short] (3.5,6.125) .. controls (3.75,6.25) and (3.75,6.125) .. (3.875,6.25);
\draw [line width=1.5pt, short] (4.75,8.5) .. controls (4.875,8.25) and (4.75,8.25) .. (4.75,8);
\draw [line width=1.5pt, short] (4.75,8) .. controls (4.75,7.75) and (4.75,7.75) .. (4.625,7.5);
\draw [line width=1.5pt, short] (4.625,7) .. controls (4.625,6.875) and (4.5,6.75) .. (4.375,6.5);
\draw [line width=1.5pt, short] (4.375,6.5) .. controls (4.125,6.375) and (4.125,6.375) .. (3.875,6.25);
\draw [line width=1.5pt, short] (4.625,7.5) .. controls (4.5,7.25) and (4.625,7.25) .. (4.625,7);
\draw [ fill={rgb,255:red,0; green,0; blue,0}, fill opacity=1,line width=1.5pt ] (0.125,12.5) circle (0.125cm);
\draw [ fill={rgb,255:red,0; green,0; blue,0}, fill opacity=1,line width=1.5pt ] (4.75,8.5) circle (0.125cm);
\draw [line width=1.5pt, dash pattern=on 3\pgflinewidth off 1.5\pgflinewidth] (4.5,9) .. controls (4.5,9.25) and (4.375,9.375) .. (4.375,9.5);
\draw [line width=1.5pt, dash pattern=on 3\pgflinewidth off 1.5\pgflinewidth] (4.375,9.5) .. controls (4.375,9.75) and (4.25,9.75) .. (4.375,10);
\draw [line width=1.5pt, dash pattern=on 3\pgflinewidth off 1.5\pgflinewidth] (4.375,10) .. controls (4.5,10.25) and (4.375,10.25) .. (4.5,10.5);
\draw [line width=1.5pt, dash pattern=on 3\pgflinewidth off 1.5\pgflinewidth] (4.5,10.5) .. controls (4.625,10.75) and (4.75,10.75) .. (4.75,11);
\draw [line width=1.5pt, dash pattern=on 3\pgflinewidth off 1.5\pgflinewidth] (4.75,11) .. controls (4.875,11.25) and (4.75,11.375) .. (4.875,11.5);
\draw [line width=1.5pt, dash pattern=on 3\pgflinewidth off 1.5\pgflinewidth] (4.875,11.5) .. controls (5.125,11.75) and (5,11.75) .. (5.25,11.875);
\draw [line width=1.5pt, dash pattern=on 3\pgflinewidth off 1.5\pgflinewidth] (5.25,11.875) .. controls (5.5,12) and (5.5,11.875) .. (5.75,12);
\draw [line width=1.5pt, dash pattern=on 3\pgflinewidth off 1.5\pgflinewidth] (5.75,12) .. controls (6,12.125) and (6,12.125) .. (6.25,12.25);
\draw [line width=1.5pt, dash pattern=on 3\pgflinewidth off 1.5\pgflinewidth] (6.25,12.25) .. controls (6.5,12.5) and (6.375,12.5) .. (6.625,12.625);
\draw [line width=1.5pt, dash pattern=on 3\pgflinewidth off 1.5\pgflinewidth] (6.625,12.625) .. controls (6.875,12.75) and (6.875,12.625) .. (7.125,12.625);
\draw [line width=1.5pt, dash pattern=on 3\pgflinewidth off 1.5\pgflinewidth] (7.125,12.625) .. controls (7.375,12.625) and (7.375,12.625) .. (7.625,12.5);
\draw [line width=1.5pt, dash pattern=on 3\pgflinewidth off 1.5\pgflinewidth] (4.5,9) .. controls (4.5,8.875) and (4.625,8.75) .. (4.75,8.5);
\draw [line width=1.5pt, short] (10,4.375) .. controls (10.25,4.625) and (10.375,4.625) .. (10.375,4.75);
\draw [line width=1.5pt, short] (10.375,4.75) .. controls (10.375,5) and (10.25,5) .. (10.375,5.25);
\draw [line width=1.5pt, short] (10.375,5.25) .. controls (10.5,5.5) and (10.5,5.625) .. (10.625,5.75);
\draw [line width=1.5pt, short] (10.625,5.75) .. controls (10.875,6) and (10.75,6) .. (11,6.125);
\draw [line width=1.5pt, short] (11,6.125) .. controls (11.25,6.25) and (11.25,6.125) .. (11.375,6.25);
\draw [line width=1.5pt, short] (12.25,8.5) .. controls (12.375,8.25) and (12.25,8.25) .. (12.25,8);
\draw [line width=1.5pt, short] (12.25,8) .. controls (12.25,7.75) and (12.25,7.75) .. (12.125,7.5);
\draw [line width=1.5pt, short] (12.125,7) .. controls (12.125,6.875) and (12,6.75) .. (11.875,6.5);
\draw [line width=1.5pt, short] (11.875,6.5) .. controls (11.625,6.375) and (11.625,6.375) .. (11.375,6.25);
\draw [line width=1.5pt, short] (12.125,7.5) .. controls (12,7.25) and (12.125,7.25) .. (12.125,7);
\draw [ fill={rgb,255:red,0; green,0; blue,0}, fill opacity=1,line width=1.5pt ] (12.125,13.25) circle (0.125cm);
\draw [line width=1.5pt, dash pattern=on 3\pgflinewidth off 1.5\pgflinewidth] (12,9) .. controls (12,9.25) and (11.875,9.375) .. (11.875,9.5);
\draw [line width=1.5pt, dash pattern=on 3\pgflinewidth off 1.5\pgflinewidth] (11.875,9.5) .. controls (11.875,9.75) and (11.75,9.75) .. (11.875,10);
\draw [line width=1.5pt, dash pattern=on 3\pgflinewidth off 1.5\pgflinewidth] (11.875,10) .. controls (12,10.25) and (11.875,10.25) .. (12,10.5);
\draw [line width=1.5pt, dash pattern=on 3\pgflinewidth off 1.5\pgflinewidth] (12,10.5) .. controls (12.125,10.75) and (12.25,10.75) .. (12.25,11);
\draw [line width=1.5pt, dash pattern=on 3\pgflinewidth off 1.5\pgflinewidth] (12.25,11) .. controls (12.375,11.25) and (12.25,11.375) .. (12.375,11.5);
\draw [line width=1.5pt, dash pattern=on 3\pgflinewidth off 1.5\pgflinewidth] (12.375,11.5) .. controls (12.625,11.75) and (12.5,11.75) .. (12.75,11.875);
\draw [line width=1.5pt, dash pattern=on 3\pgflinewidth off 1.5\pgflinewidth] (12.75,11.875) .. controls (13,12) and (13,11.875) .. (13.25,12);
\draw [line width=1.5pt, dash pattern=on 3\pgflinewidth off 1.5\pgflinewidth] (13.25,12) .. controls (13.5,12.125) and (13.5,12.125) .. (13.75,12.25);
\draw [line width=1.5pt, dash pattern=on 3\pgflinewidth off 1.5\pgflinewidth] (13.75,12.25) .. controls (14,12.5) and (13.875,12.5) .. (14.125,12.625);
\draw [line width=1.5pt, dash pattern=on 3\pgflinewidth off 1.5\pgflinewidth] (14.125,12.625) .. controls (14.375,12.75) and (14.375,12.625) .. (14.625,12.625);
\draw [line width=1.5pt, dash pattern=on 3\pgflinewidth off 1.5\pgflinewidth] (14.625,12.625) .. controls (14.875,12.625) and (14.875,12.625) .. (15.125,12.5);
\draw [line width=1.5pt, dash pattern=on 3\pgflinewidth off 1.5\pgflinewidth] (12,9) .. controls (12,8.875) and (12.125,8.75) .. (12.25,8.5);
\node [font=\fontsize{22.8pt}{29.6pt}\selectfont, inner xsep=0.080cm, inner ysep=0.085cm, rounded corners=0.020cm] at (0,13) {$\ell$};
\node [font=\fontsize{22.8pt}{29.6pt}\selectfont, inner xsep=0.080cm, inner ysep=0.085cm, rounded corners=0.020cm] at (5.25,8.5) {$p$};
\node [font=\fontsize{22.8pt}{29.6pt}\selectfont, inner xsep=0.080cm, inner ysep=0.085cm, rounded corners=0.020cm] at (12.625,8.5) {$p$};
\node [font=\fontsize{22.8pt}{29.6pt}\selectfont, inner xsep=0.080cm, inner ysep=0.085cm, rounded corners=0.020cm] at (11.625,13.375) {$q$};
\draw [ fill={rgb,255:red,0; green,0; blue,0}, fill opacity=1, line width=1pt ] (12.125,8.625) circle (0.125cm);
\end{circuitikz}
}
\caption{Illustration of intuitive evolution of the Brownian snake. 1) The initial snake. 2) The path from $p$ to $\ell$ is erased. 3) The path is extended again to length $q$ using the growth process. In panels 2 and 3 the erased path is shown in dashed lines for reference. A circle indicates the current endpoint of the branch.}
\label{fig: Brownian snake}
\end{figure}

We have now specified two distributions: the distribution $\gamma_s^l(dp, dq)$ for the length process, and the distribution $R_{p,q}(\omega,d\omega')$ arising from the growth process. These are combined in the following definition. 

\begin{definition}
    The \textit{Brownian-snake} is the Markov process in $\mathcal{W}$, denoted by $(\omega_s,s\geq0)$, whose transition kernels $\mathcal{Q}_s$ are given by 
    \begin{align} \label{Eqn: Snake transition kernel}
        \mathcal{Q}_s(\omega,d\nu) = \int_\mathbb{R} \int_\mathbb{R} \gamma_s^{\ell_\omega}(dp,dq) \, R_{p,q}(\omega,d\nu) \,.
    \end{align}
\end{definition}

For our model of dendrites, the length process determines if the dendrite is growing or retracting and the growth process describes how it extends in space when it is growing. 

\subsubsection{Generator} \label{Section: Generator}

In this section we introduce the generator of the Brownian snake. The advantage of this is that we may use the generator for the single Brownian snake as a building block to construct a more complex process tracking multiple snakes, then introduce branching by adding a new jump component to the generator. It is significantly easier to combine generators in this way, rather than attempting to define more complex processes through their transition kernels. This construction follows a similar pattern to other branching particle systems, except that here our `particles' are dendrite branches. The generator of the Brownian snake is studied in greater depth in \cite{dhersin2000stochastic,abraham2002representations}, which also provide some key results of this section. The characterisation of stochastic processes via their generators is covered extensively in \cite{ethier2009markov}, which develops ideas from \cite{stroock2007multidimensional}. 

Intuitively, the generator $\mathcal{A}$ characterises the stochastic process $\omega_t$ through the evolution of observables $f(\omega_t)$ rather than through its probability distribution. Denote by $B(\W)$ the set of bounded functions on the space $\W$. The generator can be defined as the limit (where it exists) of
\begin{align} \label{Eqn: Definition of generator}
    \mathcal{A}f(\omega) =  \lim_{t\rightarrow0} \frac{\mathbb{E} \big[f(\omega_t)\,|\, \omega_0 = \omega\big] - f(\omega)}{t} \,.
\end{align}
for $f \in B(\W)$. Those bounded functions $f$ for which this limit exist are called the domain of $\mathcal{A}$, denoted $\mathcal{D}(\mathcal{A}) \subset B(\W)$. The generator describes the infinitesimal change in these functions/observables over time. If the domain $\mathcal{D}(\mathcal{A})$ is suitably large then the generator $\mathcal{A}$ is sufficient to uniquely characterise the process $\omega_t$. The generator itself maps functions $f \in \mathcal{D}(\mathcal{A})$ to bounded functions $\mathcal{A}f$. 

This can be made more rigorous using the theory of martingales. A martingale is a stochastic process in which the expected value at some future time, conditioned on all events so far (referred to formally as the natural filtration), is the current value of the process. The following definition is taken from \cite{ethier2009markov}. 
\begin{definition}
    A solution to the \textit{martingale problem} for a generator $\mathcal{A}: \mathcal{D}(\mathcal{A}) \rightarrow B(\W)$ is a measurable stochastic process $\omega_t$ with state space $\W$, defined on some probability space $(\Omega, \mathcal{F},P)$, such that for all $f \in \mathcal{D}(\mathcal{A})$, 
    \begin{align*}
        f(\omega_t) - f(\omega_0)- \int_0^t \mathcal{A}f(\omega_s) \, ds \,,
    \end{align*}
    is a martingale with respect to its natural filtration. We say the martingale problem for $\mathcal{A}$ is \textit{well-posed} if there exists a unique solution to the martingale problem.
\end{definition}
That is, if the martingale problem for the generator $\mathcal{A}$ has a unique solution $\omega_t$ then we can say that the generator $\mathcal{A}$ characterises/defines the process $\omega_t$. As we wish to use the generator of the Brownian snake as a building block to construct more complex processes, we do not strictly need to know its form or its domain, only that the corresponding martingale problem is well-posed and that its unique solution is given by the Brownian snake. 

From the transition kernels $Q_t$ \eqref{Eqn: Snake transition kernel} we can determine expectations
\begin{align*}
    \mathbb{E} \big[f(\omega_t)\,|\, \omega_0 = \omega\big] = \int_\mathcal{W} f(\nu) \, Q_t(\omega,d\nu)
\end{align*}
and define the generator $\mathcal{A}$ as in \eqref{Eqn: Definition of generator}. As we define the generator from the transition probabilities, it follows that the Brownian snake $\omega_t$ is a solution to the martingale problem \cite[Proposition 4.1.7]{ethier2009markov}. However, it remains to show that this solution is unique.  

\begin{proposition} \label{Prop: Martingale problem for Brownian snake well-posed}
    The martingale problem for $\mathcal{A}$ defined by \eqref{Eqn: Definition of generator} is well-posed. 
\end{proposition}
\begin{proof}
    By construction, the Brownian snake is a solution of the martingale problem \cite[Proposition 4.1.7]{ethier2009markov}. For the reverse direction, it is shown in \cite{dhersin2000stochastic} that the set of test functions 
    \begin{align*}
        D = \bigg\{ F(\omega) := \int^{\ell_\omega}_0 g\big(\omega^{\leq r}\big) \, dr \,\bigg|\,g \in C_b(\mathcal{W}) \bigg\} \,,
    \end{align*}
    lie in the domain $\mathcal{D}(\mathcal{A})$. Moreover, any process $(W_s,s\geq0)$ satisfying the martingale problem on $D$ has the law of a Brownian snake \cite[Theorem 3]{dhersin2000stochastic}. From this we apply Theorem 4.4.2 of \cite{ethier2009markov} to conclude that the martingale problem for $\mathcal{A}$ is well-posed and that its solution is the Brownian snake. 
\end{proof}

This means that the generator $\mathcal{A}$ of the Brownian snake is well-defined and characterises the process $\omega_t$. We can therefore think of the Brownian snake as being defined either by the transition kernel $Q_t$ \eqref{Eqn: Snake transition kernel} or the generator $\mathcal{A}$ \eqref{Eqn: Definition of generator}. 

\begin{remark}
    Since the spatial motion is continuous and the length process is a reflected Brownian motion, so is almost surely continuous, the Brownian snake will have almost surely continuous sample paths. 
\end{remark}

\begin{remark}
    A core of $\mathcal{D}(\mathcal{A})$ is a dense subset of $\mathcal{D}(\mathcal{A})$; it is often useful to define a generator by its action on a core, then extend this to the full domain. However, it is \textbf{not} claimed here that $D$ forms a core of $\mathcal{D}(\mathcal{A})$. Instead that $D$ is a subset of $\mathcal{D}(\mathcal{A})$ and is sufficient to provide uniqueness of one-dimensional distributions. As a result, the behaviour of the generator on $D$, discussed further in Sections 2 and 3 of \cite{dhersin2000stochastic} and below in Appendix \ref{Appendix: Absorbed Brownian motion}, is indicative of the behaviour of the generator $\mathcal{A}$ but does not fully describe it. As noted above, it is not necessary in this case to describe the behaviour of the generator or its domain, only to show that the martingale problem is well-posed.
\end{remark} 

\begin{definition}
    The following definition is taken from \cite[Section 3.4]{ethier2009markov}. A set of functions $D \subset B(X)$ over a space $X$ is \textit{separating} if for any measures $\mu$ and $\nu$ on $X$, if 
    \begin{align*}
        \int f \, d\mu = \int f \, d\nu 
    \end{align*}
    for all $f \in D$ then $\mu = \nu$. That is, $D$ is sufficiently large that measures agreeing on $D$ must agree on the whole space. 
\end{definition}

This separability is important for the domain $\mathcal{D}(\mathcal{A})$ of a generator as it gives an indication that the domain is sufficiently large and therefore that the martingale problem may have a unique solution. It is shown in \ref{Prop: D(A) separating} in Appendix \ref{Appendix: A} that the domain  $\mathcal{D}(\mathcal{A})$ is separating. This property will be required later in Section \ref{Section: New branches} when adding additional components to the generator. 

The standard setup described above uses a reflecting Brownian motion for the length process. However in our case this boundary at $\ell=0$ represents dendrite branches that have been fully retracted, so it is more natural to use an absorbing boundary at $\ell=0$ so that these fully retracted branches can be removed. 

\begin{proposition} \label{Prop: well posed with absorbing BC}
    The martingale problem for $\mathcal{A}$, defined by \eqref{Eqn: Definition of generator}, using a Brownian motion absorbed at $\ell=0$ for the length process, is also well-posed. 
\end{proposition}
\begin{proof}
    A Brownian motion absorbed at zero can be constructed from a reflecting Brownian motion $\ell_t$ by defining the stopping time $\tau_0 = \inf_{t\geq0}\{\ell_t=0\}$, then defining the absorbing Brownian motion by $\Tilde{\ell}_t := \ell_{t \wedge \tau_0}$. Similarly for the Brownian snake we may define the set $U = \{ \omega \in \W : \ell_\omega = 0\}$ and $\tau_U$ the hitting time of this set. Beginning with a standard Brownian snake $\omega_t$, the stopped process $\Tilde{\omega}_t = \omega_{t \wedge \tau_U}$ corresponds to a Brownian snake whose length is an absorbed Brownian motion. Theorem 4.6.1 of \cite{ethier2009markov} then provides the well-posedness of this stopped martingale problem. 
\end{proof}

The approach in Proposition \ref{Prop: well posed with absorbing BC} uses a stopping time to enforce the boundary condition. It is also possible to begin with this boundary condition using $\gamma_t^\ell(dp,dq)$ given in Lemma \ref{Lemma: gamma for absorbed BM} below in the transition kernel \eqref{Eqn: Snake transition kernel}. 

\begin{lemma} \label{Lemma: gamma for absorbed BM}
Recall that $\gamma_s^\ell(dp, dq)$ is the joint distribution of the pair $(p,q)=(\inf_{0\leq r \leq s} \ell_r, \ell_s)$ at time $s$ for a length process started at $\ell$. When the length process $\ell_t$ is a Brownian motion absorbed at $\ell=0$ we have, 
\begin{align*}
    \gamma_t^\ell(dp,dq) &= \frac{2(\ell+q-2p)}{\sqrt{2\pi t^3}} \, \exp\bigg( -\frac{(\ell+q-2p)^2}{2t}\bigg) \, \mathds{1}_{\{ 0 < p < \ell\wedge q \}}\,dp \, dq \, \\
    &+ 2\Bigg(1 - \Phi\bigg(\frac{\ell}{\sqrt{t}}\bigg)\Bigg) \delta_{(0,0)}(dp,dq) \,.
\end{align*}
\end{lemma}
\begin{proof}
    See Appendix \ref{Appendix: Absorbed Brownian motion}.
\end{proof}

This defines an alternative version of the Brownian snake with an absorbing boundary condition at $\ell=0$. The generator for this version is the standard generator multiplied by $\mathds{1}\{\ell_\omega>0\}$ (as the evolution is the same for $\ell_\omega>0$ and there is no further evolution if $\ell_\omega=0$, see Appendix \ref{Appendix: Absorbed Brownian motion} for more details). 

As with the boundary condition at $\ell=0$ we may similarly impose an upper reflecting boundary condition at some $\ell_{\max} \gg 1$. While this is not strictly necessary, it will prove technically convenient in Section \ref{Section: Tracking multiple dendrites} and $\ell_{\max}$ can be chosen to be arbitrarily large. This is also a reasonable condition as protein degradation along branches suggests a natural maximum length \cite{miller1997axon}. 

This completes the setup for the evolution of a single dendrite branch. We next discuss how to track multiple branches and the creation of new branches. 

\subsection{Branching process} \label{Section: branching process}

We begin this section by addressing a key issue arising when combining branching and retraction, referred to as the branching-retraction problem. This is illustrated in Figure \ref{fig: branching-retraction problem}, for simplicity this diagram uses straight dendrite branches but the problem remains for any choice of growth process. 

\begin{figure}[!ht]
\centering
\resizebox{.9\textwidth}{!}{
\begin{circuitikz}
\tikzstyle{every node}=[font=\fontsize{18.2pt}{23.7pt}\selectfont]
\draw [ line width=1pt ] (11.25,14.6) rectangle (16.25,7.125);
\draw [ line width=1pt ] (1.25,14.6) rectangle (6.25,7.125);
\draw [ line width=1pt ] (16.25,14.6) rectangle (21.25,7.125);
\draw [ line width=1pt ] (6.25,14.6) rectangle (11.25,7.125);
\draw [line width=2pt, short] (2.5,8.125) -- (4.5,13.75);
\draw [ fill={rgb,255:red,0; green,0; blue,0}, fill opacity=1] (2.5,8.125) circle (0.5cm);
\draw [ color={rgb,255:red,33; green,30; blue,202}, draw opacity=1, line width=1.5pt, -{Stealth[scale=1.5]}, ] (3.5,12.5) -- (4,13.75);
\draw [line width=2pt, short] (7.5,8.125) -- (9.6,14.3);
\draw [ fill={rgb,255:red,0; green,0; blue,0}, fill opacity=1] (7.5,8.125) circle (0.5cm);
\draw [line width=2pt, short] (12.5,8.125) -- (14.125,12.875);
\draw [ fill={rgb,255:red,0; green,0; blue,0}, fill opacity=1] (12.5,8.125) circle (0.5cm);
\draw [line width=2pt, short] (17.375,8.125) -- (18.5,11.25);
\draw [ fill={rgb,255:red,0; green,0; blue,0}, fill opacity=1] (17.375,8.125) circle (0.5cm);
\draw [line width=2pt, short] (8.75,11.875) -- (9.625,12.375);
\draw [ color={rgb,255:red,33; green,30; blue,202}, draw opacity=1, line width=1.5pt, -{Stealth[scale=1.5]}, ] (8.625,13.125) -- (9.125,14.375);
\draw [ color={rgb,255:red,148; green,11; blue,11}, draw opacity=1, line width=1.5pt, {Stealth[scale=1.5]}-, ] (13.25,11.75) -- (13.75,13);
\draw [ color={rgb,255:red,33; green,30; blue,202}, draw opacity=1, line width=1.5pt, -{Stealth[scale=1.5]}, ] (9.125,11.5) -- (9.75,11.875);
\draw [line width=2pt, short] (13.75,11.875) -- (15.125,12.75);
\draw [ color={rgb,255:red,33; green,30; blue,202}, draw opacity=1, line width=1.5pt, -{Stealth[scale=1.5]}, ] (14.3,11.75) -- (15.5,12.375);
\draw [line width=2pt, short] (18.75,11.875) -- (20.5,12.875);
\draw [ color={rgb,255:red,33; green,30; blue,202}, draw opacity=1, line width=1.5pt, -{Stealth[scale=1.5]}, ] (19.625,11.875) -- (20.75,12.5);
\draw [ color={rgb,255:red,148; green,11; blue,11}, draw opacity=1, line width=1.5pt, {Stealth[scale=1.5]}-, ] (17.625,10.125) -- (18.125,11.375);
\node [font=\fontsize{22.8pt}{29.6pt}\selectfont, fill={rgb,255:red,255; green,255; blue,255}, fill opacity=1, text opacity=1, inner xsep=0.080cm, inner ysep=0.085cm, rounded corners=0.020cm] at (1.875,13.8) {1)};
\node [font=\fontsize{22.8pt}{29.6pt}\selectfont, fill={rgb,255:red,255; green,255; blue,255}, fill opacity=1, text opacity=1, inner xsep=0.080cm, inner ysep=0.085cm, rounded corners=0.020cm] at (6.875,13.8) {2)};
\node [font=\fontsize{22.8pt}{29.6pt}\selectfont, fill={rgb,255:red,255; green,255; blue,255}, fill opacity=1, text opacity=1, inner xsep=0.080cm, inner ysep=0.085cm, rounded corners=0.020cm] at (11.875,13.8) {3)};
\node [font=\fontsize{22.8pt}{29.6pt}\selectfont, fill={rgb,255:red,255; green,255; blue,255}, fill opacity=1, text opacity=1, inner xsep=0.080cm, inner ysep=0.085cm, rounded corners=0.020cm] at (16.875,13.8) {4)};
\end{circuitikz}
}
\caption{Diagram of the branching-retraction problem. 1) A single dendrite is initially growing. 2) this `parent' branch produces a growing `child' branch. 3) The parent branch begins to retract. 4) The parent branch retracts to such an extent that the child branch becomes disconnected from the rest of the neuron.}
\label{fig: branching-retraction problem}
\end{figure}
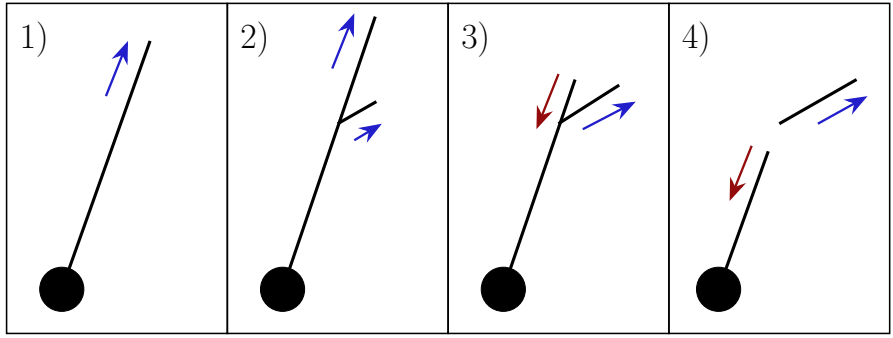

Let $\omega_t$ be an existing growing dendrite (Figure \ref{fig: branching-retraction problem} panel 1) and assume at time $t^*$ we add a branch at the point $\omega_{t^*}(u)$ for some $u \in (0,\ell_{t^*})$ (Figure \ref{fig: branching-retraction problem} panel 2). For now we omit the details of how this branch point is chosen. We now have two processes, the original `parent' branch $\omega_t$ and its `child' branch $\nu_t$, with $\nu_t(0) = \omega_{t^*}(u)$. As the dynamics of $\omega_t$ affect only its endpoint $\hat{\omega}$ the location of the branch point $\omega_{t^*}$ will not move, so this does not create a problem. However, if $\omega_t$ and $\nu_t$ were completely independent, then nothing prevents the parent branch retracting (Figure \ref{fig: branching-retraction problem} panel 3) until it eventually retracts beyond the branch point (Figure \ref{fig: branching-retraction problem} panel 4). If this occurs then the branch $\nu$ would be disconnected from the rest of the neuron. This is a highly unrealistic scenario and therefore a modelling problem that must be addressed. 

There are various ways this problem could be solved. One option is to limit the length process $\ell_\omega$ such that it cannot fall below $u$. This requires that each dendrite also tracks the locations of all branch points along its path, the number of which may change over time. It is not sufficient to track only the latest branch point as, if this branch fully retracts and is removed, it will be necessary to recall information on previous branch points. Moreover, this involves repeatedly updating the boundary condition on the length process, with these updates also depending on the behaviour of other branches, violating the independence between branches and their length processes.  

We propose an alternative method using an `activity' variable $a \in \{0,1\}$ for each branch. If a new branch $\nu$ is created with $\nu(0)=\omega(u)$ then the original branch $\omega$ is split into two parts
\begin{align*}
    \omega^- &= \{\omega^-(s) = \omega(s) \,|\, 0 \leq s \leq u\}\,, \\
    \omega^+ &= \{\omega^+(s) = \omega(u + s) \,|\, 0 \leq s \leq \ell - u\} \,.
\end{align*}
The lower part $\omega^-$ is given activity $a_{\omega^-}=0$, meaning that its length process is frozen at $\ell_{\omega^-} = u$, while the upper part $\omega^+$ retains its current activity level, $a_{\omega^+} = a_\omega$. The new branch $\nu$ will be active, $a_\nu = 1$. 

\begin{figure}[!ht]
\centering
\resizebox{.6\textwidth}{!}{
\begin{circuitikz}
\tikzstyle{every node}=[font=\fontsize{27.9pt}{36.2pt}\selectfont]
\draw [ color={rgb,255:red,113; green,113; blue,113}, draw opacity=1, line width=2pt, short] (-5,4.375) .. controls (-4.75,4.625) and (-4.625,4.625) .. (-4.625,4.75);
\draw [ color={rgb,255:red,113; green,113; blue,113}, draw opacity=1, line width=2pt, short] (-4.625,4.75) .. controls (-4.625,5) and (-4.75,5) .. (-4.625,5.25);
\draw [ color={rgb,255:red,113; green,113; blue,113}, draw opacity=1, line width=2pt, short] (-4.625,5.25) .. controls (-4.5,5.5) and (-4.5,5.625) .. (-4.375,5.75);
\draw [ color={rgb,255:red,113; green,113; blue,113}, draw opacity=1, line width=2pt, short] (-4.375,5.75) .. controls (-4.125,6) and (-4.25,6) .. (-4,6.125);
\draw [ color={rgb,255:red,113; green,113; blue,113}, draw opacity=1, line width=2pt, short] (-4,6.125) .. controls (-3.75,6.25) and (-3.75,6.125) .. (-3.625,6.25);
\draw [ color={rgb,255:red,113; green,113; blue,113}, draw opacity=1, line width=2pt, short] (-3,9) .. controls (-3,9.25) and (-3.125,9.375) .. (-3.125,9.5);
\draw [ color={rgb,255:red,113; green,113; blue,113}, draw opacity=1, line width=2pt, short] (-3.125,9.5) .. controls (-3.125,9.75) and (-3.25,9.75) .. (-3.125,10);
\draw [ color={rgb,255:red,113; green,113; blue,113}, draw opacity=1, line width=2pt, short] (-3.125,10) .. controls (-3,10.25) and (-3.125,10.25) .. (-3,10.5);
\draw [ color={rgb,255:red,113; green,113; blue,113}, draw opacity=1, line width=2pt, short] (-3,10.5) .. controls (-2.875,10.75) and (-2.75,10.75) .. (-2.75,11);
\draw [ color={rgb,255:red,113; green,113; blue,113}, draw opacity=1, line width=2pt, short] (-2.75,11) .. controls (-2.625,11.25) and (-2.75,11.375) .. (-2.625,11.5);
\draw [ color={rgb,255:red,113; green,113; blue,113}, draw opacity=1, line width=2pt, short] (-2.625,11.5) .. controls (-2.375,11.75) and (-2.5,11.75) .. (-2.25,11.875);
\draw [ color={rgb,255:red,113; green,113; blue,113}, draw opacity=1, line width=2pt, short] (-2.25,11.875) .. controls (-2,12) and (-2,11.875) .. (-1.75,12);
\draw [ color={rgb,255:red,113; green,113; blue,113}, draw opacity=1, line width=2pt, short] (-1.75,12) .. controls (-1.5,12.125) and (-1.5,12.125) .. (-1.25,12.25);
\draw [ color={rgb,255:red,113; green,113; blue,113}, draw opacity=1, line width=2pt, short] (-1.25,12.25) .. controls (-1,12.5) and (-1.125,12.5) .. (-0.875,12.625);
\draw [ color={rgb,255:red,113; green,113; blue,113}, draw opacity=1, line width=2pt, short] (-0.875,12.625) .. controls (-0.625,12.75) and (-0.625,12.625) .. (-0.375,12.625);
\draw [ color={rgb,255:red,113; green,113; blue,113}, draw opacity=1, line width=2pt, short] (-0.375,12.625) .. controls (-0.125,12.625) and (-0.125,12.625) .. (0.125,12.5);
\draw [ color={rgb,255:red,113; green,113; blue,113}, draw opacity=1, line width=2pt, short] (-3,9) .. controls (-3,8.875) and (-2.875,8.75) .. (-2.75,8.5);
\draw [ color={rgb,255:red,113; green,113; blue,113}, draw opacity=1, line width=2pt, short] (-2.75,8.5) .. controls (-2.625,8.25) and (-2.75,8.25) .. (-2.75,8);
\draw [ color={rgb,255:red,113; green,113; blue,113}, draw opacity=1, line width=2pt, short] (-2.75,8) .. controls (-2.75,7.75) and (-2.75,7.75) .. (-2.875,7.5);
\draw [ color={rgb,255:red,113; green,113; blue,113}, draw opacity=1, line width=2pt, short] (-2.875,7.5) .. controls (-3,7.25) and (-2.875,7.25) .. (-2.875,7);
\draw [ color={rgb,255:red,113; green,113; blue,113}, draw opacity=1, line width=2pt, short] (-2.875,7) .. controls (-2.875,6.875) and (-3,6.75) .. (-3.125,6.5);
\draw [ color={rgb,255:red,113; green,113; blue,113}, draw opacity=1, line width=2pt, short] (-3.125,6.5) .. controls (-3.375,6.375) and (-3.375,6.375) .. (-3.625,6.25);
\draw [line width=2pt, dash pattern=on 3\pgflinewidth off 1.5\pgflinewidth] (3.75,4.375) .. controls (4,4.625) and (4.125,4.625) .. (4.125,4.75);
\draw [line width=2pt, dash pattern=on 3\pgflinewidth off 1.5\pgflinewidth] (4.125,4.75) .. controls (4.125,5) and (4,5) .. (4.125,5.25);
\draw [line width=2pt, dash pattern=on 3\pgflinewidth off 1.5\pgflinewidth] (4.125,5.25) .. controls (4.25,5.5) and (4.25,5.625) .. (4.375,5.75);
\draw [line width=2pt, dash pattern=on 3\pgflinewidth off 1.5\pgflinewidth] (4.375,5.75) .. controls (4.625,6) and (4.5,6) .. (4.75,6.125);
\draw [line width=2pt, dash pattern=on 3\pgflinewidth off 1.5\pgflinewidth] (4.75,6.125) .. controls (5,6.25) and (5,6.125) .. (5.125,6.25);
\draw [line width=2pt, dash pattern=on 3\pgflinewidth off 1.5\pgflinewidth] (5.75,9) .. controls (5.75,9.25) and (5.625,9.375) .. (5.625,9.5);
\draw [ color={rgb,255:red,113; green,113; blue,113}, draw opacity=1, line width=2pt, short] (5.625,9.5) .. controls (5.625,9.75) and (5.5,9.75) .. (5.625,10);
\draw [ color={rgb,255:red,113; green,113; blue,113}, draw opacity=1, line width=2pt, short] (5.625,10) .. controls (5.75,10.25) and (5.625,10.25) .. (5.75,10.5);
\draw [ color={rgb,255:red,113; green,113; blue,113}, draw opacity=1, line width=2pt, short] (5.75,10.5) .. controls (5.875,10.75) and (6,10.75) .. (6,11);
\draw [ color={rgb,255:red,113; green,113; blue,113}, draw opacity=1, line width=2pt, short] (6,11) .. controls (6.125,11.25) and (6,11.375) .. (6.125,11.5);
\draw [ color={rgb,255:red,113; green,113; blue,113}, draw opacity=1, line width=2pt, short] (6.125,11.5) .. controls (6.375,11.75) and (6.25,11.75) .. (6.5,11.875);
\draw [ color={rgb,255:red,113; green,113; blue,113}, draw opacity=1, line width=2pt, short] (6.5,11.875) .. controls (6.75,12) and (6.75,11.875) .. (7,12);
\draw [ color={rgb,255:red,113; green,113; blue,113}, draw opacity=1, line width=2pt, short] (7,12) .. controls (7.25,12.125) and (7.25,12.125) .. (7.5,12.25);
\draw [ color={rgb,255:red,113; green,113; blue,113}, draw opacity=1, line width=2pt, short] (7.5,12.25) .. controls (7.75,12.5) and (7.625,12.5) .. (7.875,12.625);
\draw [ color={rgb,255:red,113; green,113; blue,113}, draw opacity=1, line width=2pt, short] (7.875,12.625) .. controls (8.125,12.75) and (8.125,12.625) .. (8.375,12.625);
\draw [ color={rgb,255:red,113; green,113; blue,113}, draw opacity=1, line width=2pt, short] (8.375,12.625) .. controls (8.625,12.625) and (8.625,12.625) .. (8.875,12.5);
\draw [line width=2pt, dash pattern=on 3\pgflinewidth off 1.5\pgflinewidth] (5.75,9) .. controls (5.75,8.875) and (5.875,8.75) .. (6,8.5);
\draw [line width=2pt, dash pattern=on 3\pgflinewidth off 1.5\pgflinewidth] (6,8.5) .. controls (6.125,8.25) and (6,8.25) .. (6,8);
\draw [line width=2pt, dash pattern=on 3\pgflinewidth off 1.5\pgflinewidth] (6,8) .. controls (6,7.75) and (6,7.75) .. (5.875,7.5);
\draw [line width=2pt, dash pattern=on 3\pgflinewidth off 1.5\pgflinewidth] (5.875,7.5) .. controls (5.75,7.25) and (5.875,7.25) .. (5.875,7);
\draw [line width=2pt, dash pattern=on 3\pgflinewidth off 1.5\pgflinewidth] (5.875,7) .. controls (5.875,6.875) and (5.75,6.75) .. (5.625,6.5);
\draw [line width=2pt, dash pattern=on 3\pgflinewidth off 1.5\pgflinewidth] (5.625,6.5) .. controls (5.375,6.375) and (5.375,6.375) .. (5.125,6.25);
\draw [line width=2pt, short] (5.625,9.5) .. controls (5.375,9.625) and (5.25,9.5) .. (5,9.625);
\draw [line width=2pt, short] (5,9.625) .. controls (4.875,9.875) and (4.75,10) .. (4.625,10.125);
\draw [line width=2pt, short] (4.625,10.125) .. controls (4.625,10.375) and (4.625,10.5) .. (4.5,10.625);
\draw [line width=2pt, {Stealth[scale=1.5]}-{Stealth[scale=1.5]}, ] (0,8.125) -- (3,8.125);
\draw [ fill={rgb,255:red,0; green,0; blue,0}, fill opacity=1, line width=2pt ] (4.5,10.625) circle (0.125cm);
\draw [ fill={rgb,255:red,0; green,0; blue,0}, fill opacity=1, line width=2pt ] (8.875,12.5) circle (0.125cm);
\draw [ fill={rgb,255:red,0; green,0; blue,0}, fill opacity=1, line width=2pt ] (0.125,12.5) circle (0.125cm);
\node [font=\fontsize{22.2pt}{28.8pt}\selectfont, fill={rgb,255:red,255; green,255; blue,255}, fill opacity=1, text opacity=1, inner xsep=0.080cm, inner ysep=0.085cm, rounded corners=0.020cm] at (-3.875,9.625) {$\omega$};
\node [font=\fontsize{22.2pt}{28.8pt}\selectfont, fill={rgb,255:red,255; green,255; blue,255}, fill opacity=1, text opacity=1, inner xsep=0.080cm, inner ysep=0.085cm, rounded corners=0.020cm] at (6,12.375) {$\omega^+$};
\node [font=\fontsize{22.2pt}{28.8pt}\selectfont, fill={rgb,255:red,255; green,255; blue,255}, fill opacity=1, text opacity=1, inner xsep=0.080cm, inner ysep=0.085cm, rounded corners=0.020cm] at (6.625,7) {$\omega^-$};
\node [font=\fontsize{22.2pt}{28.8pt}\selectfont, fill={rgb,255:red,255; green,255; blue,255}, fill opacity=0, text opacity=1, inner xsep=0.080cm, inner ysep=0.085cm, rounded corners=0.020cm] at (4.375,9.5) {$\nu$};
\end{circuitikz}
}
\caption{Illustration of the branching/activity mechanism described in Section \ref{Section: branching process}. Solid black lines show active paths with circles indicating their (active) endpoints. Dashed lines show inactive paths. Solid grey lines show paths that may be either active or inactive, again with circles indicating their (possibly active) endpoints. On the left we have a single dendrite $\omega$ before branching. On the right, the dendrite after branching.}
\label{fig:branching activity}
\end{figure}
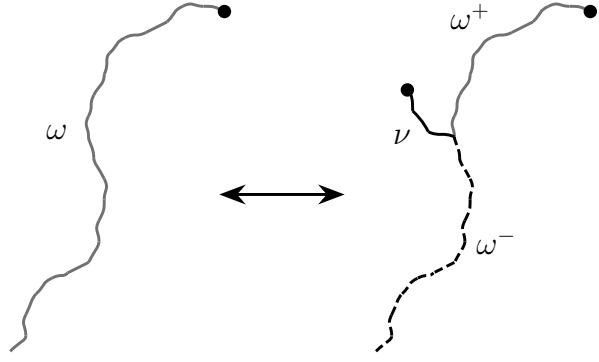

For example if we had an active branch of length $\ell_\omega = 10$ that split at $u=6$ (there is no requirement that the branch point be an integer, this is chosen for simplicity) we would have a lower inactive segment with length $\ell_{\omega^-}=6$ and an upper active segment with length $\ell_{\omega^+}=4$. 

An illustration of this is shown in Figure \ref{fig:branching activity}. Solid black lines show active paths with circles indicating their (active) endpoints. Dashed lines show inactive paths. Solid grey lines show paths that may be either active or inactive, again with circles indicating their (possibly active) endpoints. On the left we have $\omega$ before branching. On the right this has been split into the inactive $\omega^-$ and the remainder $\omega^+$, as the new active branch $\nu$ has been added. 

This process is reversible: if $\nu$ were to retract fully, reaching $\ell=0$ and being removed, then $\omega^-$ and $\omega^+$ would be recombined, inheriting the current activity of $\omega^+$. If $\omega^+$ were instead to retract fully, then $\omega^-$ and $\nu$ would be recombined, inheriting the current activity of $\nu$. 

Activity will only affect whether the length process (and thus the endpoint of the dendrite) is active. New branches can still emerge from `inactive' dendrites and branches coming off inactive dendrites may be either active or inactive. Simply multiplying the Brownian snake component of the generator by $a$ makes inactive paths static. 

The branching point $\omega(u)$ is included in both $\omega^-$ and $\omega^+$, meaning that $\hat{\omega}^- = \omega^-(u) = \omega^+(0)$. As the probability of an unrelated branch beginning or ending at exactly this point is zero, we may use this matching of start- and end-points to determine the relationship between branches. For further details see the graph representation in Section \ref{Section: Graph properties} below. 

For each dendrite we now have the pair 
\begin{align*}
    s_t^i = (\omega_t^i, a_t^i) \in \mathcal{S} := \W \times \{0,1\} \,.
\end{align*}
We also denote by $\ell^i$ the length of $\omega^i$. In general a superscript $i$ is used to refer to the elements or features of a particular branch $s^i$. Since $\mathcal{W}$ is a Polish space, $\mathcal{S}$ is also a Polish space with the natural choice of distance metric
\begin{align*}
    d_S\big( (\omega_1,a_1),(\omega_2,a_2)\big) = |a_1-a_2| + d_\W(\omega_1,\omega_2)\,.
\end{align*}

We also update our generator to reflect this inclusion of activity, denoting the new generator by $\mathcal{B}$. The domain of this generator uses functions from the original domain $\mathcal{D}(\mathcal{A})$, however these functions depend on $\omega$ only. This poses a problem since paths with $a=0$ do not evolve while paths with $a=1$ do, so the generator $\mathcal{B}$ (which describes the evolution) requires information about $a$ as well as $\omega$. To fix this we simply combine two functions from $\mathcal{D}(\mathcal{A})$:
\begin{align*}
    \mathcal{D}(\mathcal{B}) &= \{f(s) := (1-a) f_0(\omega) + a f_1(\omega) \,|\, f_0,f_1 \in \mathcal{D}(\mathcal{A}) \} \,, \\[0.2em]
    \mathcal{B}f(s) &= a \mathcal{A}f_1(\omega) \,.
\end{align*}
Functions $f \in \mathcal{D}(\mathcal{B})$ are equal to $f_0$ when $a=0$ and $f_1$ when $a=1$, so have the required dependence on $a$. Also, $\mathcal{B}f$ includes only $\mathcal{A}f_1$ as paths with $a=0$ do not evolve, reflecting the different behaviours of active and inactive paths. 

\begin{remark} \label{Remark: on including labels}
    In the same way that the activity of each branch has been added to the state space for a single neuron, it is also possible to include other labels/information. For example, we may wish to store the distance back to the soma as this could affect the growth process. For each branch, add a variable $d^i$ that records the distance (along the dendrite) from $\omega^i(0)$ back to the soma and add this variable to the state space $\mathcal{S}$. If a new branch, with index $j$, is created at $\omega^i(u)$ then set $d^j = d^i + u$. In principle, any desired information or labels could be appended in this way; for another example, see Section \ref{Section: Extensions} for more details on possible model extensions.
\end{remark}

In this section, we complete the construction of our model in three steps: firstly by introducing a measure-valued process to track multiple dendrites, then by adding a jump process corresponding to the creation of new branches, and finally by considering two methods for removing branches with length zero. The final result is the presentation of a generator that defines the stochastic process for the entire neuron. 

\subsubsection{Tracking multiple dendrites} \label{Section: Tracking multiple dendrites}

For a fixed number $N \in \mathbb{N}$ of independent dendrite branches $s_t^i = (\omega_t^i,a_t^i), i=1,\dots,N$, we may consider the joint process as an element $s^{(N)}_t = (s^1_t,\dots,s_t^N) \in S^N$. However, it is more useful to consider the corresponding empirical measure
\begin{align} \label{Eqn: Definition of empirical measure}
    \mu^N_t = \sum_{i=1}^N \delta\{s^i_t\}\,,
\end{align}
where $\delta\{s\}$ is a Dirac measure placed at $s \in \mathcal{S}$. This allows the same state space (a subset of the space of measures, defined below) to be used for any $N$, allowing the number of branches to change without changing the state space. Mappings of this form into a space of measures are common in the study of measure-valued branching processes and superprocesses \cite{horton2023stochastic,li2011measure}. Note that this measure $\mu^N$ is not a probability measure over possible positions but rather tracks the actual position of each branch in the space $S$. To be more specific about the state space, we introduce the following definitions from \cite{daley2008introduction}. 
\begin{definition}
    Denote by $\mathcal{X}$ some complete separable metric space (such as our $\mathcal{S}$) and by $\mathcal{M}(\mathcal{X})$ the space of measures on $\mathcal{X}$. A measure $\mu$ is called:
    \begin{enumerate}[label=\alph*)]
        \item \textit{boundedly finite} if $\mu(B)<\infty$ for every bounded Borel set $B$. The space of such measures is denoted by $\mathcal{M}_\mathcal{X}^\#$.
        \item a \textit{counting measure} if it is a boundedly finite, integer-valued measure. The space of such measures is denoted by $\mathcal{N}_\mathcal{X}^\#$.
    \end{enumerate}
\end{definition}

By \cite[Proposition 9.1.3]{daley2008introduction}, any measure $\mu \in \mathcal{N}_\mathcal{S}^\#$ can be written in the form 
\begin{align} \label{Eqn: Form of bfivm's}
    \mu = \sum_{i} \kappa^i \,\delta\{s^i\}
\end{align}
where $s^i \in \mathcal{S}$ are a countable set of points with finitely many $s^i$ in any bounded Borel set, and $\kappa^i$ are positive integers. For $\mu \in \mathcal{N}_\mathcal{S}^\#$ denote by $[\mu]$ the set of these $s^i$ and by $|[\mu]| \in \mathbb{N}\cup\{\infty\}$ their number. The space $\mathcal{N}_\mathcal{S}^\#$ is a complete separable metric space under an appropriate topology \cite[Proposition 9.1.4, see Section 9.1 for further details]{daley2008introduction}. For these reasons we consider the evolution of the empirical measure $\mu_t^N$ \eqref{Eqn: Definition of empirical measure} as an element of $\mathcal{N}_\mathcal{S}^\#$. The topology of this space requires us to consider test functions $f$ that vanish outside some bounded set. Fortunately the restriction $\ell \in [0, \ell_{\max}]$ using a reflecting boundary condition (introduced at the end of Section \ref{Section: Generator}) and the assumption that the spatial domain $X$ is compact (Assumption \ref{Assumption: X compact}) ensures that each $s^i$ remains in a bounded subset of $S$ and so $f$ may be simply set to zero outside this subset. 

We now move from the generator $\mathcal{B}$ (for a single dendrite with activity) to the generator for the empirical measure \eqref{Eqn: Definition of empirical measure}, denoted $\mathcal{C}$ using the setup in \cite[Section 4.2]{kurtz1998martingale}. Define
\begin{align*}
    \mathcal{D}'(\mathcal{B}) = \left\{ \prod_{i=1}^n f_i(s_i) \,\Big|\, f_i \in \mathcal{D}(\mathcal{B}), n=1,2,\dots \right\} \,,
\end{align*}
that is products of finitely many elements of $\mathcal{D}(\mathcal{B})$. For a measure $\mu$ and bounded function $f \in L^1(X)$ denote
\begin{align*}
    \langle f,\mu\rangle = \int_X f(x) \, \mu(dx) \,.
\end{align*}
We then have a generator for $\mu^N$ and domain of the form
\begin{subequations} \label{Eqn: Measure-valued generator and domain}
\begin{align}
    \mathcal{C} F(\mu^N) &= \langle \mathcal{B}f,\mu^N \rangle \,, \label{Eqn: Initial measure-valued generator} \\
    \mathcal{D}(\mathcal{C}) &= \{ F(\mu^N) = \langle f,\mu^N\rangle \,|\, f \in \mathcal{D}'(\mathcal{B}) \} \,, \label{Eqn: Measure-valued generator domain}
\end{align}
\end{subequations}
noting that this is, in fact, the same generator for all $N$. For $\mu^N$ of the form \eqref{Eqn: Definition of empirical measure} the martingale problem for this generator maps to a collection of independent, exchangeable martingale problems for the individual Brownian snakes and so is well-posed by \cite{kurtz1998martingale}. We may now track an arbitrary, finite number of active/inactive dendrites, so next consider the creation and removal of branches.

\begin{remark} \label{Remark on finite initial condition}
    The initial condition $\mu_0$ has all $s^i$ in a bounded set, since $\ell^i \in [0,\ell_{\max}]$ and $X$ is bounded, hence $\mu_0$ must have a finite number of atoms. Proposition \ref{Prop: number of atoms in mu finite} below shows that this property is maintained over time. 
\end{remark}

\subsubsection{Adding branches} \label{Section: New branches}

For each $\omega \in \W$, let $\Lambda(\omega) \in \mathbb{R}^+$ be the rate at which new branches are created from $\omega$. We assume that $\Lambda$ is bounded, that is there exists some constant $C_\Lambda$ such that $\Lambda(\omega) \leq C_\Lambda$ for all $\omega \in \W$. If a branch is created from $\omega$, let $\rho_\omega(du) \in \mathcal{P}([0,\ell_\omega])$ be a probability measure determining the starting position of the new branch. We assume that this probability is absolutely continuous with respect to the Lebesgue measure, meaning that it has a well-defined density (also denoted $\rho_\omega$) with 
\begin{align*}
    \rho_\omega(U) = \int_U \rho_\omega(u) \, du
\end{align*}
for any measurable $U \subset [0,\ell_\omega]$. In particular this means that the probability of a new branch emerging from exactly the start or end point of $\omega$ is zero. 

For example, if branches occur at a constant rate $\lambda$ per unit length of dendrite, we have $\Lambda(\omega)=\lambda \ell_\omega \leq \lambda \ell_{\max}$ and $\rho_\omega(du)=\ell^{-1} du$. More generally one could specify some function $f(u,\ell)$ such that 
\begin{align*}
    \Lambda(\omega) = \int_0^\ell f(u,\ell) \, du \,,\quad \rho_\omega(du) = f(u,\ell)/\Lambda(\omega)\,,
\end{align*}
so that branches could be made more or less likely at certain points along the dendrite, but we do not impose this specific structure. Alternatively the branching rate may depend on some internal chemical concentration, see for example \cite{graham2004transport}, that should be tracked along the branch. This mix of continuous dendrite development and random branch creation also mirrors the approach in NETMORPH \cite{koene2009netmorph}, although we track here the whole dendrite rather than only its endpoint.

Since branch creation is incorporated as a jump process on the state space $\mathcal{N}_S^{\#}$, the jump rates may depend on features of the whole measure $\mu$. For example, all jumps rates could be slowed by a factor
\begin{align} \label{Eqn: length plus 1}
    1+\int \ell \, d\mu = 1 + \sum_{s^i \in [\mu]} \ell^i\,,
\end{align}
which measures the total length of all current branches (plus one to avoid any possible division by zero) to limit the overall speed of growth. Note that as $\mu$ is a sum of Dirac masses, the integral over $\mu$ in \eqref{Eqn: length plus 1} is in fact simply the sum over the branches. In the case that there is dependence of the neuron as a whole, we denote jump rates by $\Lambda(\mu,\omega)$.

If a branch occurs at $\omega(u)$ we split the branch into $\omega^-$ and $\omega^+$ as described at the start of Section \ref{Section: branching process}. The new branch, again denoted $\nu\in\W$ is given $\ell_\nu = \ell_0$ for some small constant $\ell_0>0$ to avoid immediate removal. Immediate retraction could also be solved using lifetime and growth processes that depend on the age of the branch, see Example 5 in Section \ref{Section: Extensions} for details. The position of the branch has the law $\{\xi_u , 0 \leq u \leq \ell_0\}$ started from $\nu(0)=\omega(u)\in X$. Let $P(d\nu)$ be the distribution over $\W$ of this new branch and $\mu^+(\omega,a_\omega,u,\nu)$ be the resulting measure after adding this branch and making activity updates, where $a_\omega$ is the activity of the original parent branch $\omega$. Recall that $a_\omega$ does not affect the branching rate, only the activity of $\omega^+$. 

These updates are captured by the jump process generator
\begin{align*}
    \mathcal{C}^+F(\mu) &= \sum_{i=1}^N \Lambda(\mu,\omega^i) \, \int_\W \int_0^{\ell^i} \Big( F(\mu^+(\omega^i,a^i,u,\nu)) - F(\mu) \Big) \, \rho_{\omega^i}(du) \, P(d\nu)\\
    &= \bigg\langle \int_\W \int_0^{\ell} \Lambda(\mu,\omega) \,\Big( F(\mu^+(\omega,a,u,\nu)) - F(\mu) \Big) \, \rho_{\omega}(du) \,P(d\nu) \,,\mu \bigg\rangle \,,
\end{align*}
where $F$ is any bounded function of $\mu$. 

\subsubsection{Removing branches} \label{Section: removing branches}

It is necessary to remove branches with $\ell=0$ from the measure $\mu_t$ so that we stop tracking them. As these branches are absorbed at $\ell=0$ they will not evolve further, but should be removed so that we may update the activity of other branches. When doing so we must also recombine the sibling of branch $\omega$ with the parent of branch $\omega$, according to the setup described in Section \ref{Section: New branches}. Let $\mu^-(\omega,a)$ denote the resulting measure. 

In reality a fully retracted branch has length zero so has already vanished, but it is important for the mathematical structure of the model that these branches are removed from the measure. One option is to remove branches using stopping times. When a dendrite hits the $\ell=0$ boundary the process could be stopped, $\mu$ replaced by $\mu^-$ and the process continued. This would effectively remove the branch immediately, but would be a boundary condition that is imposed in addition to the generator. Alternatively we could add a final jump process to the generator that removes branches will $\ell=0$ at some high rate (denoted $\lambda$, with $\lambda\gg 1$). The advantage of including these removals as a jump with a fast rate is that the entire dynamics is still captured within a single generator. 

These branch removals are captured by the jump process generator 
\begin{align*}
    \mathcal{C}^-F(\mu) &= \lambda \sum_{i=1}^N \mathds{1}\{\ell^i=0\} \Big( F(\mu^-(\omega^i,a^i)) - F(\mu) \Big) \\
    &= \lambda \bigg\langle \mathds{1}\{\ell_\omega=0\} \Big( F(\mu^-(\omega,a)) - F(\mu) \Big) \,, \mu \bigg\rangle  \,,
\end{align*}
where again $F$ is any bounded function of $\mu$. 

\begin{remark}
    The jump rate $\lambda$ represents the rate at which we removed the branch from being tracked by the measure and is not the rate of a physical process. In both simulation methods discussed in Section \ref{Section: numerical simulation} we do not track the measure directly so in practice this removal rate $\lambda$ is not significant. 
\end{remark}

\subsubsection{Combined generator}

Applying Theorem 4.10.3 \cite{ethier2009markov} we may combine the `continuous' part of the generator for the evolution of each dendrite with the `jump' part for the addition and removal of branches, giving that the martingale problem for the following generator is well-posed: 
\begin{align} \label{Eqn: Full generator}
    \mathcal{L}F(\mu) &= \mathcal{C} F(\mu) + \mathcal{C}^+F(\mu) + \mathcal{C}^-F(\mu) \,,
\end{align}
with domain $\mathcal{D}(\mathcal{L}) = \mathcal{D}(\mathcal{C})$ given in \eqref{Eqn: Measure-valued generator domain}.
By specifying this generator, which defines a unique stochastic process, we have completed the construction of the model for the whole neuron. 

\subsubsection{Connection to soma} \label{Section: Soma branching}

So far we have implicitly assumed that the number of processes leaving the soma remains constant, as branching only occurs along existing dendritic paths. If desired, an additional jump process could be included to allow new processes to extend from the soma at some rate $\lambda_s$. This would simply require the position of the soma to be specified (taking the origin as the standard position) and a new jump process of the following form added to the generator:
\begin{align*}
    \bigg\langle \lambda_s \int \,\big( F(\mu^+) - F(\mu) \big) \,P(d\nu) ,\mu \bigg\rangle\,,
\end{align*}
where $\mu^+$ is the measure $\mu$ plus a new branch $\nu$ added with $\nu(0)=0$ and $\ell_\nu = \ell_0$. Note that different processes extending from the soma are independent. 

\section{Properties} \label{Section: Properties}

For simplicity, throughout this section we assume that there is only a single process extending from the soma. Intuitively (with this being proven below) this means that the dendrites form a single binary tree. The results below can be easily extended to multiple processes extending from the soma, as these are independent. 

We begin with several simple properties of the process $\mu_t$. 

\begin{proposition}
    Assume that $\mu_0 \in \mathcal{N}_\mathcal{S}^{\#}$ has the form 
    \begin{align*}
        \mu = \sum_{i} \delta\{s^i\}
    \end{align*}
    for $s^i \in \mathcal{S}$ (meaning $\kappa^i=1$ for all $i$ in \eqref{Eqn: Form of bfivm's}). Then this remains true of $\mu_t$ for all $t\geq0$. 
\end{proposition}
\begin{proof}
    The `continuous' part of the generator $\mathcal{C}$ does not change the weight of Dirac measures, only the position $s^i$. Moreover, the jump process for branching introduces new Dirac measures with mass $\kappa^i=1$. As discussed in Section \ref{Section: New branches} the probability of any two branches sharing the same starting point is zero, therefore the probability of overlapping Dirac measures is also zero. Finally, the jump process for branch removal removes Dirac measures rather than changing their mass. Hence none of the processes in the generator of $\mu_t$ would lead to a Dirac measure with $\kappa^i \neq 1$. 
\end{proof}
Therefore the process remains within the subspace $\mathcal{N} \subset \mathcal{N}_\mathcal{S}^{\#}$ of those boundedly finite, integer valued-measures with $\kappa^i=1$ for all $i$. In \cite{daley2008introduction} this subspace is denoted $\mathcal{N}_\mathcal{S}^{\#*}$. 

\begin{proposition} \label{Prop: number of atoms in mu finite}
    The number of branches $|[\mu_t]|$ is finite almost surely for all $t \geq 0$. 
\end{proposition}
\begin{proof}
    This is true at $t=0$ by Remark \ref{Remark on finite initial condition}. For a given branch, the rate at which new branches are created is bounded above by a constant $C_\Lambda$. Hence the number of atoms $|[\mu_t]|$ is bounded above by a birth process with birth rate $\lambda_n = C_\Lambda n$. Hence $|[\mu_t]|$ is almost surely finite by \cite[Theorem 2.5.2]{norris1998markov}.
\end{proof}
Therefore $\mu_t$ is always a sum of finitely many Dirac masses. We next consider more complex properties of $\mu_t$, specifically its graph representation and the length and number of active/inactive branches. 

\subsection{Graph structure} \label{Section: Graph properties}

From the measure $\mu_t$ we may overlay in $X$ the paths $\omega^i_t$ for each $s^i_t \in [\mu_t]$. As new branches have their start-points along existing ones, we expect these paths to form a connected graph structure. The goal of this section is to study some of the properties of this graph representation. 

\begin{definition} \label{Definition: graph representation}
    The \textit{graph representation} of a measure $\mu \in \mathcal{N}_\mathcal{S}^{\#}$ is a graph $G=(V,E)$ describing the spatial embedding of $\mu$, defined as follows. Let 
    \begin{align*}
        V &= \Big\{ x \in X \, \Big| \, \exists \, (\omega,a) \in [\mu] \text{ s.t. } x = \omega(0) \text{ or } x=\hat{\omega}  \Big\} \,, \\
        E &= \Big\{ (x,y) \in V \times V \, \Big| \, \exists \, (\omega,a) \in [\mu] \text{ s.t. } \omega(0)=x \text{ and } \hat{\omega}= y  \Big\} \,.
    \end{align*}
    To an edge $(x,y) \in E$ we assign the corresponding path $\omega_{xy} \in \mathcal{W}$ to be the element $(\omega_{xy},a) \in [\mu]$ with $\omega_{xy}(0)=x$ and $\hat{\omega}_{xy} = y$. 
    This network is fully determined by $\mu$ and so defines a mapping $\mu \mapsto (V,E)$ from $\mathcal{N}_\mathcal{S}^\#$ into a space of networks. Denote this mapping by $G(\mu)$. 
\end{definition}

The graph $G(\mu)$ is directed, following the natural direction along paths away from the soma. In addition, by Proposition \ref{Prop: number of atoms in mu finite}, $G(\mu)$ is almost surely a finite graph. As each path in the measure is included as an edge we will have $|E| = |[\mu]|$. At this point the size of the node set $|V|$ is unknown, but we have the bounds $0 \leq |V| \leq 2|[\mu]|$ as each path in $\mu$ can contribute at most two nodes to $V$. 

The first use of this graph is to introduce the notion of connectedness. 
\begin{definition}
    A directed graph $G$ is \textit{weakly connected} if, when all directed edges are replaced with undirected edges, the resulting graph is connected. Consequently, a measure $\mu \in \mathcal{N}_\mathcal{S}^\#$ is called \textit{connected} if $G(\mu)$ is weakly connected. 
\end{definition}
If $|[\mu]|=0$ then $G$ is the empty network, so both are immediately connected. 

Using this graph structure, we see that the branching-retraction problem, discussed at the start of Section \ref{Section: branching process} and illustrated in Figure \ref{fig: branching-retraction problem}, is precisely a loss of connectedness. Theorem \ref{Thm: Inherited graph properties} below shows that tracking branch activity solves this problem. 

We can also use the graph representation to be more specific about the initial condition $\mu_0$. Define the subset of $\mathcal{N}_T \subset \mathcal{N}$, as the set of measures $\mu$ with the following properties
\begin{enumerate}[label=\alph*)]
    \item $|[\mu]|$ is finite. 
    \item $\mu$ is connected. 
    \item $G(\mu)$ is a rooted tree graph whose root node (which corresponds to the soma) and leaf nodes have degree 1. All other nodes have degree 3, specifically with in-degree 1 and out-degree 2. 
    \item A path $(\omega,a) \in [\mu]$ is active ($a=1$) if and only if $\hat{\omega}$ is a leaf node of $G(\mu)$. 
\end{enumerate}
We call $\mathcal{N}_T$ the set of \textit{tree measures} and define $\mathcal{G}_T = \{G(\mu) : \mu \in \mathcal{N}_T\}$. It is precisely these graphs $G \in \mathcal{G}_T$ that describe the graph structure observed in real neurons. We now show that the stochastic process $\mu_t$ is a tree measure. 

\begin{theorem} \label{Thm: Inherited graph properties}
    If $\mu_0 \in \mathcal{N}_T$ then $\mu_t \in \mathcal{N}_T$ almost surely for all $t\geq0$. 
\end{theorem}
\begin{proof}
    Before proceeding with the remainder of the proof, note that Proposition \ref{Prop: number of atoms in mu finite} above guarantees that property a) holds almost surely. 

    The remainder of the proof proceeds by induction over jump times. We investigate the action of the continuous part of the generator and the jump part separately. 

    First, consider the continuous evolution of each individual path. As this evolution does not change the number or activity of paths, property d) is unaffected. If the path is inactive, then there is no change and so all other properties are unaffected. For an active $\omega$ path, we have by property d) that $\hat{\omega}$ is a leaf node of $G(\mu)$, meaning that there are no branches coming off this path. Therefore its evolution does not affect the connectedness of $\mu$ and so property b) is unaffected. The only change that could violate property c) is if the endpoint $\hat{\omega}$ coincided with the endpoint of another path (as this would create a cycle in $G(\mu)$). However, due to the randomness in both the length and growth processes, the probability of such an occurrence is zero. 

    \begin{figure}[!ht]
    \centering
    \resizebox{.6\textwidth}{!}{
    \begin{circuitikz}
    \tikzstyle{every node}=[font=\fontsize{27.9pt}{36.2pt}\selectfont]
    \draw [ color={rgb,255:red,113; green,113; blue,113}, draw opacity=1, line width=2pt, short] (-5,4.375) .. controls (-4.75,4.625) and (-4.625,4.625) .. (-4.625,4.75);
    \draw [ color={rgb,255:red,113; green,113; blue,113}, draw opacity=1, line width=2pt, short] (-4.625,4.75) .. controls (-4.625,5) and (-4.75,5) .. (-4.625,5.25);
    \draw [ color={rgb,255:red,113; green,113; blue,113}, draw opacity=1, line width=2pt, short] (-4.625,5.25) .. controls (-4.5,5.5) and (-4.5,5.625) .. (-4.375,5.75);
    \draw [ color={rgb,255:red,113; green,113; blue,113}, draw opacity=1, line width=2pt, short] (-4.375,5.75) .. controls (-4.125,6) and (-4.25,6) .. (-4,6.125);
    \draw [ color={rgb,255:red,113; green,113; blue,113}, draw opacity=1, line width=2pt, short] (-4,6.125) .. controls (-3.75,6.25) and (-3.75,6.125) .. (-3.625,6.25);
    \draw [ color={rgb,255:red,113; green,113; blue,113}, draw opacity=1, line width=2pt, short] (-3,9) .. controls (-3,9.25) and (-3.125,9.375) .. (-3.125,9.5);
    \draw [ color={rgb,255:red,113; green,113; blue,113}, draw opacity=1, line width=2pt, short] (-3.125,9.5) .. controls (-3.125,9.75) and (-3.25,9.75) .. (-3.125,10);
    \draw [ color={rgb,255:red,113; green,113; blue,113}, draw opacity=1, line width=2pt, short] (-3.125,10) .. controls (-3,10.25) and (-3.125,10.25) .. (-3,10.5);
    \draw [ color={rgb,255:red,113; green,113; blue,113}, draw opacity=1, line width=2pt, short] (-3,10.5) .. controls (-2.875,10.75) and (-2.75,10.75) .. (-2.75,11);
    \draw [ color={rgb,255:red,113; green,113; blue,113}, draw opacity=1, line width=2pt, short] (-2.75,11) .. controls (-2.625,11.25) and (-2.75,11.375) .. (-2.625,11.5);
    \draw [ color={rgb,255:red,113; green,113; blue,113}, draw opacity=1, line width=2pt, short] (-2.625,11.5) .. controls (-2.375,11.75) and (-2.5,11.75) .. (-2.25,11.875);
    \draw [ color={rgb,255:red,113; green,113; blue,113}, draw opacity=1, line width=2pt, short] (-2.25,11.875) .. controls (-2,12) and (-2,11.875) .. (-1.75,12);
    \draw [ color={rgb,255:red,113; green,113; blue,113}, draw opacity=1, line width=2pt, short] (-1.75,12) .. controls (-1.5,12.125) and (-1.5,12.125) .. (-1.25,12.25);
    \draw [ color={rgb,255:red,113; green,113; blue,113}, draw opacity=1, line width=2pt, short] (-1.25,12.25) .. controls (-1,12.5) and (-1.125,12.5) .. (-0.875,12.625);
    \draw [ color={rgb,255:red,113; green,113; blue,113}, draw opacity=1, line width=2pt, short] (-0.875,12.625) .. controls (-0.625,12.75) and (-0.625,12.625) .. (-0.375,12.625);
    \draw [ color={rgb,255:red,113; green,113; blue,113}, draw opacity=1, line width=2pt, short] (-0.375,12.625) .. controls (-0.125,12.625) and (-0.125,12.625) .. (0.125,12.5);
    \draw [ color={rgb,255:red,113; green,113; blue,113}, draw opacity=1, line width=2pt, short] (-3,9) .. controls (-3,8.875) and (-2.875,8.75) .. (-2.75,8.5);
    \draw [ color={rgb,255:red,113; green,113; blue,113}, draw opacity=1, line width=2pt, short] (-2.75,8.5) .. controls (-2.625,8.25) and (-2.75,8.25) .. (-2.75,8);
    \draw [ color={rgb,255:red,113; green,113; blue,113}, draw opacity=1, line width=2pt, short] (-2.75,8) .. controls (-2.75,7.75) and (-2.75,7.75) .. (-2.875,7.5);
    \draw [ color={rgb,255:red,113; green,113; blue,113}, draw opacity=1, line width=2pt, short] (-2.875,7.5) .. controls (-3,7.25) and (-2.875,7.25) .. (-2.875,7);
    \draw [ color={rgb,255:red,113; green,113; blue,113}, draw opacity=1, line width=2pt, short] (-2.875,7) .. controls (-2.875,6.875) and (-3,6.75) .. (-3.125,6.5);
    \draw [ color={rgb,255:red,113; green,113; blue,113}, draw opacity=1, line width=2pt, short] (-3.125,6.5) .. controls (-3.375,6.375) and (-3.375,6.375) .. (-3.625,6.25);
    \draw [line width=2pt, dash pattern=on 3\pgflinewidth off 1.5\pgflinewidth] (3.75,4.375) .. controls (4,4.625) and (4.125,4.625) .. (4.125,4.75);
    \draw [line width=2pt, dash pattern=on 3\pgflinewidth off 1.5\pgflinewidth] (4.125,4.75) .. controls (4.125,5) and (4,5) .. (4.125,5.25);
    \draw [line width=2pt, dash pattern=on 3\pgflinewidth off 1.5\pgflinewidth] (4.125,5.25) .. controls (4.25,5.5) and (4.25,5.625) .. (4.375,5.75);
    \draw [line width=2pt, dash pattern=on 3\pgflinewidth off 1.5\pgflinewidth] (4.375,5.75) .. controls (4.625,6) and (4.5,6) .. (4.75,6.125);
    \draw [line width=2pt, dash pattern=on 3\pgflinewidth off 1.5\pgflinewidth] (4.75,6.125) .. controls (5,6.25) and (5,6.125) .. (5.125,6.25);
    \draw [line width=2pt, dash pattern=on 3\pgflinewidth off 1.5\pgflinewidth] (5.75,9) .. controls (5.75,9.25) and (5.625,9.375) .. (5.625,9.5);
    \draw [ color={rgb,255:red,113; green,113; blue,113}, draw opacity=1, line width=2pt, short] (5.625,9.5) .. controls (5.625,9.75) and (5.5,9.75) .. (5.625,10);
    \draw [ color={rgb,255:red,113; green,113; blue,113}, draw opacity=1, line width=2pt, short] (5.625,10) .. controls (5.75,10.25) and (5.625,10.25) .. (5.75,10.5);
    \draw [ color={rgb,255:red,113; green,113; blue,113}, draw opacity=1, line width=2pt, short] (5.75,10.5) .. controls (5.875,10.75) and (6,10.75) .. (6,11);
    \draw [ color={rgb,255:red,113; green,113; blue,113}, draw opacity=1, line width=2pt, short] (6,11) .. controls (6.125,11.25) and (6,11.375) .. (6.125,11.5);
    \draw [ color={rgb,255:red,113; green,113; blue,113}, draw opacity=1, line width=2pt, short] (6.125,11.5) .. controls (6.375,11.75) and (6.25,11.75) .. (6.5,11.875);
    \draw [ color={rgb,255:red,113; green,113; blue,113}, draw opacity=1, line width=2pt, short] (6.5,11.875) .. controls (6.75,12) and (6.75,11.875) .. (7,12);
    \draw [ color={rgb,255:red,113; green,113; blue,113}, draw opacity=1, line width=2pt, short] (7,12) .. controls (7.25,12.125) and (7.25,12.125) .. (7.5,12.25);
    \draw [ color={rgb,255:red,113; green,113; blue,113}, draw opacity=1, line width=2pt, short] (7.5,12.25) .. controls (7.75,12.5) and (7.625,12.5) .. (7.875,12.625);
    \draw [ color={rgb,255:red,113; green,113; blue,113}, draw opacity=1, line width=2pt, short] (7.875,12.625) .. controls (8.125,12.75) and (8.125,12.625) .. (8.375,12.625);
    \draw [ color={rgb,255:red,113; green,113; blue,113}, draw opacity=1, line width=2pt, short] (8.375,12.625) .. controls (8.625,12.625) and (8.625,12.625) .. (8.875,12.5);
    \draw [line width=2pt, dash pattern=on 3\pgflinewidth off 1.5\pgflinewidth] (5.75,9) .. controls (5.75,8.875) and (5.875,8.75) .. (6,8.5);
    \draw [line width=2pt, dash pattern=on 3\pgflinewidth off 1.5\pgflinewidth] (6,8.5) .. controls (6.125,8.25) and (6,8.25) .. (6,8);
    \draw [line width=2pt, dash pattern=on 3\pgflinewidth off 1.5\pgflinewidth] (6,8) .. controls (6,7.75) and (6,7.75) .. (5.875,7.5);
    \draw [line width=2pt, dash pattern=on 3\pgflinewidth off 1.5\pgflinewidth] (5.875,7.5) .. controls (5.75,7.25) and (5.875,7.25) .. (5.875,7);
    \draw [line width=2pt, dash pattern=on 3\pgflinewidth off 1.5\pgflinewidth] (5.875,7) .. controls (5.875,6.875) and (5.75,6.75) .. (5.625,6.5);
    \draw [line width=2pt, dash pattern=on 3\pgflinewidth off 1.5\pgflinewidth] (5.625,6.5) .. controls (5.375,6.375) and (5.375,6.375) .. (5.125,6.25);
    \draw [line width=2pt, short] (5.625,9.5) .. controls (5.375,9.625) and (5.25,9.5) .. (5,9.625);
    \draw [line width=2pt, short] (5,9.625) .. controls (4.875,9.875) and (4.75,10) .. (4.625,10.125);
    \draw [line width=2pt, short] (4.625,10.125) .. controls (4.625,10.375) and (4.625,10.5) .. (4.5,10.625);
    \draw [line width=2pt, {Stealth[scale=1.5]}-{Stealth[scale=1.5]}, ] (0,8.125) -- (3,8.125);
    \draw [ fill={rgb,255:red,0; green,0; blue,0}, fill opacity=1, line width=2pt ] (4.5,10.625) circle (0.125cm);
    \draw [ fill={rgb,255:red,0; green,0; blue,0}, fill opacity=1, line width=2pt ] (8.875,12.5) circle (0.125cm);
    \node [font=\fontsize{22.2pt}{28.8pt}\selectfont, inner xsep=0.080cm, inner ysep=0.085cm, rounded corners=0.020cm] at (4,10.625) {4};
    \node [font=\fontsize{22.2pt}{28.8pt}\selectfont, inner xsep=0.080cm, inner ysep=0.085cm, rounded corners=0.020cm] at (6.125,9.5) {3};
    \draw [ fill={rgb,255:red,0; green,0; blue,0}, fill opacity=1, line width=2pt ] (-5,4.375) circle (0.125cm);
    \draw [ fill={rgb,255:red,0; green,0; blue,0}, fill opacity=1, line width=2pt ] (5.625,9.5) circle (0.125cm);
    \node [font=\fontsize{22.2pt}{28.8pt}\selectfont, inner xsep=0.080cm, inner ysep=0.085cm, rounded corners=0.020cm] at (9.375,12.5) {2};
    \node [font=\fontsize{22.2pt}{28.8pt}\selectfont, inner xsep=0.080cm, inner ysep=0.085cm, rounded corners=0.020cm] at (-5.5,4.5) {1};
    \draw [ fill={rgb,255:red,0; green,0; blue,0}, fill opacity=1, line width=2pt ] (3.75,4.375) circle (0.125cm);
    \node [font=\fontsize{22.2pt}{28.8pt}\selectfont, inner xsep=0.080cm, inner ysep=0.085cm, rounded corners=0.020cm] at (3.25,4.5) {1};
    \draw [ fill={rgb,255:red,0; green,0; blue,0}, fill opacity=1, line width=2pt ] (0.125,12.5) circle (0.125cm);
    \node [font=\fontsize{22.2pt}{28.8pt}\selectfont, inner xsep=0.080cm, inner ysep=0.085cm, rounded corners=0.020cm] at (0.625,12.5) {2};
    \node [font=\fontsize{22.2pt}{28.8pt}\selectfont, fill={rgb,255:red,255; green,255; blue,255}, fill opacity=1, text opacity=1, inner xsep=0.080cm, inner ysep=0.085cm, rounded corners=0.020cm] at (-3.875,9.625) {$\omega_{12}$};
    \node [font=\fontsize{22.2pt}{28.8pt}\selectfont, fill={rgb,255:red,255; green,255; blue,255}, fill opacity=1, text opacity=1, inner xsep=0.080cm, inner ysep=0.085cm, rounded corners=0.020cm] at (6,12.375) {$\omega_{32}$};
    \node [font=\fontsize{22.2pt}{28.8pt}\selectfont, fill={rgb,255:red,255; green,255; blue,255}, fill opacity=1, text opacity=1, inner xsep=0.080cm, inner ysep=0.085cm, rounded corners=0.020cm] at (6.625,7) {$\omega_{13}$};
    \node [font=\fontsize{22.2pt}{28.8pt}\selectfont, fill={rgb,255:red,255; green,255; blue,255}, fill opacity=0, text opacity=1, inner xsep=0.080cm, inner ysep=0.085cm, rounded corners=0.020cm] at (4.375,9.5) {$\omega_{34}$};
    \end{circuitikz}
    }
    \caption{Illustration of the branching/activity mechanism described in Figure \ref{fig:branching activity} with additional annotation to show the nodes and edge of $G(\mu)$. Solid black lines show active paths. Dashed black lines show inactive paths. Solid grey lines show paths that may be either active or inactive. On the left, a single dendrite before branching that may be either active or inactive. On the right, the dendrite after branching. The path below the branch point $\omega_{13}$ is inactive. Above it there are two paths: the remainder of the original path $\omega_{32}$ (which retains its previous activity) and the new active branch $\omega_{34}$.}
    \label{fig:branching annotated}
    \end{figure}
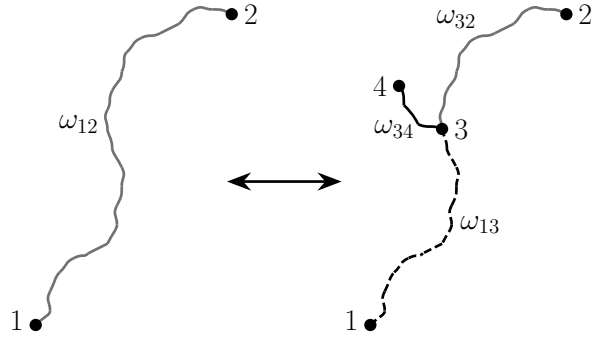

    Next, consider the creation of a new branch. For reference, a version of Figure \ref{fig:branching activity} with additional annotation showing the nodes of $G(\mu)$ is presented in Figure \ref{fig:branching annotated}. Solid black lines show active paths, dashed black lines show inactive paths, and solid grey lines show paths that may be either active or inactive. 
    
    Adding a branch introduces new nodes at the beginning (node 3) and the end (node 4) of the new path ($\omega_{34}$). By construction, the path $\omega_{13}$ connects node 3 to the rest of the graph and the paths $\omega_{32}$ and $\omega_{34}$ connect nodes 2 and 4 respectively, ensuring that property b) is preserved. Here it is crucial that when $\omega_{12}$ is split into $\omega_{13}$ and $\omega_{32}$ their start- and end-points match. Nodes 2 and 4 are leaf nodes with degree 1 and node 3 is a branch node with degree 3. The graph $G(\mu)$ has remained a rooted tree as no cycles have been introduced and the edge direction still follows the natural direction away from the soma/root. Finally, the activity rules have been preserved: $\omega_{13}$ is inactive and node 3 is not a leaf node, $\omega_{34}$ is active and node 4 is a leaf node, $\omega_{32}$ retains the activity of $\omega_{12}$ and these share the end node 2 so $\omega_{32}$ retains the correct activity. 

    Lastly, consider the alternative direction when a branch fully retracts. Assume that the path $\omega_{34}$ reaches $\ell=0$ and is removed, in which case the previous construction is simply reversed and all relevant properties continue to hold. The path $\omega_{13}$ is inactive so it is not possible for it to be removed. Finally, assume that the path $\omega_{32}$ reaches $\ell=0$ and is removed. For this to occur, it must be active, so by property d) must be a leaf node and so the situation is effectively equivalent to the removal of $\omega_{34}$. 

    Hence no part of the evolution of the measure $\mu_t$ causes a loss of any of the properties a)--d) and so $\mu_t \in \mathcal{N}_T$ for all $t\geq0$. 
\end{proof}

\begin{corollary}
    The size of the node set of $G(\mu)$ is given by $|V| = |[\mu]|+1$. 
\end{corollary}
\begin{proof}
    As $G(\mu)$ is a tree graph it follows immediately \cite[Theorem 4.1]{harary1969graph} that  $|V| = |E| + 1=|[\mu]|+1$. 
\end{proof}

\begin{corollary}
    When a branch retracts there are two possible situations: either the process ends entirely, or the retracting branch will have an inactive parent branch and exactly one `sibling' branch (i.e. another branch with the same parent). 
\end{corollary}
\begin{proof}
    For a branch to retract it must be active, meaning that it ends in a leaf node. It cannot begin in a leaf node so must begin in either a root node or a branch node.
    
    If the retracting branch begins at the root node and is removed then the process ends (as we assumed only a single process extending from the root node). 

    The last option is that the retracting branch begins at a branch node, which has in-degree 1 (the parent branch) and out-degree 2, one of which is the retracting branch and the other must therefore be a single `sibling' branch. The parent branch does not end in a leaf node so is inactive. 
\end{proof}

We may also use Theorem \ref{Thm: Inherited graph properties} to show that the graph representation gives an alternative, equivalent process to the measure-valued $\mu_t$. 

\begin{proposition} \label{Prop: G(mu) mapping bijective}
    The mapping from $\mathcal{N}_T$ to $\mathcal{G}_T$ given by $\mu \mapsto G(\mu)$ is bijective. 
\end{proposition}
\begin{proof}
    By definition the mapping is surjective. To show that it is injective assume there exists $\mu,\nu \in \mathcal{N}_T$ with $G(\mu)=G(\nu)$ but $\mu\neq \nu$. As $\mathcal{N}_T \subset \mathcal{N} \subset \mathcal{N}_\mathcal{S}^{\#}$ we can write 
    \begin{align*}
        \mu = \sum_i \delta\{s^i\} \,,\quad \nu = \sum_j \delta\{t^j\} \,,
    \end{align*}
    where $s^i, t^j \in \mathcal{S}$. As $\mu \neq \nu$ there must exist some $s^i$ that is not in the set of $t^j$'s, or vice versa, or both. Without loss of generality, we assume the first of these. Note that by part d) of Theorem \ref{Thm: Inherited graph properties} it is not possible that $\mu$ and $\nu$ have the same $\omega$'s but different $a$'s, therefore this difference must be a difference in the $\omega$'s. Thus there is an edge $\omega^i$ that appears in the edge set of $G(\mu)$ but does not appear in the edge set of $G(\nu)$. This contradicts the assumption that $G(\mu)=G(\nu)$ and therefore we must have that $\mu = \nu$. Thus the mapping is injective, hence it is bijective. 
\end{proof}

As a consequence of Theorem \ref{Thm: Inherited graph properties} and Proposition \ref{Prop: G(mu) mapping bijective} we may use $G(\mu)$ to uniquely map the stochastic process $\mu_t$ to a stochastic process $G(\mu_t)$ with state space $\mathcal{G}_T$. The generator of $G(\mu_t)$ is simply the generator of $\mu_t$ pushed through this mapping. Therefore it is equivalent to think about the stochastic process as being a measure $\mu_t$ or a graph $G(\mu_t)$. 

For the final part of this section we apply Theorem \ref{Thm: Inherited graph properties} to study the number of active and inactive branches. 
\begin{proposition} \label{Prop: Number of (in)active branches}
    Define the processes
    \begin{align*}
        A_t = \int a \, d\mu_t \,, \quad I_t = \int 1-a \, d\mu_t \,,
    \end{align*}
    which count the number of active and inactive branches respectively. 
    Then $A_t$ is equal to the number of dendrite tips and $I_t$ is equal to the number of branch points. 
\end{proposition}
\begin{proof}
    The number of dendrite tips is given by the number of leaf nodes of $G(\mu_t)$ and the number of branch points is given by the number of degree-3 nodes of $G(\mu_t)$. By Theorem \ref{Thm: Inherited graph properties} paths are active if and only if they end in a leaf node, hence the number of active paths $(A_t)$ is immediately equal to the number of leaf nodes and hence the number of dendrite tips. Also, by Theorem \ref{Thm: Inherited graph properties} all nodes have degree 1 or 3. Hence inactive paths, which do not end at a leaf node, must end at a node of degree 3 (as paths cannot end at the soma). The probability of two paths ending at the same position is zero hence the number of inactive paths is equal to the number of nodes of degree 3 and hence the number of branch points. 
\end{proof}

\begin{proposition} \label{Prop: number of active and inactive branches}
    The number of active branches is equal to the number of inactive branches plus 1. 
\end{proposition}
\begin{proof}
    From Proposition \ref{Prop: Number of (in)active branches} we know that the number of active branches is equal to the number of leaf nodes of $G(\mu_t)$ (which have in-degree 1, out-degree 0) and the number of inactive branches is equal to the number of branch nodes (in-degree 1, out-degree 2). Across the network, the sum of in-degrees must be equal to the sum of out-degrees (otherwise there would be an edge missing a beginning or end). Hence
    \begin{align*}
        \sum k_{in} = \sum k_{out} &\iff 0 + (1 \times A_t ) + (1\times I_t) = 1 + (0 \times A_t) + (2 \times I_t) \\
        &\iff A_t + I_t = 1 +2 I_t \\
        & \iff A_t = 1 + I_t \,.
    \end{align*}
    The first term in both sums came from the root node. 
\end{proof}
This result can also be understood more intuitively. Each branching event creates one additional active path and one additional inactive path. Removing paths with $\ell=0$ also removes one active and one inactive path. Therefore, the difference between $A_t$ and $I_t$ is constant over time. If all active branches continually retracted and we removed $\ell=0$ paths we would eventually reduce down to a single active path emanating from the soma. Therefore the difference between $A_t$ and $I_t$ is always equal to 1. 

\subsection{Branch number and length} \label{Section: generators for branch number and length}

We now consider the dynamics of several quantities of interest. First, we examine the total number of branches, given by the mass of the measure $\mu_t$
\begin{align*}
    M(\mu) := \int 1 \, \mu(d\omega,da) \,,
\end{align*}
which corresponds to $f_0(\omega)=f_1(\omega)=1$. The generator $\mathcal{C}$ does not change the mass and so the first term in $\mathcal{L}M$ is zero. Moreover branch creation events increase the mass by $2$ (creating one new active and one new inactive branch) and branch removal events decrease it by $2$ (correspondingly removing one active and one inactive branch). Therefore \eqref{Eqn: Full generator} reduces to, 
\begin{align*}
    \mathcal{L}M(\mu) = 2\big\langle \Lambda(\mu,\omega)  - \lambda \mathds{1}\{\ell_\omega=0\} ,\mu \big\rangle\,.
\end{align*}
The mass is thus a jump process on $\mathbb{N}$ but is not a Markov process in itself as its jump rates depend on the measure $\mu$. 

Using Proposition \ref{Prop: number of active and inactive branches} we deduce that the number of active branches and the number of inactive branches are both jump processes with generator 
\begin{align*}
    \big\langle \Lambda(\mu,\omega)  - \lambda \mathds{1}\{\ell_\omega=0\} ,\mu \big\rangle\,.
\end{align*}

We next consider the total branch length, given by 
\begin{align*}
    L(\mu) :=  \int \ell_\omega \, \mu(d\omega,da) \,,
\end{align*}
which corresponds to $f_0(\omega)=f_1(\omega)=\ell_\omega$. In this case the generator $\mathcal{C}$ does contribute, however using the results of Appendix \ref{Appendix: A} we find that $\mathcal{C} L=0$. There is no difference between active branches (which are driven by a Brownian motion which has increments of mean zero) and inactive branches (which by definition do not change length). This is dependent on the choice of Brownian motion as the length process, for a discussion of alternatives see Section \ref{Section: Extensions}. The removal of branches with length zero naturally has no impact on the total branch length. Hence only the middle term of the generator \eqref{Eqn: Full generator}, namely $\mathcal{C}^+$, remains. In this case \eqref{Eqn: Full generator} reduces to,
\begin{align*}
    \mathcal{L}L(\mu) = \big\langle \Lambda(\mu,\omega) \,\ell_0 ,\mu \big\rangle \,.
\end{align*}

To proceed further we make the following assumptions.
\begin{assumption} \label{Assupmtion: standard branching} 
    We make the following assumptions on the branching process:
    \begin{enumerate}[label=\alph*)]
        \item The branching rate $\Lambda(\mu,\omega) = \beta \ell_\omega$ for some constant $\beta$. 
        \item The branch position is chosen uniformly along the branch, that is $\rho_\omega(du) = \ell^{-1} \, du$. 
    \end{enumerate}
\end{assumption}
Under Assumption \ref{Assupmtion: standard branching} the generator for the length $L_t$ simplifies further to
\begin{align*}
    \mathcal{L}L(\mu) 
    &= \big\langle \beta \,\ell_\omega \,\ell_0 ,\mu \big\rangle \\
    &= \beta \ell_0 \, \big\langle \ell_\omega ,\mu \big\rangle \\
    &= \beta \ell_0 \, L(\mu) \,.
\end{align*}
This gives a mean total length of 
\begin{align*}
    \mathbb{E}[L_t] = L_0 e^{\beta \ell_0 t} \,.
\end{align*}
If instead the branching rate is $\Lambda(\mu,\omega) = \beta \ell_\omega L(\mu)^{-1}$ then we have a constant overall growth rate (since $\mathcal{L}L(\mu) =\beta \ell_0$) and linear, rather than exponential, expected growth
\begin{align*}
    \mathbb{E}[L_t] = L_0 + \beta \ell_0 t \,.
\end{align*}

In practical terms, in either case this indicates that when using a Brownian motion for the length process, the overall growth of the neuron is driven not by the growth of individual dendrites but by branching. This theoretical result matches the same observation following model fitting in \cite{shree2022dynamic}. This also shows that it is not necessary to include positive bias in the growth of individual branches to achieve overall growth, although naturally this would contribute to faster growth. 

It is interesting to consider tracking the length of the individual branches. Under Assumption \ref{Assupmtion: standard branching} we can attempt to derive a measure-valued branching process for the path lengths. From the full process $\mu_t$ define
\begin{align*}
    \eta_t = \sum_{(\omega,a) \in [\mu_t]} \delta\{ (\ell_\omega,a) \} \,,
\end{align*}
which is now a measure on $\mathbb{R}_{\geq0}\times \{0,1\}$ (or more precisely on $[0,\ell_{\max}]\times\{0,1\}$). Note that it is still necessary to include the activity variable as inactive branches have constant lengths and we must still avoid the branching-retraction problem illustrated in Figure \ref{fig: branching-retraction problem}. The generator structure of $\eta$ is very similar to that of $\mu$: it contains a `continuous' part corresponding to the length process of each active path and a `jump' part corresponding to the creation and removal of branches. 
\begin{align} \label{Eqn: Generator of eta}
    \mathcal{L}_\eta F(\eta) = \bigg\langle a f_1''(\ell) \mathds{1}\{\ell>0\} + &\beta \int_0^{\ell} \big( F(\eta^+) - F(\eta) \big) \, du\, \nonumber\\
    & + \lambda \mathds{1}\{\ell=0\} \big( F(\eta^-) - F(\eta) \big),\eta \bigg\rangle\,.
\end{align}
Similarly to the generator \eqref{Eqn: Full generator}, $\eta^+$ and $\eta^-$ are the measures resulting from creating and removing a branch respectively. 

Unfortunately, a problem arises in defining $\eta^-$. If a length process $\ell^i$ is absorbed at $\ell=0$ we must remove it from the measure $\eta_t$. The involves updating the activity and recombining certain branches connected to branch $i$. However, without the spatial information provided by $\omega^i$, we cannot determine which branches $i$ is connected to. Therefore, although the structure of \eqref{Eqn: Generator of eta} is very similar to that of \eqref{Eqn: Full generator}, $\eta_t$ is not itself a Markov process as it does not contain sufficient information to determine $\eta^-$. Even beyond the precise setup of this model, it is not possible to avoid this problem, as fixing the branching-retraction problem illustrated in Figure \ref{fig: branching-retraction problem} means that the spatial position or graph structure of branches must always be accounted for in some way. 

\section{Numerical simulation} \label{Section: numerical simulation}

In this section, we outline two methods for simulating the process constructed in Section \ref{Section: Construction}. The first is a standard time discretisation (fixing a temporal resolution), while the second is an approximate process that adds `segments' of a fixed length (fixing a spatial resolution). We do not show rigorously that these approximations converge to the original process, but do provide arguments to justify their use. 

\subsection{Time discretisation} \label{Section: Time discretisation}

The most straightforward method of simulation is to take a time discretisation of the process. Denote by $\Delta t$ the discretisation timestep. Clearly it is not possible to simulate a measure directly but in our case any data structure that can hold a changing number of objects (for example a dictionary) will suffice. It is also not possible to describe the full continuous path $\omega \in \mathcal{W}$, so it will be necessary to also discretise this path. We therefore track a collection of activity-labelled paths $(\Tilde{\omega}^i, a^i)$ where $\Tilde{\omega}^i$ is a vector of length $N^i$, storing a spatially discretised path and $a^i \in \{0,1\}$ is as before. Each element of $\Tilde{\omega}^i$ will have the form $\Tilde{\omega}^i_n = (x^i_n, \theta^i_n, \ell^i_n)$ where $x^i_n$ is the spatial position, $\theta^i_n$ stores any other tracked variable (for example growth angle or branch width) and $\ell^i_n$ is the length of the path up to that point. The length of $\Tilde{\omega}^i$ is therefore given by $\ell^i_{N^i}$. 

We distinguish between two time discretisation methods: first, a method for small $\Delta t$, so that branching rates may be assumed to be constant, that updates both length and growth processes; and second, a method for large $\Delta t$ in which we exploit the separation of different processes.

First, for small timesteps we use the following procedure to update, add and remove branches. This involves looping through the existing branches. As branch activity may change during this loop, a method such as Python's `deque' is useful to track which branches still need to be updated. 
\begin{enumerate}
    \item For each active branch, draw a random number $r \sim \sqrt{\Delta t} \,\mathcal{N}(0,1)$, where $\mathcal{N}(0,1)$ is a standard Gaussian random variable, to determine if this branch will grow or retract. This simulates the update to the length process which is a Brownian motion and so has normally distributed increments. 
    \begin{enumerate}
        \item If $r \leq -\ell_N$ then this branch will be fully retracted and should be deleted. By matching start- and end-points identify the branches connected to branch $i$ that need to be recombined (see Section \ref{Section: branching process}) and update path activity accordingly. 
        \item If $ - \ell_N < r<0$ then we shorten the branch $i$. Remove any points in $\Tilde{\omega}$ with $\ell_n > \ell_N + r$ then interpolate to update the length of the final section so that the new discretised path has length $\ell_N + r$. 
        \item In the unlikely event that $r=0$, do nothing. 
        \item If $r>0$ then add a new point to $\Tilde{\omega}$ using the growth process (either by simulating the process up to time $r$ starting from $x_N$ or using the law of $\Pi_{x_N}^r$ if this is known). This new element has $\ell_{N+1}=\ell_N + r$. 
    \end{enumerate}
    \item For each branch (regardless of its activity), calculate its branching rate $\Lambda(\Tilde{\omega})$. As branching is a jump process with exponentially distributed holding times, add a new branch to this one with probability $1 - \exp(-\Lambda(\Tilde{\omega}) \Delta t)$. If a new branch is to be added:
    \begin{enumerate}
        \item There are two options to determine the new branch location: use the full distribution $\rho_\omega$ to pick the branching location and interpolate to add a new element at this value; or choose a branching location from the existing elements of $\Tilde{\omega}$. In either case a branch should \textbf{not} be added at the start or end elements of $\Tilde{\omega}$. 
        \item Create the new branch from the chosen start point up to the desired length (called $\ell_0$ in Section \ref{Section: New branches}) using the growth process. Add this new branch to the list of branches to be updated in Step 1. 
        \item Split the branch $\Tilde{\omega}$ and update path activity according to the rules described in Section \ref{Section: branching process}.
    \end{enumerate}
\end{enumerate}

In Step 1 we used the matching of start- and end-points to determine branch connections. While this is faithful to the model construction, it may be computationally more straightforward to store this information separately. This could be stored as a list of parent-child pairs or as a tree network, mirroring directly the structure of the graph representation $G(\mu_t)$. 

This method assumes that the branching rate $\Lambda(\omega)$ is effectively constant within each small timestep and any new branches will not branch again themselves during this time period. If the branching rate is slow and changes continuously with $\omega$ then these assumptions are reasonable. In this method, the time discretisation $\Delta t$ also determines the spatial resolution of the paths $\Tilde{\omega}$ as one new point is added in each timestep. To obtain higher resolution, add multiple new points if $r$ is large.

We next discuss an adaptation of this method for longer timesteps under Assumption \ref{Assupmtion: standard branching} (constant branching rate per unit length) by exploiting the separation of length and growth processes. The idea is to avoid simulating the growth process between timesteps, unless it is needed to implement branching events. This method requires a short timestep $\delta t \ll \Delta t$ on which we will simulate the length process and update the branching rates. In addition, it is necessary in this method to track the parent-child pairs separately, although as discussed above this is already computationally more efficient. 
\begin{enumerate}
    \item{} [Initial setup] At the beginning of each timestep record the length of each branch $\ell^i$. Also set the tracker $\ell^i_{\min} := \ell^i$. These will be treated as variables in their own right in this method. The value of $\ell^i_{\min}$ is tracked so that the branch can be shortened to this length (as in Figure \ref{fig: Brownian snake}). Set the time tracker $t:=0$. 
    
    \item{} [Drawing event times] For each branch draw an exponentially distributed random number $\tau^i$ and set a tracker $s^i:=0$. This will be used to determine the timing of the next branching event along $i$, using a time change to simulate the non-homogeneous jump rate \cite{resnick2013adventures}. In each small timestep the tracker $s^i$ will be updated by adding the current branching rate. When $s^i$ exceeds $\tau^i$ a branching event occurs.
    
    \item{} [Length evolution between events] In each small timestep $(\delta t)$, update the length process for each active branch $\ell^i \leftarrow \ell^i + r$, where $r \sim \sqrt{\delta t} \,\mathcal{N}(0,1)$ are iid for each branch and each timestep. At each small timestep update the time tracker $t \leftarrow t + \delta t$ and minimum length trackers $\ell^i_{\min} \leftarrow \min(\ell^i_{\min}, \ell^i + r)$. We do not update the path itself on each small timestep. 

    For all branches (active or inactive), update $s^i \leftarrow s^i + \beta \ell^i \,\delta t$. This adds the current branching rate to the tracker $s^i$ (see Step 2 above). 
    
    At each small timestep check the following:
    \begin{itemize}
        \item If $t \geq \Delta t$ go to Step 4, then Stop. 
        \item If any $s^i \geq \tau^i$ go to Step 4, then Step 5 to implement a branching event on $i$. 
        \item If any length process hits $\ell^i \leq 0$ then pause and remove this process. Use the parent-child relationships to find the parent and sibling of the removed branch and add their length processes. Update the parent-child relationships accordingly, then return to Step 2.
    \end{itemize}
    \item{} [Spatial simulation] For each path remove points from $\Tilde{\omega}^i$ and interpolate so that its length is reduced to $\ell^i_{\min}$. If a path was inactive for the entire period, then $\ell^i_{\min}$ will be equal to its initial length and no update is required. Then use the growth process to sequentially extend each path until it has length $\ell^i$. When extending the path in this way, the spatial discretisation can be chosen independently of $\Delta t$ and $\delta t$ (i.e. an arbitrarily fine discretisation may be used depending on the desired spatial resolution of paths). This mirrors the process illustrated in Figure \ref{fig: Brownian snake}. Note that $\ell^i_{\min}$ is the minimum within the current $\Delta t$ timestep.
    \item{} [Branching event] To implement a branching event on branch $i$, first draw the branch position $u \sim \textit{Unif}(0,\ell^i)$ (although this could be replaced with a more general distribution if desired). Interpolate the discretised path $\omega^i$ to find $\omega^i(u)$ and add a point here. Use the growth process to generate a new path of length $\ell_0$ starting from $\omega^i(u)$. Split $\omega^i$ at $\omega^i(u)$ and update path activities according to Section \ref{Section: New branches}, then update the parent-child relationships accordingly. Then return to Step 2.
\end{enumerate}

The main advantage of this method is that it avoids unnecessarily simulating the spatial position of branches (Step 4) but does rely more heavily on the independence of different processes. This method is beneficial when the branching rate is relatively slow and simulations should have low temporal but high spatial resolution. This method also implements a non-homogeneous Poisson process for branching, assuming that the branching rate is constant over the small timestep $(\delta t)$ but allowing it to vary over the long timestep $(\Delta t)$. 

\begin{remark}
    In both methods, it is necessary at various points to interpolate between points along the discretised path. In the case that the transition probabilities of the growth process are known, the Chapman-Kolmogorov equations can be used to derive the probability density function of this intermediate point, rather than using interpolation. Assume that we must select a point between $(x_0,\ell_0)$ and $(x_1,\ell_1)$ such that the new path length is $\ell_0 + l$. Then the new point, denoted $(y,\ell_0+l)$, has distribution
    \begin{align*}
        \mathbb{P}(dy) &= \mathbb{P}[\xi_l \in dy \, | \, \xi_0 = x_0, \xi_{\ell_{1}-\ell_0} = x_{1} ] \\
        &= P_l(x_0, y) P_{\ell_{1}-\ell_0 - l}(y,x_{1}) \, dy \,.
    \end{align*}
    This allows for a more accurate stochastic interpolation in the case that $P$ is known. 
\end{remark}

\subsubsection{Example simulations} \label{Section: Example simulations discrete time}

\begin{figure}[ht!]
    \centering
    \includegraphics[width=0.9\linewidth]{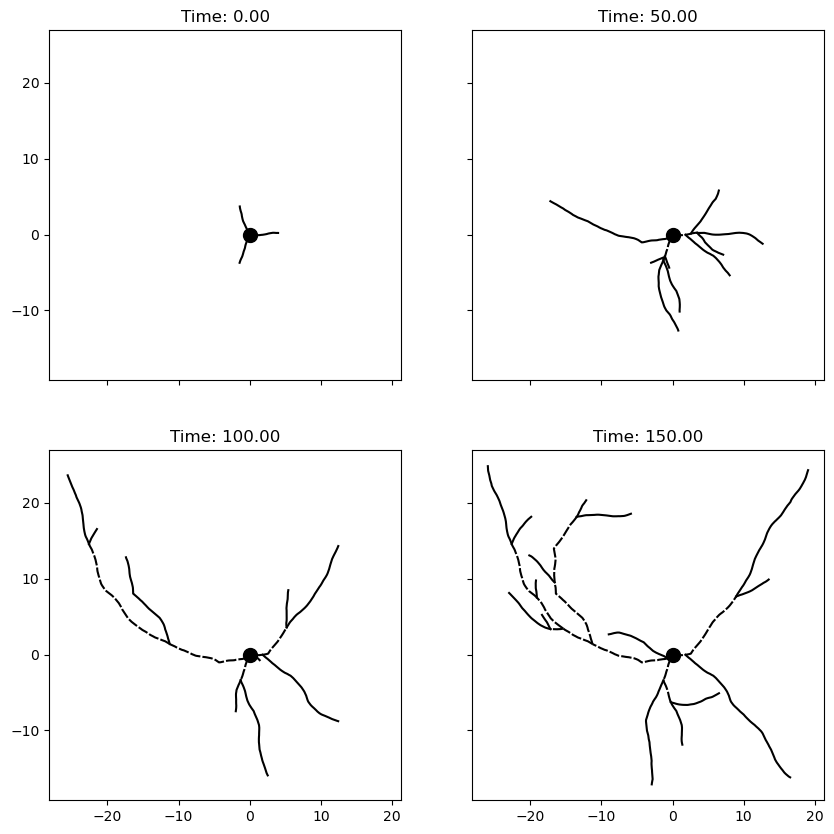}
    \caption{Example simulation using the time discretisation method for a small timestep. Full details of the simulation setup can be found in Section \ref{Section: Example simulations discrete time}. Solid lines show active branches while dashed lines show inactive branches. A small circle represents the soma. The following parameter values were used, $\Delta t = 0.1$, $\beta = 0.05$, $\varsigma = 0.5$, $b = 0.25$.}
    \label{fig: Example simulations discrete time}
\end{figure}

An example simulation for a short timestep $\Delta t=0.1$ is shown in Figure \ref{fig: Example simulations discrete time} below. The code used to generate this can be found at \url{github.com/AJNugent2/Dendrite-simulations}. The code directly simulates the graph representation $G(\mu_t)$ by tracking a tree network and adding information on the spatial position of the dendrite ($\Tilde{\omega}^i$) to the tree edges. The processes used were:
\begin{itemize}
    \item A length process given by 
    \begin{align*}
        d\ell_t = b \, dt + \sigma \, d\beta_t \,,
    \end{align*}
    as in the $b$-snake (see Section \ref{Section: Extensions}). Here the drift term $b$ was taken to be a constant bias and $\sigma$ is some positive constant. The method above is adapted to this simply by taking
    \begin{align*}
        r \sim \sqrt{\Delta t} \, \mathcal{N}(0,1) + b \Delta t \,.
    \end{align*}
    \item The growth processes was inspired by a rotational diffusion with $x^i_n = (x,y)$ (two spatial dimensions) and $\theta^i_n \in [0,2\pi)$ representing the growth angle. These evolve according to 
    \begin{align*}
        dx_t &= \cos(\theta_t) \, dt \,,\\
        dy_t &= \sin(\theta_t) \, dt \,,\\
        d\theta_t &= \varsigma \, d\beta_t \,.
    \end{align*}
    Again, $\varsigma$ is some positive constant controlling the rate at which the growth angle changes. 
    \item The branching rate was taken to be 
    \begin{align*}
        \Lambda(\mu,\omega) = c \, \ell_\omega / A_t \,,
    \end{align*}
    where $A_t$ is the current number of active branches and $c$ is a constant. 
\end{itemize}
These processes are meant as an illustrative example of what could be included in the model. It is not claimed that these represent accurate dendrite dynamics. In addition, the parameter values chosen are not fitted to data. Further work is required to estimate appropriate processes and parameters from neuron imaging data. The code available at the link above also indicates what elements should be changed to incorporate different length processes, growth processes or branching rates. 

The neuron is initialised with three short dendrites which grow and branch over time, as can be seen in Figure \ref{fig: Example simulations discrete time}. Solid lines show active branches while dashed lines show inactive branches. A small circle represents the soma. 

\subsection{Segment discretisation} \label{Section: Segment discretisation}

Rather than considering a time discretisation of the process constructed in Section \ref{Section: Construction}, we may instead simulate using an approximation of the process by adding or removing short dendrite `segments' of a fixed length $0<\varepsilon\ll 1$. This is related to the models presented in \cite{graham2001compartmental,hannezo2017unifying,hjorth2014competitive,duswald2024calibration} which also add discrete sections, as well as to cellular automata models that make stochastic updates such as \cite{hannezo2017unifying,hentschel1996diffusion,luczak2006spatial}. This approximates the process continuously in time but discretely in space, effectively setting a spatial resolution rather than a temporal one. By appropriately scaling the rate at which segments are added/removed we may recover both the Brownian snake behaviour of individual dendrites and the addition and removal of branches. Throughout this section we again make Assumption \ref{Assupmtion: standard branching}.

The segment discretisation uses a continuous time jump process whose state is a tree $G_t$, whose root $V_{root}$ corresponds to the soma. As in the description of the time discretisation, we consider only a single dendritic tree, so that the root of $G_t$ has out-degree 1 (for example, see the left panel of Figure \ref{fig: segment mapping}). This more directly reflects the tree-like structure of a neuron's dendrites; rather than reconstructing this tree by combining multiple branches we track it directly. The fixed length scale provides a consistent spatial resolution across the neuron. 

Let $(V_t,E_t)$ be the nodes and edges of $G_t$. Each node $v \in V_t$ is associated to a single point $x_v \in X$ that does not change in time. Denote by $V_t^L$ the leaf nodes of $G_t$ (that is, those with out-degree zero) and by $V_t^B$ the branch points of $G_t$ (that is, those with out-degree two). We call all other nodes (those that are not the root node, leaf nodes or branch points) internal nodes and denote by $V^I_t$ the set of internal nodes. 

There are three possible updates to $G_t$. Figure \ref{fig:segment discretisation diagram} shows an example of these, with the current state in Panel 1. Green nodes indicate leaf nodes.  
\begin{itemize}
    \item Growth (Panel 2a): a new node $v'$ is added as the child of a leaf node $v$. Its associated element of $X$ is selected randomly with law equal to that of $\{\xi_\varepsilon \, | \, \xi_0 = x_v \}$. Denote the resulting tree by $G^+\{v'\}$. This event occurs at rate $\alpha_\varepsilon/2$, independently for each leaf node. 
    \item Retraction (Panel 2b): a leaf node $v$, and the edge leading into it, are removed from $G_t$. Denote the resulting tree by $G^-\{v\}$. This event occurs at rate $\alpha_\varepsilon/2$, independently for each leaf node. 
    \item Branching (Panel 2c): from an internal node, a new branch is added. This branch is formed of a sequence of $n_\varepsilon$ nodes $\{v_1,\dots,v_{n_\varepsilon}\}$ connecting to each other in a line. Denote by $x_n$ the element of $X$ associated to $v_n$. Each $x_n$ is selected at random with law equal to that of $\{\xi_\varepsilon \, | \, \xi_0 = x_{n-1} \}$. Denote the resulting tree by $G^+\{v_1,\dots,v_{n_\varepsilon}\}$. This event occurs at rate $\beta_\varepsilon$, independently for each internal node. 
\end{itemize}

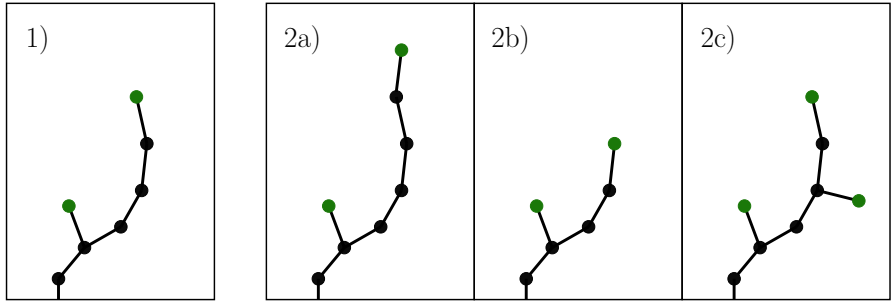
\begin{figure}[!ht]
\centering
\resizebox{.9\textwidth}{!}{
\begin{circuitikz}
\tikzstyle{every node}=[font=\fontsize{18.2pt}{23.7pt}\selectfont]
\draw [ line width=1pt ] (2.5,14) rectangle (7.5,6.875);
\draw [ line width=1pt ] (8.75,14) rectangle (13.75,6.875);
\draw [ line width=1pt ] (13.75,14) rectangle (18.75,6.875);
\draw [ line width=1pt ] (18.75,14) rectangle (23.75,6.875);
\draw [ fill={rgb,255:red,11; green,11; blue,11}, fill opacity=1] (4.375,8.125) circle (0.15cm);
\draw [ fill={rgb,255:red,11; green,11; blue,11}, fill opacity=1] (5.25,8.625) circle (0.15cm);
\draw [ fill={rgb,255:red,11; green,11; blue,11}, fill opacity=1] (5.75,9.5) circle (0.15cm);
\draw [ fill={rgb,255:red,11; green,11; blue,11}, fill opacity=1] (5.875,10.625) circle (0.15cm);
\draw [line width=2pt, short] (5.625,11.75) -- (5.875,10.625);
\draw [line width=2pt, short] (5.875,10.625) -- (5.75,9.5);
\draw [line width=2pt, short] (5.75,9.5) -- (5.25,8.625);
\draw [line width=2pt, short] (5.25,8.625) -- (4.375,8.125);
\draw [ fill={rgb,255:red,11; green,11; blue,11}, fill opacity=1] (10.625,8.125) circle (0.15cm);
\draw [ fill={rgb,255:red,11; green,11; blue,11}, fill opacity=1] (11.5,8.625) circle (0.15cm);
\draw [ fill={rgb,255:red,11; green,11; blue,11}, fill opacity=1] (12,9.5) circle (0.15cm);
\draw [ fill={rgb,255:red,11; green,11; blue,11}, fill opacity=1] (12.125,10.625) circle (0.15cm);
\draw [ fill={rgb,255:red,11; green,11; blue,11}, fill opacity=1] (11.875,11.75) circle (0.15cm);
\draw [line width=2pt, short] (12,12.875) -- (11.875,11.875);
\draw [line width=2pt, short] (11.875,11.75) -- (12.125,10.625);
\draw [line width=2pt, short] (12.125,10.625) -- (12,9.5);
\draw [line width=2pt, short] (12,9.5) -- (11.5,8.625);
\draw [line width=2pt, short] (11.5,8.625) -- (10.625,8.125);
\draw [ fill={rgb,255:red,11; green,11; blue,11}, fill opacity=1] (15.625,8.125) circle (0.15cm);
\draw [ fill={rgb,255:red,11; green,11; blue,11}, fill opacity=1] (16.5,8.625) circle (0.15cm);
\draw [ fill={rgb,255:red,11; green,11; blue,11}, fill opacity=1] (17,9.5) circle (0.15cm);
\draw [line width=2pt, short] (17.125,10.625) -- (17,9.5);
\draw [line width=2pt, short] (17,9.5) -- (16.5,8.625);
\draw [line width=2pt, short] (16.5,8.625) -- (15.625,8.125);
\draw [ fill={rgb,255:red,11; green,11; blue,11}, fill opacity=1] (20.625,8.125) circle (0.15cm);
\draw [ fill={rgb,255:red,11; green,11; blue,11}, fill opacity=1] (21.5,8.625) circle (0.15cm);
\draw [ fill={rgb,255:red,11; green,11; blue,11}, fill opacity=1] (22,9.5) circle (0.15cm);
\draw [ fill={rgb,255:red,11; green,11; blue,11}, fill opacity=1] (22.125,10.625) circle (0.15cm);
\draw [line width=2pt, short] (21.875,11.75) -- (22.125,10.625);
\draw [line width=2pt, short] (22.125,10.625) -- (22,9.5);
\draw [line width=2pt, short] (22,9.5) -- (21.5,8.625);
\draw [line width=2pt, short] (21.5,8.625) -- (20.625,8.125);
\draw [line width=2pt, short] (4,9.125) -- (4.375,8.125);
\draw [line width=2pt, short] (23,9.25) -- (22,9.5);
\draw [ color={rgb,255:red,26; green,120; blue,9}, draw opacity=1 , fill={rgb,255:red,26; green,120; blue,9}, fill opacity=1] (23,9.25) circle (0.15cm);
\draw [ fill={rgb,255:red,11; green,11; blue,11}, fill opacity=1] (3.75,7.375) circle (0.15cm);
\draw [line width=2pt, short] (4.375,8.125) -- (3.75,7.375);
\draw [line width=2pt, short] (3.75,7.375) -- (3.75,6.875);
\draw [line width=2pt, short] (10.625,8.125) -- (10,7.375);
\draw [ fill={rgb,255:red,11; green,11; blue,11}, fill opacity=1] (10,7.375) circle (0.15cm);
\draw [line width=2pt, short] (10,7.375) -- (10,6.875);
\draw [line width=2pt, short] (15.625,8.125) -- (15,7.375);
\draw [ fill={rgb,255:red,11; green,11; blue,11}, fill opacity=1] (15,7.375) circle (0.15cm);
\draw [line width=2pt, short] (15,7.375) -- (15,6.875);
\draw [line width=2pt, short] (20.625,8.125) -- (20,7.375);
\draw [ fill={rgb,255:red,11; green,11; blue,11}, fill opacity=1] (20,7.375) circle (0.15cm);
\draw [line width=2pt, short] (20,7.375) -- (20,6.875);
\draw [ color={rgb,255:red,26; green,120; blue,9}, draw opacity=1 , fill={rgb,255:red,26; green,120; blue,9}, fill opacity=1] (5.625,11.75) circle (0.15cm);
\draw [ color={rgb,255:red,26; green,120; blue,9}, draw opacity=1 , fill={rgb,255:red,26; green,120; blue,9}, fill opacity=1] (4,9.125) circle (0.15cm);
\draw [ color={rgb,255:red,26; green,120; blue,9}, draw opacity=1 , fill={rgb,255:red,26; green,120; blue,9}, fill opacity=1] (12,12.875) circle (0.15cm);
\draw [ color={rgb,255:red,26; green,120; blue,9}, draw opacity=1 , fill={rgb,255:red,26; green,120; blue,9}, fill opacity=1] (17.125,10.625) circle (0.15cm);
\draw [ color={rgb,255:red,26; green,120; blue,9}, draw opacity=1 , fill={rgb,255:red,26; green,120; blue,9}, fill opacity=1] (21.875,11.75) circle (0.15cm);
\node [font=\fontsize{19.9pt}{25.9pt}\selectfont, inner xsep=0.080cm, inner ysep=0.085cm, rounded corners=0.020cm] at (3.25,13.125) {1)};
\node [font=\fontsize{19.9pt}{25.9pt}\selectfont, inner xsep=0.080cm, inner ysep=0.085cm, rounded corners=0.020cm] at (9.625,13.125) {2a)};
\node [font=\fontsize{19.9pt}{25.9pt}\selectfont, inner xsep=0.080cm, inner ysep=0.085cm, rounded corners=0.020cm] at (14.625,13.125) {2b)};
\node [font=\fontsize{19.9pt}{25.9pt}\selectfont, inner xsep=0.080cm, inner ysep=0.085cm, rounded corners=0.020cm] at (19.625,13.125) {2c)};
\draw [line width=2pt, short] (10.25,9.125) -- (10.625,8.125);
\draw [ color={rgb,255:red,26; green,120; blue,9}, draw opacity=1 , fill={rgb,255:red,26; green,120; blue,9}, fill opacity=1] (10.25,9.125) circle (0.15cm);
\draw [line width=2pt, short] (15.25,9.125) -- (15.625,8.125);
\draw [ color={rgb,255:red,26; green,120; blue,9}, draw opacity=1 , fill={rgb,255:red,26; green,120; blue,9}, fill opacity=1] (15.25,9.125) circle (0.15cm);
\draw [line width=2pt, short] (20.25,9.125) -- (20.625,8.125);
\draw [ color={rgb,255:red,26; green,120; blue,9}, draw opacity=1 , fill={rgb,255:red,26; green,120; blue,9}, fill opacity=1] (20.25,9.125) circle (0.15cm);
\end{circuitikz}
}
\caption{Example of the three possible updates in the segment discretisation. Green nodes indicate leaf nodes. 1) The current state. 2) A possible growth update. 3) A possible retraction update. 4) A possible branching update.}
\label{fig:segment discretisation diagram}
\end{figure}

Simulating the neuron as a tree in this way also makes growth and branching essentially the same: both add a new node using the law of the growth process, the only difference is that growth events are from leaf nodes and branching events are from internal nodes. This is in contrast to the time discretisation in which growth corresponds directly to the length process and branching events must be added separately. 

The dependence of the event rates ($\alpha_\varepsilon$ and $\beta_\varepsilon$) on $\varepsilon$ will be chosen so that as $\varepsilon\rightarrow0$ the original process is recovered. The number of nodes added for each new branch $n_\varepsilon$ is chosen as 
\begin{align*}
    n_\varepsilon = \lceil \ell_0/\varepsilon \rceil \,,
\end{align*}
as this is at least $1$ and means the resulting branch has length $\varepsilon n _\varepsilon=\varepsilon \lceil \ell_0/\varepsilon \rceil \rightarrow \ell_0$ as $\varepsilon\rightarrow0$. In practice in simulations, as $\ell_0$ and $\varepsilon$ should both be small, we may simply take $n_\varepsilon = 1$. 

Denote by $\Pi^x_\varepsilon(dx)$ the probability distribution of $\{\xi_\varepsilon \, | \, \xi_0 = x \}$. Then the process $G_t$ is a jump process with generator
\begin{align} \label{Eqn: Segment discretisation generator}
    \mathcal{L}_G F(G) &= \sum_{v \in V_L} \frac{\alpha_\varepsilon}{2} \Big( F(G^-\{v\}) - F(G) \Big) \nonumber\\
    &+ \sum_{v \in V_L} \frac{\alpha_\varepsilon}{2} \int \Big( F(G^+\{v'\}) - F(G) \Big) \,\Pi^{x_v}_\varepsilon(dx) \, \nonumber\\
    &+ \sum_{v \in V_I} \beta_\varepsilon \int \Big( F(G^+\{v_1,\dots,v_{n_\varepsilon}\}) - F(G) \Big) \,\Pi^{x_v}_\varepsilon(dx_1) \dots \,\Pi^{x_{n_\varepsilon-1}}_\varepsilon(dx_{n_\varepsilon}) 
\end{align}
where $F$ is any bounded function on the space of rooted trees. As this process is a jump process, it can be simulated exactly using the Gillespie algorithm \cite{gillespie1977exact}. As there is no need to track activity and parent-child relationships this simulation method is much more straightforward. However, when $\varepsilon$ is small the event rates become faster and so simulations with the Gillespie algorithm become slower. 

\begin{figure}[!ht]
\centering
\resizebox{.6\textwidth}{!}{
\begin{circuitikz}
\tikzstyle{every node}=[font=\fontsize{27.9pt}{36.4pt}\selectfont]
\draw [ color={rgb,255:red,0; green,0; blue,0}, draw opacity=1, line width=4pt, dash pattern=on 3\pgflinewidth off 1.5\pgflinewidth] (9.25,7.125) -- (11.125,9.375);
\draw [ color={rgb,255:red,0; green,0; blue,0}, draw opacity=1, line width=4pt, dash pattern=on 3\pgflinewidth off 1.5\pgflinewidth] (11.125,9.375) -- (12.125,12.25);
\draw [ color={rgb,255:red,0; green,0; blue,0}, draw opacity=1, line width=4pt, dash pattern=on 3\pgflinewidth off 1.5\pgflinewidth] (12.125,12.25) -- (14.5,14);
\draw [line width=4pt, short] (14.5,14) -- (16.125,16.125);
\draw [line width=4pt, short] (16.125,16.125) -- (18.875,17.625);
\draw [line width=4pt, short] (18.875,17.625) -- (20,20);
\draw [line width=4pt, short] (14.5,14) -- (17.25,13.375);
\draw [line width=4pt, short] (17.25,13.375) -- (20.25,14.125);
\draw [line width=4pt, short] (11.125,9.375) -- (9.25,12);
\draw [line width=4pt, short] (9.25,12) -- (9.625,14.875);
\draw [line width=4pt, short] (9.625,14.875) -- (8.75,17.375);
\draw [ color={rgb,255:red,0; green,0; blue,0}, draw opacity=1, line width=4pt, dash pattern=on 3\pgflinewidth off 1.5\pgflinewidth] (9.25,7.125) -- (8.5,4.75);
\draw [ color={rgb,255:red,0; green,0; blue,0}, draw opacity=1, line width=4pt, dash pattern=on 3\pgflinewidth off 1.5\pgflinewidth] (8.5,4.75) -- (8.75,2.125);
\draw [ fill={rgb,255:red,0; green,0; blue,0}, fill opacity=1, line width=4pt ] (11.125,9.375) circle (0.4cm);
\draw [ fill={rgb,255:red,0; green,0; blue,0}, fill opacity=1, line width=4pt ] (8.75,17.5) circle (0.4cm);
\draw [ fill={rgb,255:red,0; green,0; blue,0}, fill opacity=1, line width=4pt ] (14.5,14) circle (0.4cm);
\draw [ fill={rgb,255:red,0; green,0; blue,0}, fill opacity=1, line width=4pt ] (19.875,19.875) circle (0.4cm);
\draw [ fill={rgb,255:red,0; green,0; blue,0}, fill opacity=1, line width=4pt ] (20.25,14.125) circle (0.4cm);
\draw [line width=4pt, short] (-9.5,7.125) -- (-7.625,9.375);
\draw [line width=4pt, short] (-7.625,9.375) -- (-6.625,12.25);
\draw [line width=4pt, short] (-6.625,12.25) -- (-4.25,14);
\draw [line width=4pt, short] (-4.25,14) -- (-2.625,16.125);
\draw [line width=4pt, short] (-2.625,16.125) -- (0.125,17.625);
\draw [line width=4pt, short] (0.125,17.625) -- (1.25,20);
\draw [line width=4pt, short] (-4.25,14) -- (-1.5,13.375);
\draw [line width=4pt, short] (-1.5,13.375) -- (1.5,14.125);
\draw [line width=4pt, short] (-7.625,9.375) -- (-9.5,12);
\draw [line width=4pt, short] (-9.5,12) -- (-9.125,14.875);
\draw [line width=4pt, short] (-9.125,14.875) -- (-10,17.375);
\draw [line width=4pt, short] (-9.5,7.125) -- (-10.25,4.75);
\draw [line width=4pt, short] (-10.25,4.75) -- (-10,2.125);
\draw [ fill={rgb,255:red,0; green,0; blue,0}, fill opacity=1, line width=4pt ] (-10,2) circle (0.75cm);
\draw [ fill={rgb,255:red,0; green,0; blue,0}, fill opacity=1, line width=4pt ] (-10.25,4.625) circle (0.4cm);
\draw [ fill={rgb,255:red,0; green,0; blue,0}, fill opacity=1, line width=4pt ] (-7.625,9.375) circle (0.4cm);
\draw [ fill={rgb,255:red,0; green,0; blue,0}, fill opacity=1, line width=4pt ] (-9.5,7.125) circle (0.4cm);
\draw [ fill={rgb,255:red,0; green,0; blue,0}, fill opacity=1, line width=4pt ] (-9.5,12) circle (0.4cm);
\draw [ fill={rgb,255:red,0; green,0; blue,0}, fill opacity=1, line width=4pt ] (-9.125,15) circle (0.4cm);
\draw [ fill={rgb,255:red,0; green,0; blue,0}, fill opacity=1, line width=4pt ] (-10,17.5) circle (0.4cm);
\draw [ fill={rgb,255:red,0; green,0; blue,0}, fill opacity=1, line width=4pt ] (-6.625,12.25) circle (0.4cm);
\draw [ fill={rgb,255:red,0; green,0; blue,0}, fill opacity=1, line width=4pt ] (-4.25,14) circle (0.4cm);
\draw [ fill={rgb,255:red,0; green,0; blue,0}, fill opacity=1, line width=4pt ] (-2.625,16.125) circle (0.4cm);
\draw [ fill={rgb,255:red,0; green,0; blue,0}, fill opacity=1, line width=4pt ] (-1.5,13.375) circle (0.4cm);
\draw [ fill={rgb,255:red,0; green,0; blue,0}, fill opacity=1, line width=4pt ] (0.125,17.625) circle (0.4cm);
\draw [ fill={rgb,255:red,0; green,0; blue,0}, fill opacity=1, line width=4pt ] (1.125,19.875) circle (0.4cm);
\draw [ fill={rgb,255:red,0; green,0; blue,0}, fill opacity=1, line width=4pt ] (1.5,14.125) circle (0.4cm);
\draw [ fill={rgb,255:red,0; green,0; blue,0}, fill opacity=1, line width=4pt ] (8.75,2) circle (0.75cm);
\draw [ color={rgb,255:red,23; green,23; blue,23}, draw opacity=1, line width=4pt, -{Stealth[scale=1.5]}, ] (1.375,9.375) -- (6.375,9.375);
\end{circuitikz}
}
\caption{Diagram showing the mapping of $G_t$ to $\Tilde{G}$ by splitting the network into active paths (solid lines) and inactive paths (dashed lines). In this case there are three active and two inactive paths in $\Tilde{G}$. In both graphs the soma (also the root node) is shown by a larger node.}
\label{fig: segment mapping}
\end{figure}
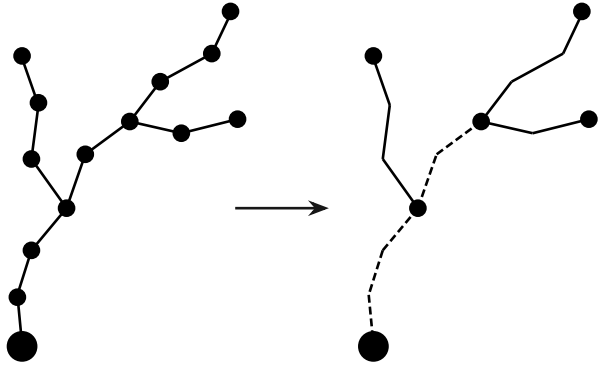

To compare the rooted tree $G_t$ against the measure-valued process $\mu_t$ constructed in Section \ref{Section: Construction}, we introduce a mapping from $G_t$ to the space of measures $\mathcal{N}_\mathcal{S}^\#$. We begin by constructing a new network $\Tilde{G}=(\Tilde{V},\Tilde{E})$ that mirrors the graph representation in Section \ref{Section: Graph properties}. This may be most easily understood through the visualisation shown in Figure \ref{fig: segment mapping}. The nodes of $\Tilde{G}$ are the nodes of $G_t$ minus the internal nodes (i.e. the root, branching and leaf nodes). The edges of $\Tilde{G}$ correspond to the paths in $G_t$ between these nodes. Taking the $x$-values along this path, we assign an element $\omega^e \in \W$ to the edge $e \in \Tilde{E}$, such that $\omega^e$ passes through these $x$-values. Mark these as active ($a^e=1$) if the edge $e$ ends in a leaf node of $G_t$ and inactive otherwise. Active edges are shown in solid lines in Figure \ref{fig: segment mapping} and inactive edges are shown in dashed lines. Finally, define the measure
\begin{align*}
    \mu(G_t) := \sum_{e \in \Tilde{E}} \delta\{(\omega^e, a^e)\} \,.
\end{align*}
This gives a mapping from the segment-graph $G_t$ to a measure that has the same structure and spatial embedding. In particular, the graph representation of $\mu(G_t)$ is exactly $\Tilde{G}$.

This construction allows us to compare $\mu_t$ against $\mu(G_t)$, and so compare the original process against the segment discretisation. For example, we may now meaningfully talk about the length process $\ell^i_t$ corresponding to a branch $\omega^i_t$ in $\mu(G_t)$. Below we outline an argument that the segment discretisation converges to the original process as $\varepsilon\rightarrow0$. 

For an active path in $\mu(G_t)$, the length $\ell_t$ increases/decreases in jumps of size $\varepsilon>0$ as segments are added/removed. This has generator 
\begin{align}
    \mathcal{L}_G f(\ell) 
    &= \frac{\alpha_\varepsilon}{2} \Big( f(\ell+\varepsilon) - f(\ell) \Big) + \frac{\alpha_\varepsilon}{2} \Big( f(\ell-\varepsilon) - f(\ell) \Big) \nonumber \\
    &\approx \frac{\alpha_\varepsilon}{2} \Big( \varepsilon f'(\ell) + \frac{1}{2} \varepsilon^2 f''(\ell) \Big) + \frac{\alpha_\varepsilon}{2} \Big( - \varepsilon f'(\ell) + \frac{1}{2} \varepsilon^2 f''(\ell) \Big)  \label{Eqn: Generator of jump length process approx} \\
    &\approx \frac{\alpha_\varepsilon}{2} \varepsilon^2 f''(\ell) \,. \nonumber
\end{align}
Assume that $\alpha_\varepsilon = \frac{\alpha}{\varepsilon^2}$ which gives 
\begin{align*}
    \mathcal{L}_G f(\ell)\rightarrow \frac{\alpha}{2} f''(\ell) \,,
\end{align*}
as $\varepsilon\rightarrow0$, which is exactly the generator of Brownian motion. If the leaf node is only connected to its branching point, removing it fully retracts this branch. This corresponds to an absorbing boundary at $\ell=0$. Therefore as $\varepsilon\rightarrow0$ we exactly recover the correct dynamics of the length process. 

Consider next the embedding in $X$. Recall that the edges $\omega^e$ are constructed to pass through a sequence of points $\{x_0,\dots,x_n\}$ with each $x_{i+1}$ chosen according to the law $\Pi_\varepsilon^{x_i}$. As $\varepsilon\rightarrow0$ these points become dense and, as $\xi$ is a Feller process, $\omega^e$ will converge to a path of $\xi$. Therefore the convergence of the length process $\ell_t$ to a Brownian motion implies the convergence of $\omega^e$ to a Brownian snake for each active branch.

We next consider the addition and removal of branches. Each internal node in $G_t$ gives rise to new branches at a rate $\beta_\varepsilon$, therefore the corresponding path $\omega \in \mu(G_t)$ has a branch rate of $\beta_\varepsilon$ multiplied by the number of internal nodes. Since each connection between internal nodes in $G_t$ contributes a path of length $\varepsilon$, $\ell_\omega = (\text{\# nodes} - 1) \times \varepsilon$, hence the number of nodes can be calculated as $1 + \ell_\omega \varepsilon^{-1}$ and the number of internal nodes is $\ell_\omega \varepsilon^{-1} - 1$. The branching rate for $\omega$ is therefore $\beta_\varepsilon (\ell_\omega \varepsilon^{-1} - 1)$. Choosing $\beta_\varepsilon = \beta\varepsilon$, this is equal to $\beta (\ell_\omega - \varepsilon) \rightarrow \beta \ell_\omega$ as required. The distribution of branching locations is limited to the location of internal nodes but is uniform amongst them. As $\varepsilon\rightarrow0$ this will naturally converge to the uniform distribution on $(0,\ell_\omega)$ as required. Therefore the branching mechanism of $\mu(G_t)$ also converges to that of the original process.

\begin{remark}
    The extension and removal of leaf nodes occurs at a rate $\alpha_\varepsilon = \alpha \varepsilon^{-2}$ and so becomes extremely fast as $\varepsilon\rightarrow0$, to properly capture the behaviour of Brownian motion. Meanwhile the branch rate $\beta_\varepsilon = \beta \varepsilon$ becomes slow as $\varepsilon\rightarrow0$ so that branch events remain sufficiently rare. 
\end{remark}

In this case there is no need to explicitly include the removal of a branch with $\ell=0$ from the measure $\mu(G_t)$, as a branch having no remaining nodes automatically removes it.  

Finally, note that property d) of Theorem \ref{Thm: Inherited graph properties} means the activity rule for $\mu(G_t)$ (only paths ending in leaf nodes are active) is consistent with the setup described in Section \ref{Section: branching process}. 

Therefore, if the rates $\alpha_\varepsilon$ and $\beta_\varepsilon$ are scaled appropriately, this segment discretisation will approximate the original process as $\varepsilon\rightarrow0$. On the one hand this seems perfectly natural - we are still modelling the three processes of growth, retraction and branching. On the other hand it is somewhat remarkable that this simple addition and removal of segments suitably approximates a much more complex process tracking and creating multiple independent activity-labelled branches. 

\begin{remark} \label{Remark: on bias in segments}
    It is possible to introduce bias in the growth/retraction process, however some care must be taken in choosing the correct event rates. To generalise the factor $1/2$, fix some $r > 0$ and consider
    \begin{align*}
        \alpha_{\text{growth}} = \frac{r}{r+1} \alpha_\varepsilon = \frac{r}{r+1} \frac{\alpha}{\varepsilon^2} \,, \quad
        \alpha_{\text{retraction}} = \frac{1}{r+1} \alpha_\varepsilon = \frac{1}{r+1} \frac{\alpha}{\varepsilon^2} \,,
    \end{align*}
    in which case $r=1$ gives back the initial rates. However \eqref{Eqn: Generator of jump length process approx} then becomes
    \begin{align*}
        Af(\ell) &\approx \frac{r}{r+1} \frac{\alpha}{\varepsilon^2} \Big( \varepsilon f'(\ell) + \frac{1}{2} \varepsilon^2 f''(\ell) \Big) + \frac{1}{r+1} \frac{\alpha}{\varepsilon^2} \Big( - \varepsilon f'(\ell) + \frac{1}{2} \varepsilon^2 f''(\ell) \Big) \\
        &= \Big(\frac{r-1}{r+1}\Big) \frac{\alpha}{\varepsilon}  f'(\ell) + \frac{\alpha}{2} f''(\ell) \,,
    \end{align*}
    which diverges as $\varepsilon\rightarrow0$ for $r\neq 1$. Instead add a biased order $\varepsilon^{-1}$ component to the rate
    \begin{align*}
        \alpha_{\text{growth}} = \frac{1}{2} \frac{\alpha}{\varepsilon^2} + \frac{r}{r+1} \frac{\alpha}{\varepsilon} \,, \quad
        \alpha_{\text{retraction}} = \frac{1}{2} \frac{\alpha}{\varepsilon^2} + \frac{1}{r+1} \frac{\alpha}{\varepsilon} \,,
    \end{align*}
    so that
    \begin{align*}
        Af(\ell) &\approx \Big( \frac{1}{2} \frac{\alpha}{\varepsilon^2} + \frac{r}{r+1} \frac{\alpha}{\varepsilon} \Big)\Big( \varepsilon f'(\ell) + \frac{1}{2} \varepsilon^2 f''(\ell) \Big) + \Big( \frac{1}{2} \frac{\alpha}{\varepsilon^2} + \frac{1}{r+1} \frac{\alpha}{\varepsilon} \Big) \Big( - \varepsilon f'(\ell) + \frac{1}{2} \varepsilon^2 f''(\ell) \Big) \\
        &=\frac{\alpha}{2} f''(\ell) + \alpha \Big( \frac{r-1}{r+1} \Big) f'(\ell) + \frac{\alpha}{2} \varepsilon f''(\ell) \\
        &\rightarrow \frac{\alpha}{2} f''(\ell) + \alpha \Big( \frac{r-1}{r+1} \Big) f'(\ell)  \,.
    \end{align*}
    This corresponds to a length process given by the SDE below, discussed further in Section \ref{Section: Alternative length processes},
    \begin{align*}
        d\ell_t = \alpha \Big( \frac{r-1}{r+1} \Big) \, dt + \sqrt{\alpha} \,\,d\beta_t \,.
    \end{align*}
\end{remark}

\subsubsection{Example simulations} \label{Section: Example simulations segments}

\begin{figure}[ht!]
    \centering
    \includegraphics[width=0.9\linewidth]{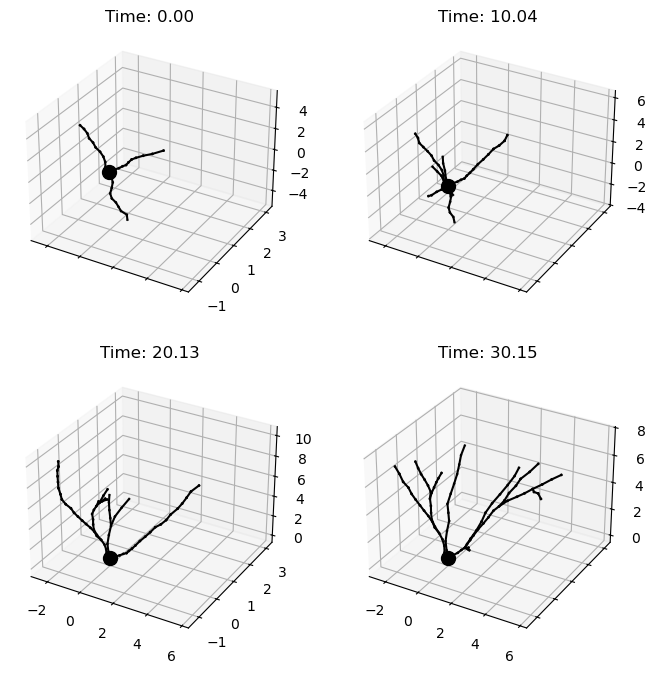}
    \caption{Example of segment discretisation simulation in 3D. The graph $G_t$ is plotted at four timepoints. The soma is indicated by a larger node. The following parameter values were used: $\varepsilon=0.5$, $\alpha=1$, $\beta=2$, $\beta_s = 0.5$, $\ell_0=1$, $\varsigma = 0.25$. As in Figure \ref{fig: Example simulations discrete time} the process is initialised with three short dendrites.}
    \label{fig:segment discretisation simulation}
\end{figure}

Figure \ref{fig:segment discretisation simulation} shows an example of the segment discretisation. The standard updates shown in Figure \ref{fig:segment discretisation diagram} are used, with the addition of a fourth update that adds a new branch from the soma (as discussed in Section \ref{Section: Soma branching}) at rate $\beta_s \varepsilon$. A similar growth process to the example in Section \ref{Section: Example simulations discrete time} is used, except that for this example the model is placed in three spatial dimensions, meaning two growth angles are required. This demonstrates the modelling flexibility gained by separating the growth process from other features. The only other alteration is a change in the branching rate, each internal node branches at rate $\beta (L_t)^{-1}$, where $L_t$ is the total length of the neuron and is equal to $\varepsilon$ multiplied by the total number of nodes. This means that the addition of new branches becomes less frequent as the neuron becomes larger. Note that having the branching rate depend on features of the graph $G_t$ mirrors including dependence on $\mu$ in the branching rate $\Lambda(\mu,\omega)$. 

Figure \ref{fig:segment discretisation simulation} shows plots of the graph $G_t$ at four timepoints. Due to the nature of the segment discretisation there is no notion of active or inactive branches, so the entire graph is plotted in the same way, with the exception of the soma being indicated by a larger node. As we are using a Gillespie simulation the timepoints at which snapshots are plotted are not exactly $t=10,20,30$ but the closest event times recorded. The code for this simulation can also be found at \url{github.com/AJNugent2/Dendrite-simulations}.

The initial direction of each branch was chosen randomly and the overall growth into positive $z$ values is a random occurrence, rather than a result of any bias in the growth direction. In this case it can be explained by the early retraction of the branch initially pointing downwards, while the two branches that persist both initially point upwards. 

Though their forms are very different, in the one case tracking branches and in the other a dynamic graph, both simulation methods discussed approximate the same process. The choice between them depends upon the desired spatial and temporal resolution, as well as the complexity of the length, growth and branching processes one wishes to include.

\section{Extensions} \label{Section: Extensions}

In this section we discuss several possible adaptations of the construction in Section \ref{Section: Construction}. These are intended to make existing elements of the process more realistic or flexible, or to add new features.

\subsection{Alternative length processes} \label{Section: Alternative length processes}

The setup described in Section \ref{Section: Length and growth processes} uses a Brownian motion for the length process, but when modelling dendrites it may be useful to consider alternatives to this.

One such alternative is the $b$-snake introduced in \cite{abraham2002representations,dhersin2000stochastic}. This is a modification of the Brownian snake in which a drift is introduced to the length process, which is given by
\begin{align*}
    d\ell_t = - b(\omega_t) \, dt +  d\Tilde{\beta}_t \,,
\end{align*}
where $\Tilde{\beta}_t$ is a linear reflecting Brownian motion and $b$ is a function on $\mathcal{W}$. It has also been shown that the corresponding version of the martingale problem for the $b$-snake is well-posed \cite[Section 2]{abraham2002representations}, hence Proposition \ref{Prop: Martingale problem for Brownian snake well-posed} may be extended immediately to this case. Again we may stop the process at $\ell=0$ and use Theorem 4.6.1 of \cite{ethier2009markov} to obtain well-posedness of the process absorbed at $\ell=0$. 

This function $b:\W \rightarrow \mathbb{R}$ could be used to encourage growth in a particular direction, for example in response to a chemoattractant, or to limit the dendrite's growth, for example by increasing $b$ as $\ell_t$ approaches some maximum length. Both these examples use only information about the endpoint $(\ell, \hat{\omega})$, but in principle information from along the entire path could be utilised. This also demonstrates that it is possible to introduce coupling between the length process and the spatial position. 

Another option is the use of time changes. This can be seen in the construction of the On/Off Brownian snake in \cite{jacobi2024off} where there are periods of `hibernation' in which the length process is frozen. This allows the inclusion of periods of inactivity or periods of faster/slower growth. Such random time changes are discussed in \cite[Chapter 6]{ethier2009markov}. 

A much broader class of length processes is considered in Appendix \ref{Appendix: Deterministic motion} in the case of deterministic growth. In this setting, each branch is assumed to simply grow forwards in a straight line in some given direction. While this significantly reduces the complexity of the model, this assumption is highly unrealistic and removes much of the value in modelling the spatial position of neurons.

It has also been proposed \cite{shree2022dynamic} that dendritic growth follows a `dynamic instability' and switches between longer periods of growth, pauses and retraction. This again corresponds to a change in the length process $\ell_t$. Rather than considering an absorbed Brownian motion define
\begin{align*}
    \ell_t = \ell_0 + \int_0^t v_s \, ds \,,
\end{align*}
where $v_s \in \mathbb{R}$ is the growth/retraction rate at time $s$ and is defined by a jump process on $\mathbb{R}$. As with the Brownian motion we may define $\tau$ as the hitting time of zero and use the stopped process $\ell_{t \wedge \tau}$ to absorb the process at $\ell=0$. This length process is slightly different from the others we have considered as the evolution of $\ell_t$ depends on the additional variable $v_t$, which is not included in the state space $\W$. Therefore the construction below of the branching process must be altered slightly to include this variable (in much the same way that the `activity' variable was introduced). The structure of the generator in this case is examined in Appendix \ref{Appendix: Dynamic instability}. 

In practice, a version of the Brownian snake could be defined for any reasonable length process (for example a Feller process that is almost surely finite with almost surely continuous sample paths). This allows existing models for the evolution of dendrite lengths to be incorporated into this setup. From this length process, the generator of the corresponding Brownian snake can be defined by \eqref{Eqn: Definition of generator}. The proof that such a process is \textit{uniquely} defined by its generator (extending the results of \cite{dhersin2000stochastic}) is missing in general, leaving a technical gap in the construction. However the remainder of the construction, in which we switch to a measure-valued process and add jumps for new branches, follows exactly as in Section \ref{Section: branching process}. 

As the generator is different to the standard case the results of Section \ref{Section: generators for branch number and length} on the total length process $L_t$ will no longer hold, but the remainder of the results in Section \ref{Section: Properties} are independent of this and so still hold. The simulation of the length process in the time discretisation and the rates $\alpha_\varepsilon$ in the segment discretisation will also need to be adapted (see Remark \ref{Remark: on bias in segments} for the latter). 

\subsection{Heterogeneity}

In certain situations it may be desirable to have some heterogeneity between branches. A clear example of this is to differentiate between dendrites and an axon. We discuss several ways heterogeneity could be incorporated and give examples of each. 

The simplest case is that of constant parameter/parameters per branch. For each branch let $\vartheta^i$ be a collection of parameters assigned to branch $i$ upon its creation. These parameters are different for each branch but do not change over time. These could be deterministic or chosen randomly from some given distribution, possibly depending on the time of the branch creation or features of the parent branch. Call the parameter space $\Theta$. 

From a technical point of view, these parameters are tracked in a similar way to the activity variable $a^i$. Defining $\mathcal{S} = \W \times \{0,1\} \times \Theta$ we still have that the measure $\mu_t$ is defined as an element of $\mathcal{N}_\mathcal{S}^{\#}$. This immediately allows us to track the value of $\vartheta$ for each branch. The question of how it affects the generator depends on the purpose of the parameter, we give several examples below. 

\begin{example}
    Heterogeneous branching rate.

    In this first example we assume that some dendrites branch more than others, with $\vartheta > 0$ representing the relative branching rate of the dendrite, so $\vartheta > 1$ indicates faster than normal branching and $\vartheta < 1$ indicates slower than normal branching. Assume that this parameter is chosen independently for each branch upon its creation from a given density $\rho_\vartheta \in \mathcal{M}(\Theta)$. This translates directly into the branching rate, so that this part of the generator becomes
    \begin{align*}
        \mathcal{C}^+F(\mu) 
        &= \bigg\langle \int\int_0^{\ell} \textcolor{blue}{\int_\Theta} \textcolor{blue}{\vartheta} \Lambda(\omega) \,\Big( F(\mu^+(\omega,a,u,\textcolor{blue}{\vartheta'})) - F(\mu) \Big) \, \textcolor{blue}{\rho_\vartheta(d\vartheta')} \, \rho_{\omega}(du) \,P(d\nu) \,,\mu \bigg\rangle \,.
    \end{align*}
    with the newly added terms highlighted in blue. There are two effects to be accounted for: the branching rate changing from $\Lambda(\omega)$ to $\Lambda(\omega,\vartheta)=\vartheta \Lambda(\omega)$ and new branches being given their own parameter $\vartheta'$ from the distribution $\rho_\vartheta$. 
\end{example}

\begin{example}
    Heterogeneous growth rates.

    In this example we consider that some branches might be more likely to grow/retract than others. Assume we have a $b$-snake (see Section \ref{Section: Alternative length processes}) with length process
    \begin{align*}
        d\ell_t = d\Tilde{\beta}_t - b(\omega_t,\vartheta) \, dt \,,
    \end{align*}
    where the drift term $b(\omega_t,\vartheta)$ is now dependent on the parameter $\vartheta$. For a given $\vartheta$ we can define a version of the generator for the $b$-snake with this drift term, denoting this by $\mathcal{A}(\vartheta)$. This tracks through to the generator with activity $\mathcal{B}(\vartheta)$ and the generator for the empirical measure $\mathcal{C}(\vartheta)$. Although $\vartheta$ does not itself affect branching we still need to update the branching term so that new branches are given a parameter value. The branching generator is 
    \begin{align*}
        \mathcal{C}^+F(\mu) 
        &= \bigg\langle \int\int_0^{\ell} \textcolor{blue}{\int_\Theta} \Lambda(\omega) \,\Big( F(\mu^+(\omega,a,u,\textcolor{blue}{\vartheta'})) - F(\mu) \Big) \, \textcolor{blue}{\rho_\vartheta(d\vartheta')} \, \rho_{\omega}(du) \,P(d\nu) \,,\mu \bigg\rangle \,.
    \end{align*}
    Other parameters affecting the dynamics of an individual branch would be incorporated in the same way, using a single-Branch generator $\mathcal{A}(\vartheta)$ that is dependent on $\vartheta$ and following this change through. 
\end{example} 

\begin{example}
    Modelling the axon. 

    Assume that one of the branches is an axon. In this case $\vartheta \in \{0,1\}$ with $\vartheta=0$ for dendrites and $\vartheta=1$ for the axon. The axon will have an increased growth rate but a decreased branching rate, combining the effects of the first two examples. The value of $\vartheta$ is simply copied from the parent branch to the child branch. 
\end{example}

\begin{example}
    Tree depth and distance to soma.

    In this example $\vartheta$ is not chosen randomly but is used to track information about the graph structure. When a new branch is created its tree depth is simply the tree depth of its parent branch plus one. Additionally, the distance to the soma is that of its parent branch plus $u$, where $u$ is the distance along the parent branch at which the new branch is added. These features could simply be recorded or they could influence the dynamics; for example, the branching rate could decrease as the tree depth or distance to soma increases.
\end{example}

The situation is more complex when these parameters change over time. If the change is deterministic then one could simply add a known time dependence $\vartheta(t)$. As the process is well-posed for fixed $\vartheta$ it is also well-posed for deterministically changing $\vartheta$, provided that all event rates remain uniformly bounded (see \cite[Chapter 4.7A]{ethier2009markov} for details). We give an example in this case. 

\begin{example}
    Dependence on branch age.

    Similarly to Example 4 we introduce a deterministic variable that can either simply be tracked, or used to influence the dynamics. In this case the variable is the age of a branch, initially set to 0 when the branch is created and increasing linearly with rate 1. Tracking the age of each branch allows the length, growth and branching processes to depend on this age. For example, the length process might increase at a constant rate (rather than being stochastic) until the branch reaches a certain age. This could represent some resources dedicated to creating the branch and persist for a short period of time and prevent the immediate retraction of the new branch. (We could instead record the time at which the branch was created and compare that with the current time to determine its age, although this still makes the process non-homogeneous).
\end{example}

If the parameters are themselves stochastic processes then new components would need to be included in the generator to describe their evolution. Such complexity is unnecessary for all the examples considered so far and we do not study this case further.

\subsection{Additional coupling}

There are two types of coupling to consider: internal coupling between the different length/growth/branching processes and external coupling with the environment or other processes therein. 

Firstly consider internal coupling. There are already many opportunities to include such coupling, for example: the drift term $b(\omega_t)$ allows a coupling between the growth and length processes, the branching rate $\Lambda(\mu,\omega)$ can depend on features of the whole neuron, and the parameters described in the previous section can introduce dependence on features such as tree depth. For example, in certain types of neurons there is evidence of self-avoidance of dendrites \cite{matthews2007dendrite,soba2007drosophila,emoto2004control} leading to a `space-filling' behaviour that is reproduced in some models \cite{shree2022dynamic,cuntz2012scaling,luczak2006spatial}. In this case it is again important that the process $\omega_t$ tracks the entire path as this provides boundaries for the spatial process. While introducing further internal coupling in the model construction is technically challenging, it is typically straightforward to do so in numerical simulations (in particular in the time discretisation with small $\Delta t$). 

External coupling can similarly come in many forms. Typically this would affect the growth process, for example a chemoattractant in the background that biases the direction of growth, but in principle all process can be coupled to the path $\omega$, meaning they have access to this spatial information and can therefore take cues from the environment. This means that the length process, growth process and branching rates can all depend on environmental information, either in the form of a fixed background signal or arising from coupling with other neurons. The flexibility of the model allows different levels of coupling to be included and therefore could be used to test which of these couplings are present in real neurons. 

\subsection{Example}

An example including several of these effects is shown in Figure \ref{fig:apical and basal}. Here we have separated dendrites into two types: basal dendrites (shown in blue) grow in the `normal' way, while apical dendrites (shown in red) are biased to grow upwards and have a position-dependent branching rate (shown in the background). This setup is intended to replicate the apical tuft seen in hippocampal neurons, for example in \cite{krieger2007synaptic}. Each neuron is initialised with three basal dendrites and one apical dendrite. 
The parameters in this example are not fitted and we do not aim to conclude that the formation of the apical tuft is due to a signal of this sort, instead this is simply an illustrative example of our model's flexibility. 

\begin{figure}[ht!]
    \centering
    \includegraphics[width=\linewidth]{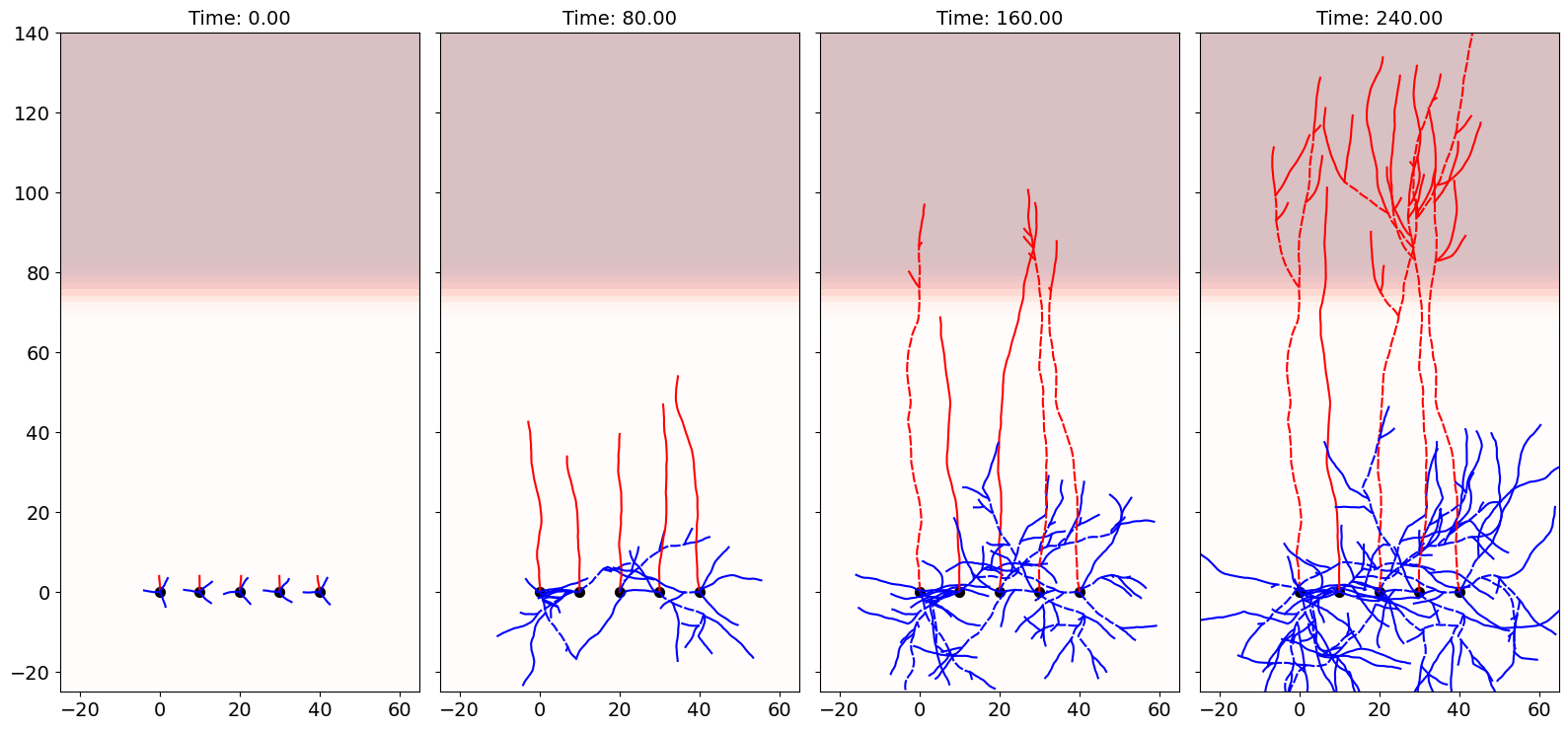}
    \caption{Example showing dendrite heterogeneity and response to external signals. Dendrites are split into basal (blue) and apical (red) dendrites, with the latter biased to grow upwards and with a position-dependent branching rate.}
    \label{fig:apical and basal}
\end{figure}

\section{Conclusion} \label{Section: Conclusion}

The model presented in this paper solves three key challenges in correctly capturing the behaviour of neurons. Firstly, by presenting a novel a path-valued process we track the information necessary to include retraction and side-branching. In particular, the Brownian snake provides a natural model for individual dendrites by encoding the switch between growth and retraction and avoids pre-specifying branch or end points such as in \cite{cuntz2010one,kirchner2025dendritic}. Secondly, using a measure to track multiple dendrites allows a clean description of the branching process in a single state space. Finally, the introduction of active and inactive branches solves the branching-retraction problem without limiting the length process or creating unnecessary coupling between branches. Some existing models, such as \cite{shree2022dynamic,ferreira2020achieving,koene2009netmorph}, implicitly track a path-valued process and so can be understood as examples of this general framework, meaning solutions such as branch (in)activity could be adopted directly. 

Moreover, the separation of length, growth and branching processes creates great flexibility, as demonstrated in the examples in Section \ref{Section: numerical simulation} and the extension in Section \ref{Section: Extensions}, including the ability to easily extend the model into 3D or a specific spatial domain. Real dendrites show a range of morphologies, see for example \cite{cajal1995histology,polavaram2014statistical} and the NeuroMorpho database, therefore it is useful that a single model can be adapted to reflect this heterogeneity rather than restricting to a single neuronal cell type. This flexibility also raises the possibility of adapting the model to other scenarios in which similar challenges occur, for example in root growth or modelling branching tissues such as the lungs or kidneys. 

As well as providing a rigorous modelling framework, we have also outlined several different methods for numerical simulation at a range of temporal and spatial resolutions. These will prove useful in future work fitting models to data to understand what processes and parameters control real dendrite development in different settings. In both simulation methods, there is a separation between the parameters/rates controlling growth and retraction and the parameters/rate controlling branching. This allows these processes to be fitted independently. In addition, parameter fitting could be done without measuring traditional morphology statistics, such as Sholl analysis, meaning that these could be used to validate model predictions. 

The results of Section \ref{Section: Properties} show that the model produces the expected tree-like structure of neurons and provides insight into the process, for example, predicting the total length and number of (active) branches. This is made possible through the description of the generator of the full process, which was also used to understand the correct scaling of rates in the segment discretisation. 

In summary, by providing a more rigorous solution to modelling retraction and avoiding the branching-retraction problem, we have developed a general model that captures the vital features of dendrite growth and have shown its effectiveness through both analytical results and simulations. We hope that the ideas in this paper, in particular the use of path-valued processes and monitoring branch activity will be adopted in other dendrite growth models and prove useful in comparisons with neural imaging data. 

\bibliography{bibliography}

\section*{Acknowledgements}

This work was supported by a Leverhulme Trust Research Project Grant (RPG-2024-360) to LCA and KMP. The authors would also like to thank Lavinia Mitiko Takarabe, Joshua Looker and the lab groups of Laura Andreae and Benjamin Walker for their support and useful discussions. 

\vspace{1em}
\noindent CRediT author statement: 

\noindent\textbf{Andrew Nugent:} Conceptualisation, methodology, software, formal analysis, investigation, writing - original draft, writing - editing, visualisation. 
\textbf{Karen Page:} Conceptualisation, writing - editing, supervision, project administration, funding acquisition. 
\textbf{Alexey Zaikin:} Supervision, project administration. 
\textbf{Laura C. Andreae:} Conceptualisation, writing - editing, supervision, project administration, funding acquisition.

\vspace{1em}
\noindent Data Availability Statement: No data was used in the preparation of this manuscript. 

\vspace{1em}
\noindent Compliance with Ethical Standards: The authors declare no conflicts of interest. The research in this paper did not involve the use of human participants or animals (no data was used in the preparation of this manuscript).  

\newpage
\appendix

\section{Form of the generator} \label{Appendix: Form of generator}

\subsection{Cylinder functions and separation} \label{Appendix: A}

The following Proposition \ref{Prop: D(A) separating} will be required to add a jump process corresponding to the creation/removal of branches. We firstly require the following Lemma \ref{Lemma: cylinder functions in D(A)}.

\begin{lemma} \label{Lemma: cylinder functions in D(A)}
    The class of cylinder functions
    \begin{align} \label{Def: cylinder functions}
        D_c = \Big\{ F(\omega) := \varphi\big( \ell_\omega, \omega(u_1 \wedge \ell_\omega), \dots, \omega(u_k \wedge \ell_\omega) \big) \,\Big|\, k \in \mathbb{N}, u_1\leq \dots\leq u_k,  \varphi \in C^2(\mathbb{R}^{k+1}) \Big\} \,,
    \end{align}
    lie in the domain $\mathcal{D}(\mathcal{A})$. 
\end{lemma}
\begin{proof}
    Let $F \in D_c$, defined by \eqref{Def: cylinder functions}, then
    \begin{align} \label{Eqn: Generator of cylinder functions}
        &\frac{1}{t}\mathbb{E}_\omega \big[ F(\omega_t) - F(\omega)\big] \nonumber\\
        &\quad= \frac{1}{t}\mathbb{E}_\omega \big[ \varphi\big( \ell_t, \omega_t(u_1 \wedge \ell_t), \dots, \omega_t(u_k \wedge \ell_t) \big) - \varphi\big( \ell_\omega, \omega(u_1 \wedge \ell_\omega), \dots, \omega(u_k \wedge \ell_\omega)\big)\big] \,.
    \end{align}
    If $u_k \wedge \ell_\omega < \ell_\omega$ then for small $t$, since $\ell_t$ changes continuously, we will have that $F(\omega_t) = F(\omega)$ as the difference at the endpoint will not be captured in this finite dimensional projection. Therefore assume that $u_k \wedge \ell_\omega = \ell_\omega$. This implies that $u_k \geq \ell_\omega$. 

    If $u_k \geq \ell_\omega$ then for sufficiently small times we will have that $u_k \wedge \ell_t = \ell_t$. In this case $\varphi$ is effectively a function of $\ell_t$ and $\hat{\omega}$ only, we suppress dependence on $\omega(u_i)$ for $i=1,\dots,k-1$ below for readability. If there are multiple $u_i$ with $u_i \geq \ell_\omega$ then $\varphi$ is still effectively a function of $\ell_t \in \mathbb{R}$ and $\hat{\omega} \in X$ only. 
    
    As both $\ell_t$ and the spatial motion are continuous and $\varphi$ is smooth, the limit of \eqref{Eqn: Generator of cylinder functions} as $t\rightarrow0$ will exist. More specifically, 
    \begin{align*}
        \frac{1}{t}\mathbb{E}_\omega \big[ F(\omega_t) - F(\omega)\big] 
        &= \frac{1}{t}\mathbb{E}_\omega \big[ \varphi\big( \ell_t, \hat{\omega}_t\big) - \varphi\big( \ell, \hat{\omega}\big)\big] \\
        &= \frac{1}{t}\mathbb{E}_\omega \big[ \varphi\big( \ell_t - \ell + \ell, \hat{\omega}_t - \hat{\omega} + \hat{\omega} \big) - \varphi\big( \ell, \hat{\omega}\big)\big] \\
        &= \frac{1}{t}\mathbb{E}_\omega \big[ (\ell_t - \ell) \, \partial_1 \varphi(\ell,\hat{\omega}) + (\hat{\omega}_t-\hat{\omega}) \, \nabla_2\varphi(\ell,\hat{\omega}) \\
        &\hspace{1.5cm}+ (1/2)(\ell_t - \ell)^2 \, \partial_1^2 \varphi(\ell,\hat{\omega}) \\[0.5em]
        &\hspace{1.5cm}+ (1/2)(\hat{\omega}_t-\hat{\omega})^T \, \nabla_2^2\varphi(\ell,\hat{\omega}) (\hat{\omega}_t-\hat{\omega})\\[0.5em]
        &\hspace{1.5cm}+ (\ell_t - \ell)\,\partial_1 \nabla_2 \varphi(\ell,\hat{\omega})(\hat{\omega}_t-\hat{\omega}) \big]
    \end{align*}
    Looking at some terms individually 
    \begin{align*}
        \frac{1}{t} \mathbb{E}_\omega \big[ (\ell_t - \ell) \, \partial_1 \varphi(\ell,\hat{\omega}) \big] 
        &= \partial_1 \varphi(\ell,\hat{\omega})\, \frac{1}{t} \mathbb{E}_\omega \big[ \ell_t - \ell \big] = 0 \,, \\
        \frac{1}{t} \mathbb{E}_\omega \big[ (\ell_t - \ell)^2 \, \partial_1^2 \varphi(\ell,\hat{\omega}) \big] 
        &= \partial_1^2 \varphi(\ell,\hat{\omega})\, \frac{1}{t} \mathbb{E}_\omega \big[ (\ell_t - \ell)^2 \big] = \partial_1^2 \varphi(\ell,\hat{\omega}) \,.
    \end{align*}
    Conditional on $\ell_t - \ell = r$, $\hat{\omega}_t$ is distributed according to $\xi_r$ started from $\hat{\omega}$. Hence all functions involving $\hat{\omega}_t - \hat{\omega}$ can be bounded using the generator of this continuous spatial motion. As $\ell_t$ is a Brownian motion, $r$ is of order $\sqrt{t}$ so all such terms will vanish. Thus in this case the generator has the form 
    \begin{align*}
        \mathcal{A}F(\omega) = \frac{1}{2} \partial_1^2 \varphi(\ell_\omega, \omega) \,.
    \end{align*}
\end{proof}

\begin{proposition} \label{Prop: D(A) separating}
    The domain $\mathcal{D}(\mathcal{A})$ is separating on $\W$. 
\end{proposition}
\begin{proof}
    The class of cylinder functions \eqref{Def: cylinder functions} capture the finite dimensional projections of $\W$ and therefore generate the Borel $\sigma$-field, thus this class of functions separates measures \cite[Example 1.3]{billingsley2013convergence}. By Lemma \ref{Lemma: cylinder functions in D(A)} the cylinder functions lie in the domain $\mathcal{D}(\mathcal{A})$ and so $\mathcal{D}(\mathcal{A})$ is also separating.  
\end{proof}

\subsection{Absorbed Brownian motion} \label{Appendix: Absorbed Brownian motion}

In \cite{dhersin2000stochastic} the behaviour of the generator on test functions $f\in D$ is described using the generator of a simpler process, known as the A-path \cite{venttsel1986infinitesimal,dawson2006measure}. In this process the length grows deterministically at constant rate 1, that is $\ell_t = \ell + t$. As in the definition of the Brownian snake, $(w(\ell+s), 0 \leq s \leq t)$ is distributed according to the process $\xi$ started at $\omega(\ell)$. Letting $L$ denote the generator of the A-path, which is well-defined for $F \in D$, it is shown that the generator of the Brownian snake satisfies
\begin{align*}
    \mathcal{A}F(\omega) = \frac{1}{2} Lg(\omega) \,.
\end{align*}
A small modification to the Brownian snake can be made by replacing the length process (originally a reflected Brownian motion) with a Brownian motion absorbed at $\ell=0$. 

We firstly give the proof of Lemma \ref{Lemma: gamma for absorbed BM}, restated below, which is a standard application of the reflection principle.

\begin{lemma}
    When $\ell_t$ is a Brownian motion absorbed at $\ell=0$ we have 
    \begin{align*}
        \gamma_t^\ell(dp,dq) &= \frac{2(\ell+q-2p)}{\sqrt{2\pi t^3}} \, \exp\bigg( -\frac{(\ell+q-2p)^2}{2t}\bigg) \, \mathds{1}_{\{ 0 < p < \ell\wedge q \}}\,dp \, dq \, \\
        &\quad + 2\bigg(1 - \Phi\left(\frac{\ell}{\sqrt{t}}\right)\bigg) \delta_{(0,0)}(dp,dq) \,.
    \end{align*}
\end{lemma}
\begin{proof}
    Let $m_t = \inf_{0\leq s \leq t }\ell_s$. For a standard Brownian motion, the reflection principle gives that 
    \begin{align*}
        \mathbb{P}[m_t \leq p, \ell_t \leq q] 
        &= \mathbb{P}[\ell_t \leq q] - \mathbb{P}[m_t > p, \ell_t \leq q] \\
        &= \mathbb{P}[\ell_t \leq q] - \mathbb{P}[\ell_t \leq 2p -q]
    \end{align*}
    hence
    \begin{align*}
        \mathbb{P}[m_t \leq p , \ell_t \in dq] = \frac{1}{\sqrt{2\pi t}} \Bigg( \exp\bigg(- \frac{(q-\ell)^2}{2t}\bigg) - \exp\bigg( -\frac{(2p-q-\ell)^2}{2t}\bigg) \Bigg)\, dq \,.
    \end{align*}
    Then differentiating with respect to $a$ gives
    \begin{align*}
        \mathbb{P}[m_t \in dp , \ell_t \in dq] = \frac{2(\ell+q-2p)}{\sqrt{2\pi t^3}} \, \exp\bigg( -\frac{(\ell+q-2p)^2}{2t}\bigg) \, dp \, dq
    \end{align*}
    for $p \leq \min(\ell,q)$. To include absorption at zero we introduce a stopping time. For $B_t$ a standard Brownian motion started at $\ell > 0$ let $\tau_0 = \inf\{B_t = 0\}$, then $\ell_t = B_{t \wedge \tau_0}$. For $t<\tau_0$ the process is a standard Brownian motion so has the density above, while for $t\geq \tau_0$ both $\ell_t$ and $m_t$ are equal to zero. Again using the reflection principle we have 
    \begin{align*}
        \mathbb{P}[\tau_0 \leq t] = \mathbb{P}\Big[ \inf_{0 \leq s \leq t} B_t \leq 0 \Big] = \mathbb{P}\Big[ \sup_{0 \leq s \leq t} B_t \geq 2\ell \Big] = 2\mathbb{P}[ B_t \geq 2\ell] = 2(1 - \Phi(\ell/\sqrt{t})) \,,
    \end{align*}
    which gives the result. 
\end{proof}

Using this we may prove the following, which confirms the form of the generator when using an absorbed Brownian motion for the length process. 

\begin{proposition} \label{Prop: Generator of snake with absorbed BM}
    When $\ell_t$ is a Brownian motion absorbed at $\ell=0$, the test functions $D$ form a subset of the domain $\mathcal{D}(A)$. More specifically, for all $F \in D$ we have the limit
    \begin{align*}
        \lim_{t\rightarrow0} \frac{T(t)F - F}{t} 
        = \frac{1}{2} Lg(\omega) \, \mathds{1}_{\{ \ell_\omega >0 \}} \,.
    \end{align*}
\end{proposition}
\begin{proof}
    Take any $\omega \in \W$. If $\ell_\omega = 0$ then this remains constant, hence 
    \begin{align*}
        T(t)F(\omega) = \mathbb{E}_\omega \big[ F(\omega_t) \big] = \mathbb{E}_\omega \big[ F(\omega) \big] = F(\omega)
    \end{align*}
    and so the generator is equal to zero.

    The remainder of the proof below is very similar to that of \cite[Theorem 1]{dhersin2000stochastic} and so some details are omitted. Assume $\ell_\omega > 0$. Letting $m_t = \inf_{0 \leq s \leq t} \ell_s$ we have that $\omega_t(u) = \omega(u)$ for $u \in [0,m_t]$. Therefore we only need to consider differences after this point, 
    \begin{align*}
        F(\omega_t) - F(\omega) = \int_{m_t}^{\ell_t} g\big(\omega_t^{\leq r}\big) \, dr - \int_{m_t}^{\ell} g\big(\omega^{\leq r}\big) \, dr \,.
    \end{align*}
    Using the independence of the length process we may condition on its value. 
    \begin{align*}
        \mathbb{E}_\omega \big[ F(\omega_t) - F(\omega) \big]
        &= \mathbb{E}_\omega \bigg[ \int_{m_t}^{\ell_t} g\big(\omega_t^{\leq r}\big) \, dr - \int_{m_t}^{\ell} g\big(\omega^{\leq r}\big) \, dr \bigg] \\
        &= \int\int \mathbb{E}_\omega \bigg[ \int_{p}^{q} g\big(\omega_t^{\leq r}\big) \, dr - \int_{p}^{\ell} g\big(\omega^{\leq r}\big) \, dr \bigg]\, \gamma_t^\ell(dp,dq)
    \end{align*}
    Under this conditioning, $\omega_t$ is an extension of $\omega^{\leq p}$, with the extension distributed according to the process $\xi$ started at $\omega(p)$. This means that $(\omega_t, p \leq t \leq q)$ is an A-path and thus we can describe it using the generator of the A-path. That is 
    \begin{align*}
        \mathbb{E}_\omega \Big[ g\big( \omega_t^{\leq r} \big) \big| \, \ell_t = q, m_t = p\Big] = g(\omega^{\leq p}\big) + \int_0^{r-p} \mathbb{E} \big[Lg(\omega'_u) \, | \, \omega'_0 = \omega^{\leq p}\big] \, du \,,
    \end{align*}
    where $\omega'$ is the A-path. Combining this with the above we have 
    \begin{align*}
        \mathbb{E}_\omega \big[ F(\omega_t) - F(\omega) \big]
        &= \int_{(\mathbb{R}^+)^2} \bigg[ \int_{p}^{q} \mathbb{E}_\omega \Big[ g\big(\omega_t^{\leq r}\big) \Big] \, dr - \int_{p}^{\ell} g\big(\omega^{\leq r}\big) \, dr \bigg]\, \gamma_t^\ell(dp,dq) \\
        &= \int_{(\mathbb{R}^+)^2} \int_{p}^{q} g(\omega^{\leq p}\big) \, dr \, \gamma_t^\ell(dp,dq) + R_1 - R_2 \,,
    \end{align*}
    where similarly to \cite{dhersin2000stochastic} we define
    \begin{align} \label{Eqn: R1 and R2 definition}
        R_1 &= \int_{(\mathbb{R}^+)^2} \int_{p}^{q} g(\omega^{\leq p}\big) \, dr \, \gamma_t^\ell(dp,dq) \,,\\
        R_2 &= \int_{(\mathbb{R}^+)^2}  \int_{p}^{\ell} g\big(\omega^{\leq r}\big) \, dr \, \gamma_t^\ell(dp,dq) \,.
    \end{align}
    In the case of a reflected Brownian motion the difference between these terms vanishes sufficiently quickly as $t\rightarrow0$ that they can be discounted. In the case of an absorbing Brownian motion we may use Lemma \ref{Lemma: gamma for absorbed BM} to check that in fact $R_1=R_2$. Therefore the only remaining term is 
    \begin{align*}
        \mathbb{E}_\omega \big[ F(\omega_t) - F(\omega) \big]
        &= \mathbb{E}_\ell \Bigg[ \int_{p}^{q} \int_0^{r-p} \mathbb{E} \big[Lg(\omega'_u) \, | \, \omega'_0 = \omega^{\leq p}\big] \, du  \, dr \Bigg] \\
        &= \mathbb{E}_\ell \Bigg[ \int_{0}^{q-p} \int_{p+u}^{q} \mathbb{E} \big[Lg(\omega'_u) \, | \, \omega'_0 = \omega^{\leq p}\big] \, dr  \, du \Bigg] \\
        &= \mathbb{E}_\ell \Bigg[ \int_{0}^{q-p} (q-p-u) \, \mathbb{E} \big[Lg(\omega'_u) - Lg(\omega) \, | \, \omega'_0 = \omega^{\leq p}\big]  \, du \Bigg] \\
        &\quad+ \mathbb{E}_\ell \Bigg[ \int_{0}^{q-p} (q-p-u) \, \mathbb{E} \big[Lg(\omega) \, | \, \omega'_0 = \omega^{\leq p}\big]  \, du \Bigg] 
    \end{align*}
    The first term vanishes by dominated convergence \cite{dhersin2000stochastic}. Looking at the second term
    \begin{align*}
        &\mathbb{E}_\ell \Bigg[ \int_{0}^{q-p} (q-p+u) \, \mathbb{E} \big[Lg(\omega) \, | \, \omega'_0 = \omega^{\leq p}\big]  \, du \Bigg] \\
        &\hspace{1cm}= Lg(\omega) \,\mathbb{E}_\ell \big[(q-p)^2\big] - Lg(\omega) \, \mathbb{E}_\ell \Bigg[ \int_{0}^{q-p} u \, du \Bigg] \\
        &\hspace{1cm}= \frac{1}{2} Lg(\omega) \,\mathbb{E}_\ell \big[(q-p)^2\big] \,.
    \end{align*}
    As with the reflecting Brownian motion, $\mathbb{E}_\ell \big[(q-p)^2\big] = t + o(t)$ and so
    \begin{align*}
        \frac{1}{t} \mathbb{E}_\omega \big[ F(\omega_t) - F(\omega) \big] \rightarrow \frac{1}{2} Lg(\omega) \, \mathds{1}_{\{ \ell_\omega >0 \}} \,.
    \end{align*}
\end{proof}
This confirms that the only change to the generator is the inclusion of the indicator $\mathds{1}_{\{ \ell_\omega >0 \}}$ as expected. 

\section{Deterministic motion} \label{Appendix: Deterministic motion}

Modelling is significantly simpler in the case of deterministic spatial growth, as it is no longer necessary to track the entire path of the dendrite, but is sufficient to know its length $\ell_t$ and deterministic growth trajectory. 

Assume $X$ is a compact subset of $\mathbb{R}^n$ and there is some known, continuous vector field $V:X \rightarrow X$ describing the spatial growth of the dendrite's endpoint. That is, a dendrite tip at $x \in X$ will grow in the direction $V(x)$. To introduce some variety this vector field may depend continuously on some (constant) parameters $\theta$ that will be randomly chosen for each dendrite upon its creation. Denote this vector field by $V_\theta(x)$. 

For example, let $X\subset \mathbb{R}^2$ and set $\theta \in [0,2\pi)$ to be a randomly chosen growth direction of the dendrite. Then $V_\theta(x) = (\cos(\theta),\sin(\theta))$ describes constant growth in this direction. This is similar to the straight growth assumed in \cite{li1992neurite}. 

Beginning from $x_0$ the vector field $V_\theta$ creates a flow $\phi(x_0,\theta,\ell)$ defined by
\begin{align*}
    \phi(x_0,\theta,\ell) = x_0 + \int_0^\ell V_\theta( \phi(x_0,\theta,s) ) \, ds \,.
\end{align*}
That is, $\phi$ is the solution to the initial value problem
\begin{align*}
    \partial_\ell \phi(x_0,\theta,\ell) &= V_\theta( \phi(x_0,\theta,\ell) ) \,,\\
    \phi(x_0,\theta,0) &= x_0 \,.
\end{align*}
The dendrite's endpoint follows this flow, meaning 
\begin{align} \label{Eqn: omega for deterministic motion}
    \omega_t = \{ \phi(x_0, \theta, s) \,|\, s \in [0,\ell_t] \} \,.
\end{align}

Returning to the example above, this is simply 
\begin{align*}
    \omega_t = \{ x_0 + (s \cos(\theta),s \sin(\theta)) \,|\, s \in [0,\ell_t] \} \,.
\end{align*}

The dendrite $\omega_t$ is therefore determined by the triple $(x_0,\theta,\ell_t)$ and so, rather than using the state space $\W$, we may use the simpler state space $X \times \Theta \times \mathbb{R}^+$, where $\Theta$ is the set of possible values of the parameter $\theta$. We assume that $\Theta$ is a compact subset of $\mathbb{R}^m$ for some $m\geq 0$. Note that only the final coordinate $\ell_t$ is dynamic. 

\begin{assumption} \label{Assumption: ell Feller}
    The length process $\ell_t$ is a Feller process on $\mathbb{R}$ with generator $\mathcal{A}_\ell$ whose domain is the set $\hat{C}(\mathbb{R})$ of continuous functions vanishing at infinity. Moreover, the process has almost surely continuous sample paths. 
\end{assumption}

\begin{proposition}
    Under Assumption \ref{Assumption: ell Feller}, the process \eqref{Eqn: omega for deterministic motion} is a Feller process with almost surely continuous sample paths. 
\end{proposition}
\begin{proof}
    Firstly note that the state space is locally compact, and indeed compact if we restrict to $\ell\in[0,\ell_{\max}]$. As the components $x_0$ and $\theta$ are fixed, the process is simply a continuous function of $\ell_t$, hence is itself a Feller process with almost surely continuous sample paths \cite[Chapter 4.2]{ethier2009markov}.
\end{proof}

For a fixed $n$ branches we may again define the measure-valued process
\begin{align*}
    \mu_t = \sum_{i=1}^N \delta\{(x_0^i, \theta^i, \ell^i_t)\} \,,
\end{align*}
which has generator 
\begin{align*}
    Af(\mu) &=
    \lim_{t\rightarrow0} \frac{1}{t} \mathbb{E} \big[ \langle f,\mu_t \rangle - \langle f,\mu\rangle \big] \\
    &= \lim_{t\rightarrow0}\frac{1}{t} \mathbb{E} \bigg[ \sum_{i=1}^N f(x_0^i, \theta^i, \ell^i_t) - f(x_0^i, \theta^i, \ell^i_0)\bigg] \\
    &= \lim_{t\rightarrow0} \frac{1}{t} \mathbb{E} \bigg[ \sum_{i=1}^N f_{x_0^i, \theta^i}(\ell^i_t) - f_{x_0^i, \theta^i}(\ell^i_0) \bigg] \\
    &= \sum_{i=1}^N A_\ell f_{x_0^i, v^i}(\ell^i_0) \,,
\end{align*}
where $f_{x_0^i, v^i}(\ell) := f(x_0^i, \theta^i, \ell)$. Defining
\begin{align*}
    A_\ell f(x_0,\theta,\ell) := A_\ell f_{x_0, \theta}(\ell) \,,
\end{align*}
this gives
\begin{align}
    Af(\mu) = \langle A_\ell f, \mu \rangle \,.
\end{align}
Jumps may be added similarly to the construction in Section \ref{Section: New branches} and again removed when $\ell=0$. 

This simplification does not resolve the issue of combining retraction and branching discussed at the beginning of Section \ref{Section: branching process} and therefore it would again be necessary to include an activity $a^i$ for each branch (and update the start position $x_0$ accordingly in these jumps). 

This simplification removes the need to consider a path-valued process by using the flow $\phi$ to construct the path from three variables. This leads to greater freedom in choosing $\ell_t$, but a massive loss of freedom by enforcing deterministic growth.  

\section{Dynamic instability} \label{Appendix: Dynamic instability}

In this section we discuss the dynamic instability version of the length process introduced at the end of Section \ref{Section: Alternative length processes}. Recall that 
\begin{align*}
    \ell_t = \ell_0 + \int_0^t v_s \, ds \,,
\end{align*}
where $v_s \in \mathbb{R}$ is the growth/retraction rate at time $s$ and is defined by a jump process on $\mathbb{R}$. This process has a generator of the form
\begin{align*}
    \mathcal{A}_v f(v) = \lambda(v) \int f(v') - f(v) \, \rho(v,dv') \,, 
\end{align*}
where $\lambda$ is the total rate of jumps out of state $v$ and $\rho(v,dv')$ describes the distribution of the new state $v'$. 

Since $v$ is not a variable captured in the state space $\W$, we consider functions $F(\omega,v)$. The evolution of $\omega$ will depend on $v$ (specifically its sign). Jumps in $v$ will be included in the generator through a term of the form
\begin{align*}
    \lambda(v) \int F(\omega,v') - F(\omega,v) \, \rho(v,dv') \,,
\end{align*}
but these jumps will have no effect on the evolution of $\omega$ (at first order). 

We consider three cases below. For simplicity we temporarily write $F$ as a function of $\omega$ only. 

\textbf{Case 1:} If $v_0 = 0$ then $\omega$ is unchanged so
\begin{align*}
    \lim_{t\rightarrow0} \frac{\mathbb{E} \big[ F(\omega_t) \big] - F(\omega_0)}{t} = 0 \,.
\end{align*}

\textbf{Case 2:} If $v_0 < 0$ then $\omega_t = \omega^{\leq \ell_0 + t v_0}$, so
\begin{align*}
    \lim_{t\rightarrow0} \frac{\mathbb{E} \big[ F(\omega_t) \big] - F(\omega_0)}{t} 
    &= \lim_{t\rightarrow0} \frac{F(\omega^{\leq \ell_0 + t v_0}) - F(\omega_0)}{t} \\
    &= \lim_{T\rightarrow0} v_0\frac{F(\omega^{\leq \ell_0 + T}) - F(\omega_0)}{T} 
\end{align*}
using $T = tv_0$. Assume that $F$ is differentiable as $\omega$ is shortened, i.e. that the following limit exists for all $\omega$, 
\begin{align*}
    \partial_\ell F(\omega) := \lim_{r \uparrow \ell_\omega} \frac{F(\omega^{\leq r}) - F(\omega)}{r} \,,
\end{align*}
then
\begin{align*}
    \lim_{t\rightarrow0} \frac{\mathbb{E} \big[ F(\omega_t) \big] - F(\omega_0)}{t} 
    &= \lim_{t\rightarrow0} \frac{F(\omega^{\leq \ell_0 + t v_0}) - F(\omega_0)}{t} \\[0.5em]
    &= v_0 \, \partial_\ell F(\omega_0) \,.
\end{align*}

\textbf{Case 3:} If $v_0>0$ then $\omega_t$ is an extension of $\omega_0$ using the law of $\xi$. Therefore we can apply the generator $L$ of the A-path (see Appendix \ref{Appendix: Absorbed Brownian motion}) to give
\begin{align*}
    \lim_{t\rightarrow0} \frac{\mathbb{E} \big[ F(\omega_t) \big] - F(\omega_0)}{t} 
    &= v_0 \, LF(\omega) \,.
\end{align*}
Hence
\begin{align*}
    \mathcal{A}F(\omega,v) &= v_0 \Big( \mathds{1}\{v>0\} LF(\omega,v) + \mathds{1}\{v<0\}\partial_\ell F(\omega,v) \Big) \\
    &+ \lambda(v) \int F(\omega,v') - F(\omega,v) \, \rho(v,dv') \,.
\end{align*}
We do not repeat the calculations in \cite{dhersin2000stochastic} to show that the martingale problem for this generator has a unique solution. Moreover the measure-valued process constructed in Section \ref{Section: branching process} must be adapted to track $v^i$ for each dendrite. To satisfy the requirements imposed by working in $\mathcal{N}_S^\#$ (and due to physical limitations), $v$ should be restricted to some bounded subset of $\mathbb{R}$. However, the setup essentially follows that described in Section \ref{Section: branching process}. 

\end{document}